\documentclass[11pt,reqno]{article}
\usepackage{graphicx}
\usepackage{amssymb,amsmath}
\usepackage{amsthm}
\usepackage{color,graphicx}
\usepackage{hyperref}
\usepackage{color}
\usepackage{epstopdf}
\usepackage{caption}
\usepackage{subcaption}
\usepackage[dvips]{epsfig}

\newcommand{\be}{\begin{equation}}

\newcommand{\ee}{\end{equation}}

\numberwithin{equation}{section}
\numberwithin{figure}{section}

\newtheorem{theorem}{Theorem}[section]
\newtheorem{proposition}[theorem]{Proposition}
\newtheorem{remark}[theorem]{Remark}
\newtheorem{lemma}[theorem]{Lemma}
\newtheorem{corollary}[theorem]{Corollary}

\newtheorem{definition}[theorem]{Definition}

\begin{document}

\title{\bf Periodic waves in the fractional modified \\ Korteweg--de Vries equation}

\author{\em F\'{a}bio Natali$^1$, Uyen Le$^2$, and Dmitry E. Pelinovsky$^2$ \\
	{\small $^1$ Departamento de Matem\'{a}tica - Universidade Estadual de Maring\'{a},} \\
	{\small Avenida Colombo 5790, CEP 87020-900, Maring\'{a}, PR, Brazil} \\
	{\small $^2$ Department of Mathematics and Statistics, McMaster University,}\\ {\small Hamilton, Ontario, Canada, L8S 4K1}}

\maketitle

\begin{abstract}
Periodic waves in the modified Korteweg-de Vries (mKdV) equation are revisited in the setting of the fractional Laplacian. Two families of solutions in the local case are given by the sign-definite dnoidal and sign-indefinite cnoidal solutions. Both solutions can be characterized in the general fractional case as global minimizers of the quadratic part of the energy functional subject to the fixed $L^4$ norm: the sign-definite (sign-indefinite) solutions are obtained in the subspace of even (odd) functions. Morse index is computed for both solutions and the spectral stability criterion is derived. 
We show numerically that the family of sign-definite solutions has a generic 
fold bifurcation for the fractional Laplacian of lower regularity 
and the family of sign-indefinite solutions has a generic symmetry-breaking 
bifurcation both in the fractional and local cases.
\end{abstract}


\section{Introduction}

The purpose of this work is to study existence, variational characterization, and bifurcations of periodic solutions of the following stationary equation:
\begin{equation}
\label{ode-stat}
D^{\alpha} \psi + c \psi + b = 2 \psi^3,
\end{equation}
where $\psi(x) : \mathbb{T} \mapsto \mathbb{R}$ is the wave profile on a circle $\mathbb{T} : = [-\pi,\pi]$, $(c,b)$ are real parameters, and $D^{\alpha}$ is the fractional Laplacian on $\mathbb{T}$ defined via Fourier series by 
\begin{equation*}
\psi(x) = \sum_{n \in \mathbb{Z}} \psi_n e^{inx}, \quad 
(D^{\alpha} \psi)(x) = \sum_{n \in \mathbb{Z}} |n|^{\alpha} \psi_n e^{inx}.
\end{equation*}
Thanks to the scaling transformation for the cubic nonlinearity, the fundamental period of the wave has been scaled to $2\pi$. In addition, we are interested in the simplest periodic waves with the single-lobe profile according to the following definition.

\begin{definition}\label{defilobe}
	We say that the periodic wave satisfying the stationary equation (\ref{ode-stat})
	has a single-lobe profile $\psi$ if it admits only one isolated maximum and minimum on $\mathbb{T}$.
\end{definition}

The stationary equation (\ref{ode-stat}) arises in the context of  
the fractional mKdV (modified Korteweg--de Vries) equation written in the form:
\begin{equation}\label{rDE}
u_t + 6u^2 u_x-(D^{\alpha}u)_x=0,
\end{equation}
with $u(t,x) : \mathbb{R} \times \mathbb{T} \mapsto \mathbb{R}$.
Traveling waves of the form $u(t,x) = \psi(x-c t)$ 
satisfy the stationary equation (\ref{ode-stat}), where $c$ is 
the wave speed and $b$ is an integration constant. The fractional mKdV equation (\ref{rDE}) admits formally the following conserved quantities:
\begin{equation}\label{Eu}
E(u) = \frac{1}{2} \int_{-\pi}^{\pi} \left[ (D^{\frac{\alpha}{2}}u)^2 - u^4 \right] dx,
\end{equation}
\begin{equation}\label{Fu}
F(u)=\frac{1}{2}\int_{-\pi}^{\pi}u^2dx,
\end{equation}
and
\begin{equation}\label{Mu}
M(u)=\int_{-\pi}^{\pi}u\,dx,
\end{equation}
which have meaning of energy, momentum, and mass, respectively.
The stationary equation (\ref{ode-stat}) is
the Euler--Lagrange equation for the action functional,
\begin{equation}\label{lyafun}
G(u)=E(u)+c F(u)+  b M(u),
\end{equation}
so that $G'(\psi)=0$. The Hessian operator from the action functional (\ref{lyafun}) yields the linearized operator around the wave $\psi$ in the form:
\begin{equation}\label{operator}
\mathcal{L} := G''(\psi) = D^{\alpha} + c - 6\psi^2.
\end{equation}
In addition to questions of the existence, we study the spectral stability of the periodic wave with the spatial profile $\psi$ in the time evolution of the fractional mKdV equation (\ref{rDE}) according to the following definition.

\begin{definition}
	\label{defspe} The periodic wave is said to be spectrally stable 
	with respect to perturbations of the same period if	$\sigma(\partial_x \mathcal{L}) \subset i\mathbb R$ in $L^2(\mathbb{T})$.
	Otherwise, it is said to be {\it spectrally unstable}.
\end{definition}

The fractional mKdV equation (\ref{rDE}) appears to be a generic model 
for one-dimensional long waves with weak dispersion \cite{SautPilod,SautSoliton}. For $\alpha = 2$, 
the mKdV equation is completely integrable \cite{AKNS}. 
For $\alpha = 1$, it is referred to as the modified Benjamin--Ono equation \cite{Kalish}.

Global well-posedness results for the initial data in $H^s(\mathbb{R})$ with $s>\frac{1}{4}$ and in $H^{s}(\mathbb{T})$ with $s\geq\frac{1}{2}$ 
were obtained for $\alpha = 2$  in \cite{CKSTT}. Local well-posedness results for initial data in $H^{s}(\mathbb{R})$ with $s\geq\frac{1}{2}$ were obtained for $\alpha = 1$ in \cite{KT}. Energy and momentum are conserved 
in the time evolution of such solutions. Local solutions 
with sufficiently large initial data in $H^{\frac{1}{2}}(\mathbb{R})$ blow up in a finite time \cite{KMR,Martel-Pilod}.

Periodic waves of the stationary equation (\ref{ode-stat}) were only studied in the local case of $\alpha = 2$. There exist two families 
of periodic solutions for $b = 0$ with the single-lobe profile:  sign-definite solutions are expressed by the dnoidal elliptic functions and 
sign-indefinite solutions are expressed by the cnoidal elliptic functions. 
For $b \neq 0$, all periodic waves can be expressed as a rational 
function of Jacobian elliptic functions \cite{CPgardner}. Spectral and orbital stability of periodic waves in the local case of $\alpha = 2$ was also considered in the literature.

Employing the arguments in \cite{bona} and \cite{W1,W2}, orbital stability of sign-definite dnoidal waves with $b = 0$ was proven in \cite{angulo1}. Spectral stability of sign-indefinite cnoidal waves with $b = 0$ was studied 
in \cite{DK} by using the count of negative eigenvalues of the operator $\mathcal{L}$ restricted to the orthogonal complement of ${\rm span}(1,\psi)$ (also also \cite{haragus,Pel}). It was discovered in 
\cite{DK} that the cnoidal waves were spectrally stable for smaller speeds $c$ and spectrally unstable for larger speeds $c$. Spectral and orbital stability and instability of the cnoidal waves was proven in \cite{AN1} by adopting the arguments of \cite{lin} in the periodic context
and employing the approach in \cite{henry} based on the existence of a sufficiently smooth data-to-solution-map. Orbital stability of a particular family of positive periodic waves of the dnoidal type with $b\neq0$ was proven in \cite{AP2} by adopting the arguments of \cite{grillakis1}.

In the limit of $c \to \infty$, periodic waves with the single-lobe profile on the fixed circle $\mathbb{T}$ concentrate near centers of symmetry and approach the solitary waves. With the scaling transformation, 
\begin{equation}\label{soliton}
\psi(x) = c^{\frac{1}{2}} Q(c^{\frac{1}{\alpha}} x), \quad x \in \mathbb{R},
\end{equation}  
for $c > 0$, the solitary wave profile $Q(x) :\mathbb{R} \mapsto \mathbb{R}$ satisfies the $c$-independent problem 
\begin{equation}\label{soliton-ode}
D^{\alpha} Q + Q = 2 Q^3.
\end{equation} 
Existence and uniqueness (modulus translations) of solitary waves with the spatial profile $Q$ satisfying equation (\ref{soliton-ode}) was shown in \cite{FL2013} based on their variational characterization 
as minimizers of the Gagliardo--Nirenberg--type inequality considered in \cite{W3,W4}:
\begin{equation}
\label{soliton-variational}
\| u \|^4_{L^4(\mathbb{R})} \leq C_{\alpha} \| D^{\frac{\alpha}{2}} u \|_{L^2(\mathbb{R})}^{\frac{2}{\alpha}} \| u \|_{L^2(\mathbb{R})}^{4-\frac{2}{\alpha}}, \quad u \in H^{\frac{\alpha}{2}}(\mathbb{R}),
\end{equation}
where $C_{\alpha} > 0$ is the $u$-independent constant. 
Solitary waves satisfying (\ref{soliton-ode}) give the best value for 
$C_{\alpha}$ which saturates the Gagliardo--Nirenberg--type inequality (\ref{soliton-variational}). 

For $\alpha=2$, the classical arguments in \cite{Albert,bona,grillakis1} can be used to prove the orbital stability of solitary waves in the energy space $H^1(\mathbb{R})$. 
Spectral and orbital stability of solitary waves in the general case $\alpha\in\left(\frac{1}{2},2\right)$ was considered in \cite{A}
based on the dependence of the momentum $F(\psi) = c^{1-\frac{1}{\alpha}} F(Q)$
on the wave speed $c > 0$. It was shown that the solitary waves were unstable if $\alpha\in \left(\frac{1}{2},1\right)$ and stable if $\alpha\in (1,2)$ in agreement with increasing and decreasing 
dependence of $F(\psi)$ versus $c$ respectively. The critical case $\alpha=1$ was inconclusive because $\frac{d}{dc}F(\psi)=0$ in this case.
For the critical case, existence of blow up solutions with minimal mass was proven in \cite{Martel-Pilod} by combining sharp energy estimates and a refined localized argument from \cite{KMR}. Hence, the solitary waves 
are unstable for $\alpha = 1$.

The previous review of literature shows that the stationary equation (\ref{ode-stat}) in the periodic domain $\mathbb{T}$ has been unexplored for $\alpha < 2$. Compared to the cubic case, a similar problem with the quadratic nonlinearity has been recently studied in many aspects, e.g. 
existence and stability of traveling periodic waves were analyzed by using perturbative \cite{J}, variational \cite{Bronski,BC,hur}, and fixed-point \cite{C,lepeli} methods. 

The standard approach to characterize the spectral and orbital stability of periodic waves with respect to perturbations of the same period is based on the minimization of energy $E(u)$ subject to the fixed momentum $F(u)$ and mass $M(u)$. However, parameters $(c,b)$ appear to be Lagrange multipliers of the action functional (\ref{lyafun}) and the smoothness of the momentum and mass with respect to Lagrange multipliers cannot be taken for granted (compared to what was done in \cite{hur}). 

In order to resolve problems of the variational characterization of the traveling periodic waves in the fractional KdV equation (with quadratic nonlinearity), two new approaches were recently developed. 

In \cite{stefanov}, the periodic waves with the single-lobe profile were constructed by minimizing the energy $E(u)$ subject to only one constraint of the fixed momentum $F(u)$. It was shown that such minimizers were 
degenerate only up to the translation symmetry and were spectrally stable.

From a different point of view, we characterized the traveling periodic waves in our previous work \cite{NPU} by minimizing the quadratic part of the action functional $G(u)$ subject to the fixed cubic part of the energy $E(u)$ and the zero mean constraint $M(u) = 0$. This approach combined with the Galilean transformation allowed us to represent all possible periodic waves of the single-lobe profile $\psi$ and to derive a simple stability criterion from the derivative of the momentum $F(\psi)$ with respect to the wave speed $c$. 

When ideas of \cite{NPU} are extended to the cubic nonlinearity 
in the framework of the stationary equation (\ref{ode-stat}), we face the difficulty 
that the Galilean transformation generates a quadratic nonlinear term 
and connects solutions of the fractional mKdV equation to solutions of the fractional Gardner equation. As a result, we are not able yet to characterize all possible periodic waves of the single-lobe profile 
in the stationary equation (\ref{ode-stat}). Instead, we shall study the two particular families of solutions which correspond to $b = 0$ and generalize the sign-definite dnoidal and sign-indefinite cnoidal elliptic solutions of the local case $\alpha = 2$. Both families are obtained as minimizers of the quadratic part of the action functional $G(u)$ subject to the fixed quartic part of the energy, but one family is obtained in the subspace of even functions and the other family is obtained in the subspace of odd functions. 
For the purpose of simplicity, we refer to the first family as {\em the even periodic waves} and to the second family as {\em the odd periodic waves}. 

The following two theorems present the main results of this paper. In what follows, we write $H^s_{\rm per}$ instead of $H^s_{\rm per}(\mathbb{T})$.
The subspace of odd (even) functions in $L^2$ is denoted by $L^2_{\rm odd}$ ($L^2_{\rm even}$). Similarly, the subspace of odd (even) periodic functions in $H^s_{\rm per}$ is denoted by $H^s_{\rm per, odd}$ ($H^s_{\rm per, even}$).

\begin{theorem}[Odd periodic wave]
	\label{theorem-main-1}
	Fix $\alpha \in \left(\frac{1}{2},2 \right]$. For every $c_0 \in (-1,\infty)$, there exists a solution to the stationary 
 equation (\ref{ode-stat}) with $b = 0$ and the odd, single-lobe profile $\psi_0$, which is obtained from a constrained minimizer of the following variational problem:
	\begin{equation}
	\label{minimizer-1}
	\inf_{u\in H_{\rm per, odd}^{\frac{\alpha}{2}}} \left\{ \int_{-\pi}^{\pi} \left[ (D^{\frac{\alpha}{2}}u)^2 + c_0 u^2 \right] dx : \quad
	\int_{-\pi}^{\pi} u^4 dx = 1 \right\}.
	\end{equation}
There exists a $C^1$ mapping $c \mapsto \psi(\cdot,c) \in H^{\alpha}_{\rm per, odd}$ in a local neighborhood of $c_0$ such that $\psi(\cdot,c_0) = \psi_0$. The spectrum of $\mathcal{L}$ in $L^2(\mathbb{T})$
includes two negative eigenvalues and if $1 \in {\rm Range}(\mathcal{L})$, 
a simple zero eigenvalue. Assuming $1 \in {\rm Range}(\mathcal{L})$ and setting $\sigma_0 := \langle \mathcal{L}^{-1} 1, 1 \rangle$, the periodic wave with the profile $\psi_0$ is spectrally stable if 
\begin{equation}
\sigma_0 \leq 0, \quad \frac{d}{dc} \| \psi \|^2_{L^2} \geq 0
\end{equation}
and is spectrally unstable with exactly one real, positive eigenvalue of $\partial_x \mathcal{L}$ in $L^2(\mathbb{T})$ if
\begin{equation}
{\rm either} \;\; \sigma_0 \frac{d}{dc} \| \psi \|_{L^2}^2 > 0 
\quad \mbox{\rm or} \;\; \sigma_0 = 0, \;\; \frac{d}{dc} \| \psi \|_{L^2}^2 < 0, \quad \mbox{\rm or} \;\; \sigma_0 > 0, \;\; \frac{d}{dc} \| \psi \|_{L^2}^2 = 0.
\end{equation}
If $1 \notin {\rm Range}(\mathcal{L})$, then
the periodic wave is spectrally unstable with exactly one real positive eigenvalue of $\partial_x \mathbb{L}$ in $L^2(\mathbb{T})$ if
\begin{equation}
\frac{d}{dc} \| \psi \|^2_{L^2} \geq 0.
\end{equation}
\end{theorem}

\begin{remark}
	If $\sigma_0 = 0$, the odd periodic wave of Theorem \ref{theorem-main-1} undertakes the stability bifurcation, which also results in the bifurcation of new solutions in the stationary equation (\ref{ode-stat}) with $b \neq 0$. The stability bifurcation was first discovered in \cite{DK} for $\alpha = 2$. We show numerically that this bifurcation is generic for every $\alpha \in \left(\frac{1}{2},2\right)$.
\end{remark}

\begin{remark}
	Based on numerical studies, we conjecture that the case $1 \notin {\rm Range}(\mathcal{L})$ is impossible for the odd periodic wave in Theorem \ref{theorem-main-1} 
	for every $\alpha \in \left(\frac{1}{2},2\right]$ and every $c \in (-1,\infty)$. Nevertheless, the case $1 \notin {\rm Range}(\mathcal{L})$ is observed 
	for the new solutions bifurcating from the odd periodic wave in Theorem \ref{theorem-main-1}.
\end{remark}

\begin{theorem}[Even periodic wave]
	\label{theorem-main-2}
	Fix $\alpha \in \left(\frac{1}{2},2 \right]$. For every $c_0 \in \left(\frac{1}{2},\infty \right)$, there exists a solution to the stationary 
	equation (\ref{ode-stat}) with $b = 0$ and the even, single-lobe profile $\psi_0$, which is obtained from a constrained minimizer of the following variational problem:
	\begin{equation}
	\label{minimizer-2}
	\inf_{u\in H_{\rm per, even}^{\frac{\alpha}{2}}} \left\{ \int_{-\pi}^{\pi} \left[ (D^{\frac{\alpha}{2}}u)^2 + c_0 u^2 \right] dx : \quad
	\int_{-\pi}^{\pi} u^4 dx = 1 \right\}.
	\end{equation}
The spectrum of $\mathcal{L}$ in $L^2(\mathbb{T})$
	includes one simple negative eigenvalue and if $1 \in {\rm Range}(\mathcal{L})$, a simple zero eigenvalue. 
With the transformation, 
\begin{equation}
\psi_0(x) = a_0 + \phi_0(x), \quad 
a_0 := \frac{1}{2\pi} \int_{-\pi}^{\pi} \psi_0(x) dx, \quad \omega_0 := c_0 - 6a_0^2,
\end{equation}
assuming $\omega_0 \in (-1,\infty)$, 
there exists a $C^1$ mapping $(\omega,a) \mapsto \phi(\cdot,\omega,a) \in H^{\alpha}_{\rm per, even}$ in a local neighborhood of $(\omega_0,a_0)$ such that $\phi(\cdot,\omega_0,a_0) = \phi_0$ and the mean value of $\phi$ is zero. 	The periodic wave $\psi_0$ is spectrally stable if 
	\begin{equation}
\frac{\partial}{\partial \omega} \| \phi \|^2_{L^2} \geq 0
	\end{equation}
and is spectrally unstable with exactly one real, positive eigenvalue of $\partial_x \mathcal{L}$ in $L^2(\mathbb{T})$ if
	\begin{equation}
\frac{\partial}{\partial \omega} \| \phi \|^2_{L^2} < 0.
	\end{equation}
\end{theorem}

\begin{remark}
We derive the criterion for $1 \notin {\rm Range}(\mathcal{L})$, in which case $\mathcal{L}$ has the double zero eigenvalue and the even periodic wave of Theorem \ref{theorem-main-2} undertakes the fold bifurcation. Two solutions of the stationary equation (\ref{ode-stat}) with $b = 0$ coexist for the same value of $c$ near the fold bifurcation.  We show numerically that the fold bifurcation is generic for every $\alpha \in \left(\frac{1}{2},\alpha_0 \right)$, where 
$$
\alpha_0 := \frac{\log 8 - \log 5}{\log 2} \approx 0.6781. 
$$
If $\alpha \in \left(\frac{1}{2},\alpha_0 \right)$, two solutions 
with the even, single-lobe profile exist for the same value of $c \in (c_0,\frac{1}{2})$ with $c_0 \in \left(0,\frac{1}{2}\right)$ beyond 
the admissible range of values of $c$ in Theorem \ref{theorem-main-2}.
\end{remark}

\begin{remark}
	Based on numerical evidences, we conjecture that $\omega \in (-1,\infty)$ is always satisfied for the even periodic wave in Theorem \ref{theorem-main-2}. 
\end{remark}

The paper is organized as follows. The odd periodic wave of Theorem \ref{theorem-main-1} is characterized in Section 2. Examples of the odd periodic waves are given in Section 3 with Stokes expansions, exact elliptic solutions in the local case $\alpha = 2$, and numerical approximations for $\alpha = 1$. The even periodic waves of Theorem \ref{theorem-main-2} is characterized in Section 4. Similar examples of the even periodic wave are given in Section 5. Section 6 discusses an open problem to characterize all possible periodic waves with the single-lobe profile among solutions to the stationary equation (\ref{ode-stat}) with $b \neq 0$.

\section{Odd periodic waves}

Here we consider {\em the odd periodic waves} and provide a proof of Theorem \ref{theorem-main-1}.

If $\psi \in H^{\alpha}_{\rm per}$ is a solution to the stationary equation
(\ref{ode-stat}) with $b = \frac{1}{\pi}\int_{-\pi}^{\pi}\psi^3dx$, then
$\psi$ satisfies the zero-mean constraint and the boundary-value problem:
\begin{equation}\label{ode-bvp}
D^{\alpha}\psi + c\psi = 2\Pi_0 \psi^3, 
\end{equation}
where $\Pi_0 f := f - \frac{1}{2\pi} \int_{-\pi}^{\pi} f(x) dx$ is the projection
operator reducing the mean value of $2\pi$-periodic functions to zero.

Since we use variational methods, we consider weak solutions of the boundary-value problem (\ref{ode-bvp}) in $H^{\frac{\alpha}{2}}_{\rm per}$.
By the same bootstrapping argument as in Proposition 1 in \cite{stefanov} 
or Proposition 2.4 in \cite{NPU},
if $\psi \in H^{\frac{\alpha}{2}}_{\rm per}$ is a weak solution of the boundary-value
problem (\ref{ode-bvp}), then $\psi \in H^{\infty}_{\rm per}$ and, in particular,
it is a strong solution to the boundary-value problem (\ref{ode-bvp}) in $H^{\alpha}_{\rm per}$.

The following theorem and its corollary give the construction and properties of the periodic waves in a subspace of odd functions. 
These odd periodic waves satisfy the boundary-value problem (\ref{ode-bvp}).

\begin{theorem}
\label{theorem-existence}
Fix $\alpha > \frac{1}{2}$. For every $c > -1$, there exists
the ground state (minimizer) $\chi \in H^{\frac{\alpha}{2}}_{\rm per, odd}$
of the following constrained minimization problem:
\begin{equation}
\label{minBfunc}
q_{c, {\rm odd}} := \inf_{u\in H^{\frac{\alpha}{2}}_{\rm per, odd}} \left\{ \mathcal{B}_c(u) : \quad \int_{-\pi}^{\pi} u^4 dx = 1 \right\},
\end{equation}
where
\begin{equation}
\label{def-B}
\mathcal{B}_c(u) := \frac{1}{2}\int_{-\pi}^{\pi} \left[ (D^{\frac{\alpha}{2}}u)^2 + c u^2 \right] dx.
\end{equation}
If $\alpha \leq 2$, the ground state has the single-lobe profile, which is even with respect to
the points at $x = \pm \pi/2$.
\end{theorem}

\begin{proof}
It follows that $\mathcal{B}_c$ is a smooth functional bounded on $H^{\frac{\alpha}{2}}_{\rm per, odd}$.
Moreover, $\mathcal{B}_c$ is proportional to the quadratic form of the operator $c + D^{\alpha}$
with the spectrum in $L^2_{\rm odd}$ given by $\{ c + m^{\alpha}, \;\; m \in \mathbb{N} \}$.
By Poincar\'e's inequality, we have 
\begin{equation}
\label{positivity}
\mathcal{B}_c(u) \geq \frac{1}{2} (c + 1) \| u \|_{L^2}^2, \quad u \in H^{\frac{\alpha}{2}}_{\rm per, odd},
\end{equation}
and by G{\aa}rding's inequality, for every $c > -1$ there exists $C > 0$ such that
$$
\mathcal{B}_c(u) \geq C \| u \|_{H^{\frac{\alpha}{2}}_{\rm per}}^2, \quad u \in H^{\frac{\alpha}{2}}_{\rm per, odd}.
$$
Hence $\mathcal{B}_c$ is equivalent to the squared norm in $H^{\frac{\alpha}{2}}_{\rm per, odd}$ so that $q_{c, {\rm odd}} \geq 0$.

Let $\{u_n\}_{n \in \mathbb{N}}$ be a minimizing sequence for the constrained minimization problem \eqref{minBfunc}, that is, a sequence satisfying
$$
\mathcal{B}_c(u_n)\rightarrow  q_{c,{\rm odd}} \quad  \mbox{as} \quad  n\rightarrow \infty.
$$
Since $\{u_n\}_{n \in \mathbb{N}}$ is bounded in $H^{\frac{\alpha}{2}}_{\rm per, odd}$, there exists
$\chi \in H^{\frac{\alpha}{2}}_{\rm per, odd}$  such that, up to a subsequence,
$$
u_n\rightharpoonup \chi \quad \mbox{in} \ H^{\frac{\alpha}{2}}_{\rm per, odd},  \quad  \mbox{as} \quad n\rightarrow \infty.
$$
For every $\alpha>\frac{1}{2}$, the energy space $H^{\frac{\alpha}{2}}_{\rm per, odd}$ is compactly embedded in $L^4_{\rm odd}$ (see, e.g., Theorem 4.2 in \cite{Ambrosio}).
Thus, there is a positive constant $C$ such that for every $\chi \in H^{\frac{\alpha}{2}}_{\rm per, odd}$, 
\begin{equation}
\label{Sob-embedding}
\| \chi \|_{L^4} \leq C \| \chi \|_{H^{\frac{\alpha}{2}}}, 
\end{equation}
and
$$
u_n\rightarrow \chi \quad \mbox{in} \ L^4_{\rm odd},  \quad  \mbox{as} \quad n\rightarrow \infty.
$$
Using the estimate
\begin{eqnarray*}
\left|\int_{-\pi}^{\pi}(u_n^4-\chi^4)dx\right|
&\leq&\int_{-\pi}^{\pi}|u_n^4-\chi^4|dx \\
&\leq& \left( \|\chi\|^3_{L^4} + \|\chi\|_{L^4}^2 \|u_n\|_{L^4}
 + \|\chi\|_{L^4} \|u_n\|_{L^4}^2 + \|u_n\|_{L^4}^3 \right) \|u_n-\chi\|_{L^4},
\end{eqnarray*}
it follows that $\chi$ satisfies the constraint: $\int_{-\pi}^{\pi} \chi^4 dx = 1$. 
In view of (\ref{Sob-embedding}), this implies that  $q_{c, {\rm odd}} > 0$.

Thanks to the weak lower semi-continuity of $\mathcal{B}_c$, we have
\begin{equation*}
\mathcal{B}_c(\chi) \leq \liminf_{n\rightarrow \infty} \mathcal{B}(u_n) = q_{c, {\rm odd}}.
\end{equation*}
Since $\chi$ satisfies the constraint, we also have 
$\mathcal{B}_c(\chi) \geq q_{c, {\rm odd}}$, hence 
$\mathcal{B}_c(\chi) = q_{c, {\rm odd}}$ and $\chi \in H^{\frac{\alpha}{2}}_{\rm per, odd}$ is the ground state (global minimizer) of the variational problem (\ref{minBfunc}).

If $\alpha \in (0,2]$, the symmetric rearrangements of $u$ do not increase
$\mathcal{B}_c(u)$ while leaving the constraint on the $L^4$-norm invariant due to the fractional Polya--Szeg\"{o} inequality, see Lemma A.1 in \cite{CJ2019}.
As a result, the minimizer $\chi \in H^{\frac{\alpha}{2}}_{\rm per, odd}$ of $\mathcal{B}_c(u)$ must decrease symmetrically away from
the maximum point. Since $\chi(x) = 0$ at $x = 0$ and $x = \pm \pi$, the symmetry points of $\chi$ are located at $x = \pm \pi/2$,
so that the single-lobe profile is even with respect to the points at $x = \pm \pi/2$.
\end{proof}

\begin{corollary}
\label{corollary-existence}
Let $\chi$ be the ground state of Theorem \ref{theorem-existence}.
There exists $C > 0$ such that $\psi(x) = C \chi(x)$ satisfies the boundary-value problem (\ref{ode-bvp}).
\end{corollary}

\begin{proof}
By Lagrange's Multiplier Theorem, the constrained minimizer $\chi \in H^{\frac{\alpha}{2}}_{\rm per, odd}$ satisfies
the stationary equation
\begin{equation}
\label{lagrange}
D^{\alpha}\chi + c \chi = \mu \chi^3,
\end{equation}
for some constants $\mu$. From the constraint $\int_{-\pi}^{\pi} \chi^4 dx = 1$, we have
$\mu = 2 \mathcal{B}_c(\chi)$. Since $\mathcal{B}_c(\chi) > 0$, the scaling transformation $\psi = C \chi$
with $C := \sqrt{\mathcal{B}_c(\chi)}$ maps the stationary equation (\ref{lagrange})
to the form (\ref{ode-stat}) with $b = 0$. By the bootstrapping argument, $\psi \in H^{\alpha}_{\rm per, odd}$,
hence it is a strong solution to the boundary-value problem (\ref{ode-bvp}).
\end{proof}

\begin{remark}
Since $\psi$ and $\psi^3$ are odd for the solution constructed in Corollary \ref{corollary-existence},
$\psi$ is a solution to the stationary equation (\ref{ode-stat}) with $b = 0$ for every $c \in (-1,\infty)$.
\label{remark-b}
\end{remark}

Let $\psi \in H^{\alpha}_{\rm per, odd}$ be a solution to the boundary-value problem (\ref{ode-bvp})
for some $c \in (-1,\infty)$ obtained by Theorem \ref{theorem-existence} and Corollary \ref{corollary-existence}.
Let $\mathcal{L}$ be the linearized operator around the wave $\psi$ given by (\ref{operator}) and
\begin{equation}\label{operator-c}
\mathcal{L} : \quad H^{\alpha}_{\rm per} \subset L^2(\mathbb{T}) \to L^2(\mathbb{T}).
\end{equation}
In what follows, we determine the multiplicity of the zero eigenvalue of $\mathcal{L}$ denoted as $z(\mathcal{L})$
and the number of negative eigenvalues of $\mathcal{L}$ with the account of their multiplicities denoted as $n(\mathcal{L})$.
It follows from the stationary equation (\ref{ode-stat}) with $b = 0$ that
\begin{eqnarray}
\label{range-1}
\mathcal{L} 1 = c - 6 \psi^2
\end{eqnarray}
and
\begin{eqnarray}
\label{range-2}
\mathcal{L} \psi = - 4 \psi^3.
\end{eqnarray}
By the translational symmetry, we always have $\mathcal{L} \partial_x \psi = 0$.

Since $\psi$ is the single-lobe profile of the periodic wave in the sense of Definition \ref{defilobe}
and since $\psi$ is even with respect to the points at $x = \pm \pi/2$, then we can place
the unique maximum of $\psi$ at $x = \pi/2$ and adopt several results of \cite{hur}
with the same proof after translation $x \mapsto x - \pi/2$.

\begin{proposition}\cite{hur}
\label{prop-nodal}
Let $\alpha \in (\frac{1}{2},2]$ and $\psi \in H^{\alpha}_{\rm per, odd}$ be a solution obtained in Theorem \ref{theorem-existence} and Corollary \ref{corollary-existence}. An eigenfunction of $\mathcal{L}$ defined by
(\ref{operator}) and \emph{}(\ref{operator-c})
corresponding to the $n$-th eigenvalue of $\mathcal{L}$ for $n = 1,2,3$ changes its sign at most
$2(n-1)$ times over $\mathbb{T}$. Moreover, the eigenfunction of $\mathcal{L} |_{L^2_{\rm even}}$
for the $n$-th eigenvalue of $\mathcal{L} |_{L^2_{\rm even}}$ changes its sign at most $2(n-1)$ times over $\mathbb{T}$.
\end{proposition}

\begin{proposition}\cite{hur}
	Assume $\alpha \in (\frac{1}{2},2]$ and $\psi \in H^{\alpha}_{\rm per, odd}$ be a solution obtained in Theorem \ref{theorem-existence} and Corollary \ref{corollary-existence}. If $\{ 1, \psi, \psi^2\} \in {\rm Range}(\mathcal{L})$,
then ${\rm Ker}(\mathcal{L}) = {\rm span}(\partial_x \psi)$.
\label{prop-kernel}
\end{proposition}

\begin{proposition}\cite{hur}
	Assume $\alpha \in (\frac{1}{2},2]$ and $\psi \in H^{\alpha}_{\rm per, odd}$ be a solution obtained in Theorem \ref{theorem-existence} and Corollary \ref{corollary-existence}.
Then, $\partial_x \psi \in {\rm Ker}(\mathcal{L})$
corresponds to the lowest eigenvalue of $\mathcal{L}$ in the space of odd functions with respect to $x = \pi/2$.
\label{prop-odd}
\end{proposition}

\begin{remark}
	Compared to \cite{hur}, the potential of $\mathcal{L}$ is $\psi^2 \in H^{\alpha}_{\rm per}$ rather than $\psi \in H^{\alpha}_{\rm per}$, where we recall that $\alpha > \frac{1}{2}$. This implies that the spatial period of the potential of $\mathcal{L}$ is $\pi$ rather than $2\pi$. Nevertheless, we are interested in the properties of $\mathcal{L}$ on $L^2(\mathbb{T})$ rather than on $L^2(\frac{1}{2} \mathbb{T})$.
\end{remark}

By an elementary application of the implicit function theorem (similarly to Lemma 3.8 in \cite{NPU}), we can also obtain
the following result.

\begin{lemma} 
\label{prop-continuation}
Assume $\alpha \in (\frac{1}{2},2]$ and $\psi_0 \in H^{\alpha}_{\rm per, odd}$ be a 
solution obtained in Theorem \ref{theorem-existence} and Corollary \ref{corollary-existence} for $c = c_0$.
Assume ${\rm Ker}(\mathcal{L} |_{L_{\rm odd}^2})$ is trivial.
Then, there exists a $C^1$ mapping in an open subset of $c_0$ denoted by $\mathcal{I} \subset \mathbb{R}$:
\begin{equation}
\label{map-in-c}
\mathcal{I} \ni c \mapsto \psi(\cdot;c) \in H_{\rm per, odd}^{\alpha}
\end{equation}
such that $\psi(\cdot;c_0) = \psi_0$ and $\mathcal{L} \partial_c \psi(\cdot;c_0) = -\psi_0$.
\end{lemma}

\begin{proof}
	Let $\Upsilon:(-1,\infty)\times H_{\rm per,odd}^{\alpha} \to
	L^2_{\rm odd}(\mathbb{T})$ be defined by $\Upsilon(c,f) := D^{\alpha} f + c f - 2 f^3$. 
	By hypothesis of the lemma, we have $\Upsilon(c_0,\psi_0)=0$. Moreover, $\Upsilon$ is smooth and	its Fr\'echet derivative with respect to $f$ evaluated at $(c_0,\psi_0)$ is given by $\mathcal{L}$ computed at $\psi_0$.	
	Since ${\rm Ker}(\mathcal{L} |_{L_{\rm odd}^2})$ is empty by the assumption, we conclude that $\mathcal{L}$ is one-to-one. It is also onto since its spectrum consists of nonzero isolated eigenvalues with finite algebraic multiplicities because $H_{\rm per,odd}^{\alpha}$
	is compactly embedded in $L^2_{\rm odd}(\mathbb{T})$ if $\alpha > 1/2$ and because $\mathcal{L}$ is a self-adjoint operator. Hence, $\mathcal{L}$ is a bounded linear operator	with a bounded inverse. Thus, since $\Upsilon$ and its derivative with respect to $f$  are smooth maps on their domains,
	the result follows from the implicit function theorem.
\end{proof}

By using Propositions \ref{prop-nodal}, \ref{prop-kernel}, \ref{prop-odd}, and Lemma \ref{prop-continuation}, we
compute $n(\mathcal{L})$ and $z(\mathcal{L})$ in the following lemma.

\begin{lemma}
\label{lemma-index}
Let $\alpha \in (\frac{1}{2},2]$ and $\psi \in H^{\alpha}_{\rm per, odd}$ be a 
solution obtained in Theorem \ref{theorem-existence} and Corollary \ref{corollary-existence}. Then, $n(\mathcal{L}) = 2$ and
$$
z(\mathcal{L}) = \left\{ \begin{array}{l} 1, \quad \mbox{\rm if  } 1 \in {\rm Range}(\mathcal{L}), \\
2, \quad \mbox{\rm if  } 1 \notin {\rm Range}(\mathcal{L}). \end{array} \right.
$$
\end{lemma}

\begin{proof}
Since $\psi \in H^{\alpha}_{\rm per, odd}$ is a minimizer of the constrained variational problem
(\ref{minBfunc}) with only one constraint, we have $n (\mathcal{L} |_{L^2_{\rm odd}}) \leq 1$.
On the other hand, we have
$$
\langle \mathcal{L} \psi, \psi \rangle_{L^2} = - 4 \| \psi \|^4_{L^4} < 0,
$$
with odd $\psi$, hence $n(\mathcal{L} |_{L^2_{\rm odd}}) \geq 1$, so that $n(\mathcal{L} |_{L^2_{\rm odd}}) = 1$.

Since $\partial_x \psi \in {\rm Ker}(\mathcal{L})$ and $\partial_x \psi$ is even with two nodes on $\mathbb{T}$,
then by Proposition \ref{prop-nodal}, $0$ is not the first eigenvalue of $\mathcal{L} |_{L^2_{\rm even}}$, so that
$n(\mathcal{L} |_{L^2_{\rm even}}) \geq 1$. However, another negative eigenvalue of $\mathcal{L} |_{L^2_{\rm even}}$
is impossible since the eigenfunction for the second eigenvalue of $\mathcal{L} |_{L^2_{\rm even}}$ must have two nodes
by Proposition \ref{prop-nodal} and the nodes are located at the symmetry points $x = \pm \pi/2$, hence
this eigenfunction is not orthogonal to $\partial_x \psi \in {\rm Ker}(\mathcal{L})$. Therefore,
$0$ is the second eigenvalue of $\mathcal{L} |_{L^2_{\rm even}}$, which yields $n(\mathcal{L} |_{L^2_{\rm even}}) = 1$
and
$$
n(\mathcal{L}) = n(\mathcal{L} |_{L^2_{\rm odd}}) + n(\mathcal{L} |_{L^2_{\rm even}}) = 2.
$$

It remains to consider $z(\mathcal{L}) \geq 1$. By the symmetry of $\psi$ and $\psi^2$,
the operator $\mathcal{L}$ in (\ref{operator})
and (\ref{operator-c}) has a $\pi$-periodic potential, which is even with respect to both $x = 0$ and $x = \pi/2$.
The negative eigenvalue of $\mathcal{L}$ in $L^2_{\rm even}$, which is the lowest eigenvalue of $\mathcal{L}$ in $L^2(\mathbb{T})$, 
corresponds to the sign-definite $\pi$-periodic function, which is even with respect to both
$x = 0$ and $x = \pi/2$. The negative eigenvalue in $L^2_{\rm odd}$, which is the second eigenvalue of $\mathcal{L}$ in $L^2(\mathbb{T})$, corresponds to the eigenfunction
with two nodes on $\mathbb{T}$, which is even with respect to $x = \pi/2$.
The eigenfunction $\partial_x \psi$ for the zero eigenvalue in $L^2_{\rm even}$, which is the third
eigenvalue of $\mathcal{L}$ in $L^2(\mathbb{T})$, has two nodes and is odd with respect to $x = \pi/2$. By Proposition \ref{prop-odd},
the zero eigenvalue is the lowest eigenvalue for the eigenfunctions that are odd with respect to $x = \pi/2$.

Finally, we consider the eigenfunctions with four nodes on $\mathbb{T}$ since $0$ is the third eigenvalue of $\mathcal{L}$. 
These eigenfunctions have the same parity with respect to $x = 0$ and $x = \pi/2$. It follows from the previous
argument that the odd eigenfunction
of $\mathcal{L}$ in $L^2_{\rm odd}$ with four nodes corresponds to the positive eigenvalue of $\mathcal{L}$.
Therefore, $z(\mathcal{L} |_{L^2_{\rm odd}}) = 0$ and by Lemma \ref{prop-continuation},
the mapping $c \mapsto \psi(\cdot;c)$ is $C^1$ in $c$ with $\mathcal{L} \partial_c \psi = -\psi$,
so that $\psi \in {\rm Range}(\mathcal{L})$.

Assume that the even eigenfunction of $\mathcal{L}$
in $L^2_{\rm even}$ with four nodes (call it $f$) belongs to ${\rm Ker}(\mathcal{L})$, hence ${\rm Ker}(\mathcal{L}) = {\rm span}(\partial_x \psi, f)$.
Either $\langle 1, f \rangle_{L^2} = 0$ or $\langle 1, f \rangle_{L^2} \neq 0$.
If $\langle 1, f \rangle_{L^2} = 0$, then $1 \in {\rm Range}(\mathcal{L})$. It follows from (\ref{range-1}) that
if $1 \in {\rm Range}(\mathcal{L})$, then $\psi^2 \in {\rm Range}(\mathcal{L})$.
Therefore, $\{ 1, \psi, \psi^2\} \in {\rm Range}(\mathcal{L})$ and by Proposition \ref{prop-kernel},
${\rm Ker}(\mathcal{L}) = {\rm span}(\partial_x \psi)$, so that $f$ does not belong to ${\rm Ker}(\mathcal{L})$.
Hence, $z(\mathcal{L}) = 1$ if $1 \in {\rm Range}(\mathcal{L})$.
If $\langle 1, f \rangle_{L^2} \neq 0$, then $1 \notin {\rm Range}(\mathcal{L})$ because ${\rm Range}(\mathcal{L})$
is orthogonal to ${\rm Ker}(\mathcal{L})$. Hence, $z(\mathcal{L}) = 2$ if and only if $1 \notin {\rm Range}(\mathcal{L})$.
\end{proof}

Next, we introduce the subspace of $L^2$ with zero mean and denote it by $X_0$:
\begin{equation}
\label{zero}
X_0 := \Big\{f \in L^2(\mathbb{T}): \quad \int_{-\pi}^{\pi} f(x) dx = 0 \Big\}.
\end{equation}
Denote $\Pi_0 \mathcal{L} \Pi_0$ by $\mathcal{L} |_{X_0}$. By an explicit computation, it follows
that if $f \in H^{\alpha}_{\rm per} \cap X_0$, then
\begin{equation}
\label{operator-zero}
\mathcal{L}|_{X_0} f := \mathcal{L} f + \frac{3}{\pi} \langle f, \psi^2\rangle.
\end{equation}
The following result is similar to Lemma 3.5 in \cite{NPU}.

\begin{lemma}
\label{lem-mean-value}
Let $\alpha \in (\frac{1}{2},2]$ and $\psi \in H^{\alpha}_{\rm per, odd}$ be a 
solution obtained in Theorem \ref{theorem-existence} and Corollary \ref{corollary-existence}. If there exists a nonzero $f \in {\rm Ker}(\mathcal{L} |_{X_0})$
such that $\langle f, \partial_x \psi \rangle = 0$, then
\begin{equation}
\label{kernel-mean-value}
z(\mathcal{L}) = 1, \;\; \mbox{\rm and} \quad \langle f, \psi^2 \rangle \neq 0.
\end{equation}
\end{lemma}

\begin{proof}
Since $f \in {\rm Ker}(\mathcal{L} |_{X_0})$, then $\langle 1, f \rangle = 0$ and
$f$ satisfies
\begin{equation}
\label{kernel-X-0}
\mathcal{L} f = - \frac{3}{\pi} \langle f, \psi^2\rangle.
\end{equation}
Either $\langle f, \psi^2 \rangle = 0$ or $\langle f, \psi^2 \rangle \neq 0$.

If $\langle f, \psi^2 \rangle = 0$, then $f \in {\rm Ker}(\mathcal{L})$ so that $z(\mathcal{L}) = 2$
and $1 \notin {\rm Range}(\mathcal{L})$ by Lemma \ref{lemma-index}.  However,
$1 \perp {\rm span}(\partial_x \psi, f) = {\rm Ker}(\mathcal{L})$ implies $1 \in {\rm Range}(\mathcal{L})$, which is a contradiction.

If $\langle f, \psi^2 \rangle \neq 0$, then it follows from (\ref{kernel-X-0}) that $1 \in {\rm Range}(\mathcal{L})$
and hence $z(\mathcal{L}) = 1$ by Lemma \ref{lemma-index}. 
This yields (\ref{kernel-mean-value}).
\end{proof}

\begin{remark}
Assuming $1 \in {\rm Range}(\mathcal{L})$, let us define $\sigma_0 := \langle \mathcal{L}^{-1} 1, 1 \rangle$.
Then, $z(\mathcal{L}|_{X_0}) = 2$ if and only if $\sigma_0 = 0$. On the other hand,
$z(\mathcal{L}) = 2$ if and only if $\sigma_0$ is unbounded.
\end{remark}

Next, we consider if the ground state of the variational problem (\ref{minBfunc}) with a single constraint
is a local minimizer of the following variational problem with two constraints:
\begin{equation}
\label{infB}
r_c := \inf_{u\in H^{\frac{\alpha}{2}}_{\rm per}} \left\{ \mathcal{B}_c(u) : \quad \int_{-\pi}^{\pi} u^4 dx = 1, \quad \int_{-\pi}^{\pi} u dx = 0 \right\}.
\end{equation}
It is clear that $r_c \leq q_{c, {\rm odd}}$ and therefore, minimizers of (\ref{minBfunc}) could be saddle points of (\ref{infB}).
The following lemma provides the relevant criterion.

\begin{lemma}
\label{lemma-characterization}
Let $\alpha \in (\frac{1}{2},2]$ and $\psi \in H^{\alpha}_{\rm per, odd}$ be a 
solution obtained in Theorem \ref{theorem-existence} and Corollary \ref{corollary-existence}.
If $1 \in {\rm Range}(\mathcal{L})$, then
\begin{equation}
n(\mathcal{L}|_{\{1, \psi^3 \}^{\perp}}) = \left\{ \begin{array}{ll} 0, \quad & \sigma_0 \leq 0, \\
1, \quad & \sigma_0 > 0, \end{array}\right. \quad
z(\mathcal{L}|_{\{1, \psi^3 \}^{\perp}}) = \left\{ \begin{array}{ll} 1, \quad & \sigma_0 \neq 0, \\
2, \quad & \sigma_0 = 0, \end{array}\right.
\label{count-constraints}
\end{equation}
where $\sigma_0 := \langle \mathcal{L}^{-1} 1, 1 \rangle$.
If $1 \notin {\rm Range}(\mathcal{L})$, then
\begin{equation}
n(\mathcal{L}|_{\{1, \psi^3 \}^{\perp}}) = 1, \quad
z(\mathcal{L}|_{\{1, \psi^3 \}^{\perp}}) = 1.
\label{count-constraints-not-in-range}
\end{equation}
\end{lemma}

\begin{proof}
By using the result of Theorem 4.1 in \cite{Pel-book},
we construct the following symmetric $2$-by-$2$ matrix related to the two constraints in (\ref{infB}):
$$
P(\lambda) := \left[\begin{array}{cc}\langle (\mathcal{L} - \lambda I)^{-1} \psi^3,\psi^3 \rangle &
\langle (\mathcal{L} - \lambda I)^{-1}\psi^3, 1 \rangle \\
\langle (\mathcal{L} - \lambda I)^{-1} 1, \psi^3 \rangle &
\langle (\mathcal{L} - \lambda I)^{-1}1,1\rangle
\end{array}\right].
$$
If $1 \in {\rm Range}(\mathcal{L})$, then
\begin{eqnarray}
\label{projection}
\langle \mathcal{L}^{-1}1,1\rangle = \sigma_0, \quad
\langle \mathcal{L}^{-1} 1, \psi^3 \rangle = \langle \mathcal{L}^{-1} \psi^3, 1 \rangle  = 0, \quad
\langle \mathcal{L}^{-1} \psi^3, \psi^3 \rangle = -\frac{1}{4} \int_{-\pi}^{\pi} \psi^4 dx,
\end{eqnarray}
thanks to equation (\ref{range-2}).
By Theorem 4.1 in \cite{Pel-book}, we have the following identities:
\begin{equation}
\label{identneg}
\left\{ \begin{array}{l}
n(\mathcal{L} \big|_{\{1,\psi^3\}^{\bot}}) = n(\mathcal{L}) - n_0 - z_0, \\
z(\mathcal{L} \big|_{\{1,\psi^3\}^{\bot}}) = z(\mathcal{L}) + z_0,
\end{array} \right.
\end{equation}
where $n_0$ and $z_0$ are the numbers of negative and zero eigenvalues of $P(0)$.
Since $n(\mathcal{L}) = 2$ and $z(\mathcal{L}) = 1$ by Lemma \ref{lemma-index},
the count (\ref{identneg}) yields (\ref{count-constraints}) due to (\ref{projection}).

If $1 \notin {\rm Range}(\mathcal{L})$, then $z(\mathcal{L}) = 2$ but $z(\mathcal{L}|_{X_0}) = 1$ by Lemma \ref{lem-mean-value}. 
By Theorem 4.1 in \cite{Pel-book}, the count (\ref{identnegLL})
must be replaced by
\begin{equation}\label{identneg-singular}
\left\{ \begin{array}{l}
n(\mathcal{L} \big|_{\{1,\psi^3\}^{\bot}}) = n(\mathcal{L}) - n_0 - z_0, \\
z(\mathcal{L} \big|_{\{1,\psi^3\}^{\bot}}) = z(\mathcal{L}) + z_0 - z_{\infty},
\end{array} \right.
\end{equation}
where $z_{\infty} = 1$, $z_0 = 0$, and $n_0 = 1$. The count (\ref{identneg-singular}) yields (\ref{count-constraints-not-in-range}).
\end{proof}

It follows by Lemma \ref{lemma-characterization} that the ground state of the variational problem
(\ref{minBfunc}) is a local minimizer of the variational problem (\ref{infB}) if $\sigma_0 \leq 0$,
which is only degenerate by the translational symmetry if $\sigma_0 \neq 0$, whereas it is the saddle
point of the variational problem (\ref{infB}) if $\sigma_0 > 0$ or if 
$1 \notin  {\rm Range}(\mathcal{L})$, in which case $\sigma_0$ is unbounded.

Equipped with the variational characterization of Lemma \ref{lemma-characterization}, we can clarify the spectral stability of the odd periodic waves. The following theorem gives the relevant result.

\begin{theorem}
	\label{theorem-stability}
Let $\alpha \in (\frac{1}{2},2]$ and $\psi \in H^{\alpha}_{\rm per, odd}$ be a 
solution obtained in Theorem \ref{theorem-existence} and Corollary \ref{corollary-existence}.
If $1 \in {\rm Range}(\mathcal{L})$, then the periodic wave is spectrally stable if
\begin{equation}
\sigma_0 \leq 0, \quad \frac{d}{dc} \| \psi \|_{L^2}^2 \geq 0,
\label{stability-constraints}
\end{equation}
and is spectrally unstable with exactly one real positive eigenvalue of $\partial_x \mathbb{L}$ in $L^2(\mathbb{T})$ if
\begin{equation}
{\rm either} \;\; \sigma_0 \frac{d}{dc} \| \psi \|_{L^2}^2 > 0 
\quad \mbox{\rm or} \;\; \sigma_0 = 0, \;\; \frac{d}{dc} \| \psi \|_{L^2}^2 < 0, \quad \mbox{\rm or} \;\; \sigma_0 > 0, \;\; \frac{d}{dc} \| \psi \|_{L^2}^2 = 0,
\label{instability-constraints}
\end{equation}
where $\sigma_0 := \langle \mathcal{L}^{-1} 1, 1 \rangle$.
If $1 \notin {\rm Range}(\mathcal{L})$, then
the periodic wave is spectrally unstable with exactly one real positive eigenvalue of $\partial_x \mathbb{L}$ in $L^2(\mathbb{T})$ if
\begin{equation}
\label{instability-constraints-singular}
\frac{d}{dc} \| \psi \|^2_{L^2} \geq 0.
\end{equation}
\end{theorem}

\begin{proof}
	It is well-known \cite{haragus} that the periodic wave $\psi$ is spectrally stable
	if it is a constrained minimizer of energy (\ref{Eu}) under fixed momentum (\ref{Fu}) and mass (\ref{Mu}).
	Since $\mathcal{L}$ is the Hessian operator	for $G(u)$ in (\ref{operator}), the spectral stability holds if
	\begin{equation}\label{LLpositive}
	\mathcal{L} \big|_{\{1,\psi\}^{\bot}} \geq 0.
	\end{equation}
	On the other hand, the periodic wave $\psi$ is spectrally unstable with exactly one real positive
eigenvalue if $n \left(\mathcal{L} \big|_{\{1,\psi\}^{\bot}}\right) = 1$, whereas the
case $n \left(\mathcal{L} \big|_{\{1,\psi\}^{\bot}}\right) = 2$ is inconclusive (see \cite{Pel}).

Similarly to the proof of Lemma \ref{lemma-characterization},
	we construct the following symmetric $2$-by-$2$ matrix related to the two constraints in (\ref{LLpositive}):
	$$
	D(\lambda) := \left[\begin{array}{cc}\langle (\mathcal{L} - \lambda I)^{-1} \psi,\psi \rangle &
	\langle (\mathcal{L} - \lambda I)^{-1}\psi, 1 \rangle \\
	\langle (\mathcal{L} - \lambda I)^{-1} 1, \psi \rangle &
	\langle (\mathcal{L} - \lambda I)^{-1}1,1\rangle
	\end{array}\right].
	$$
If $1 \in {\rm Range}(\mathcal{L})$, then
	\begin{eqnarray}
	\label{projections-stab}
		\langle \mathcal{L}^{-1}1,1\rangle = \sigma_0, \quad
		\langle \mathcal{L}^{-1} 1, \psi \rangle = \langle \mathcal{L}^{-1} \psi, 1 \rangle = 0, \quad
		\langle \mathcal{L}^{-1} \psi, \psi \rangle = - \frac{1}{2} \frac{d}{dc} \| \psi \|^2_{L^2},
	\end{eqnarray}
where we have used $\mathcal{L} \partial_c \psi = -\psi$ from Lemma \ref{prop-continuation},
which can be applied since $z(\mathcal{L} |_{L^2_{\rm odd}}) = 0$ follows from the proof of Lemma \ref{lemma-index}.
	By Theorem 4.1 in \cite{Pel-book}, we have the following identities:
	\begin{equation}\label{identnegLL}
	\left\{ \begin{array}{l}
	n(\mathcal{L} \big|_{\{1,\psi\}^{\bot}}) = n(\mathcal{L}) - n_0 - z_0, \\
	z(\mathcal{L} \big|_{\{1,\psi\}^{\bot}}) = z(\mathcal{L}) + z_0,
	\end{array} \right.
	\end{equation}
where $n_0$ and $z_0$ are the numbers of negative and zero eigenvalues of $D(0)$.
Since $n(\mathcal{L}) = 2$ and $z(\mathcal{L}) = 1$ by Lemma \ref{lemma-index},
the count (\ref{identnegLL}) implies $n(\mathcal{L} \big|_{\{1,\psi\}^{\bot}}) = 0$ 
due to (\ref{projections-stab}) 
if the conditions (\ref{stability-constraints}) are satisfied and
$n(\mathcal{L} \big|_{\{1,\psi\}^{\bot}}) = 1$
if the condition (\ref{instability-constraints}) is satisfied.

If $1 \notin {\rm Range}(\mathcal{L})$, then $z(\mathcal{L}) = 2$ but $z(\mathcal{L}|_{X_0}) = 1$
by Lemma \ref{lem-mean-value}. By Theorem 4.1 in \cite{Pel-book}, the count (\ref{identnegLL})
must be replaced by
	\begin{equation}\label{identnegLL-singular}
	\left\{ \begin{array}{l}
	n(\mathcal{L} \big|_{\{1,\psi\}^{\bot}}) = n(\mathcal{L}) - n_0 - z_0, \\
	z(\mathcal{L} \big|_{\{1,\psi\}^{\bot}}) = z(\mathcal{L}) + z_0 - z_{\infty},
	\end{array} \right.
	\end{equation}
where $z_{\infty} = 1$ and $n_0 + z_0 = 1$ if the condition (\ref{instability-constraints-singular})
is satisfied. In this case, $n(\mathcal{L} \big|_{\{1,\psi\}^{\bot}}) = 1$ and the periodic wave is spectrally unstable.
\end{proof}

\begin{remark}
If $1 \in {\rm Range}(\mathcal{L})$, the case $\sigma_0 > 0$ and $\frac{d}{dc} \| \psi \|_{L^2}^2 < 0$ is inconclusive because $n(\mathcal{L} \big|_{\{1,\psi\}^{\bot}}) = 2$. In this case, one needs to find if the spectral stability problem has eigenvalues $\lambda \in i \mathbb{R}$
with so-called negative Krein signature, see \cite{Pel} for further details.
The same is true if $1 \notin {\rm Range}(\mathcal{L})$ and $\frac{d}{dc} \| \psi \|_{L^2}^2 < 0$.
\label{rem-stability}
\end{remark}

\section{Examples of odd periodic waves}

Here we present three examples of the odd periodic waves obtained in Theorem \ref{theorem-existence} and Corollary \ref{corollary-existence}. 
Small-amplitude waves are studied with the explicit analytical computations of the perturbation expansions. Periodic waves in the local case $\alpha = 2$ are handled by explicit analytical computations of elliptic functions. Periodic waves for $\alpha = 1$ are considered by means of numerical approximations. 

\subsection{Stokes expansion of small-amplitude waves}

Stokes expansions of small-amplitude periodic waves near the bifurcation point $c = -1$ are
rather standard in getting precise results on the existence and stability of periodic waves \cite{J,lepeli,NPU}.
The following proposition describes the properties of the small-amplitude periodic waves.

\begin{proposition}
For each $\alpha \in \left(\frac{1}{2},2\right]$,
there exists $c_0\in (-1,\infty)$ such that the odd periodic wave exists for $c \in (-1,c_0)$
with $n(\mathcal{L}) = 2$, $z(\mathcal{L}) = 1$ and is spectrally stable.
\label{prop-Stokes}
\end{proposition}

\begin{proof}
We solve the stationary equation (\ref{ode-stat}) with $b = 0$ in the space of odd functions by using Stokes expansions in terms of small amplitude $A$:
\begin{equation}
\label{Stokes}
\psi(x) = A \psi_1(x) + A^3 \psi_3(x) + \mathcal{O}(A^5)
\end{equation}
and
\begin{equation}
\label{Stokes-speed}
 c = -1 + A^2 c_2 + \mathcal{O}(A^4).
\end{equation}
We obtain recursively: $\psi_1(x) = \sin(x)$, 
$$
\psi_3(x) = \frac{1}{2(1 - 3^{\alpha})} \sin(3x),
$$
and $c_2 = \frac{3}{2}$ uniformly in $\alpha$.

Since $\mathcal{L} = D^{\alpha} - 1 + \mathcal{O}(A^2)$,
then $1 \in {\rm Range}(\mathcal{L})$ for small $A$, so that $n(\mathcal{L}) = 2$ and $z(\mathcal{L}) = 1$
for any $\alpha \in \left(\frac{1}{2},2\right]$ by Lemma \ref{lemma-index}.

Furthermore, $\mathcal{L}^{-1} 1 = -1 + \mathcal{O}(A^2)$, so that $\sigma_0 = 
\langle \mathcal{L}^{-1} 1, 1 \rangle = -2 \pi + \mathcal{O}(A^2) < 0$,
which implies $n(\mathcal{L}|_{\{1, \psi^3 \}^{\perp}}) = 0$ and $z(\mathcal{L}|_{\{1, \psi^3 \}^{\perp}}) = 1$
by Lemma \ref{lemma-characterization}. Hence, the 
odd periodic wave for small amplitude $A$
represents a local minimizer of the variational problem
(\ref{infB}) with two constraints for any $\alpha \in \left(\frac{1}{2},2\right]$.

Finally, we obtain $\| \psi \|^2_{L^2} = \pi A^2 + \mathcal{O}(A^4)$ so that
$$
\frac{d}{dc} \| \psi \|^2_{L^2} = \frac{2\pi}{3} \left[ 1 + \mathcal{O}(A^2) \right] > 0.
$$
By Theorem \ref{theorem-stability}, the Stokes wave (\ref{Stokes}) for small $A$ is spectrally stable since the criterion (\ref{stability-constraints}) is satisfied.
\end{proof}

\begin{remark}
	For $c$ near $-1$, one can show by extending Lemma 2.3 in \cite{NPU} that the small-amplitude wave of Proposition \ref{prop-Stokes} coincides with the solution obtained in Theorem \ref{theorem-existence} and Corollary \ref{corollary-existence}.
\end{remark}

\subsection{Local case $\alpha = 2$}

In the case of the modified KdV equation ($\alpha = 2$), the stationary equation (\ref{ode-stat}) with $b = 0$ can be solved
in the space of odd functions by using the Jacobian cnoidal function.
Let us recall the normalized solution $\psi_0(z) = k {\rm cn}(z;k)$ of the
second-order equation
\begin{equation}
\label{normalized-cn}
\psi_0''(z) + (1-2k^2) \psi_0(z) + 2 \psi_0(z)^3 = 0.
\end{equation}
By adopting a scaling transformation and a translation of the even function ${\rm cn}(z;k)$ by a quarter-period,
we obtain the exact solution for the odd periodic wave in the form:
\begin{equation}
\label{cnoidal}
\psi(x) = \frac{2}{\pi} k K(k) {\rm cn}\left[\frac{2}{\pi} K(k) x - K(k);k\right]
= \frac{2}{\pi}  k \sqrt{1-k^2} K(k) \frac{{\rm sn}\left[\frac{2}{\pi}  K(k) x;k\right]}{{\rm dn}\left[\frac{2}{\pi}  K(k) x;k\right]}
\end{equation}
with
\begin{equation}
\label{cnoidal-speed}
c = \frac{4}{\pi^2} K(k)^2 (2k^2-1),
\end{equation}
where $K(k)$ is the complete elliptic integral of the first kind.
We recall some properties of complete elliptic integrals 
$K(k)$ and $E(k)$ of the first and second kinds, respectively:
\begin{eqnarray*}
&& \mbox{\rm (a)} \quad E(0) = K(0) = \frac{\pi}{2}, \\
&& \mbox{\rm (b)} \quad E(k) \to  1, \quad K(k) \to \infty, \quad \mbox{\rm as} \quad k \to 1,
\end{eqnarray*}
and
$$
\mbox{\rm (c)} \quad \frac{d}{dk} E(k) = \frac{E(k) - K(k)}{k} < 0, \quad \frac{d}{dk} K(k) = \frac{E(k)}{k(1-k^2)} - \frac{K(k)}{k} > 0.
$$
It follows from (a) and (c) that
\begin{equation}
\label{property-elliptic}
(1-k^2) K(k) < E(k) < K(k), \quad k \in (0,1).
\end{equation}

The following proposition summarizes properties of the odd periodic waves for $\alpha = 2$.

\begin{proposition}
Fix $\alpha = 2$. The odd periodic wave (\ref{cnoidal}) 
exists for every $c \in (-1,\infty)$ with $n(\mathcal{L}) = 2$ and $z(\mathcal{L}) = 1$.
There exists $c_* \in (-1,\infty)$ such that the odd periodic wave is spectrally
stable for $c \in (-1,c_*]$ and is spectrally unstable with one real positive eigenvalue
for $c \in (c_*,\infty)$.
\label{prop-cnoidal}
\end{proposition}

\begin{proof}
The mapping $(0,1) \ni k \mapsto c(k) \in (-1,\infty)$ is one-to-one and onto. This follows from
$$
\frac{\pi^2}{8} \frac{dc}{dk} = \frac{K(k)}{k(1-k^2)} \left[ (1-k^2) [K(k) - E(k)] + k^2 E(k) \right] > 0,
$$
where property (\ref{property-elliptic}) has been used. Hence,
the odd periodic wave parameterized by $k \in (0,1)$ in (\ref{cnoidal}) and (\ref{cnoidal-speed})
exists for every $c \in (-1,\infty)$.

The first five eigenvalues and eigenfunctions of the normalized linearized operator
\begin{equation}
\label{operator-cn}
\mathcal{L}_0 = -\partial_z^2 + 2k^2 - 1 - 6 k^2 {\rm cn}(z;k)^2
\end{equation}
are known in space $L^2(-2K(k),2K(k))$ in the explicit form \cite{angulo1,DK}. The two negative eigenvalues and a simple zero
eigenvalue with the corresponding eigenfunctions are given by
\begin{eqnarray*}
& \lambda_0 = 1 - 2 k^2 - 2 \sqrt{1-k^2+k^4}, \quad & \varphi_0(z) = 1 + k^2 + \sqrt{1-k^2+k^4} - 3 k^2 {\rm sn}(z;k)^2,\\
& \lambda_1 = -3k^2, \quad & \varphi_1(z) = {\rm cn}(z;k) {\rm dn}(z;k), \\
& \lambda_2 = 0, \quad & \varphi_2(z) = {\rm sn}(z;k) {\rm dn}(z;k).
\end{eqnarray*}
The next two positive eigenvalues with the corresponding eigenfunctions are given by
\begin{eqnarray*}
& \lambda_3 = 3 (1 - k^2), \quad & \varphi_3(z) = {\rm sn}(z;k) {\rm cn}(z;k), \\
& \lambda_4 = 1 - 2 k^2 + 2 \sqrt{1-k^2+k^4}, \quad & \varphi_4(z) = 1 + k^2 - \sqrt{1-k^2+k^4} - 3 k^2 {\rm sn}(z;k)^2.
\end{eqnarray*}
Eigenvalues and eigenvectors of the linearized operator $\mathcal{L}$ are obtained after
the same scaling and translational transformation as in (\ref{cnoidal}). In agreement with Lemma \ref{lemma-index},
we have $n(\mathcal{L}) = 2$, $z(\mathcal{L}) = 1$, and $1 \in {\rm Range}(\mathcal{L})$
for every $c \in (-1,\infty)$. Moreover, we compute
$$
\frac{1}{2\sqrt{1-k^2+k^4}} \mathcal{L}_0 \left[ \frac{\lambda_4 \varphi_0 - \lambda_0 \varphi_4}{\lambda_0 \lambda_4} \right] = 1
\quad \mbox{\rm and} \quad \frac{1}{2\sqrt{1-k^2+k^4}} \left[ \varphi_0 - \varphi_4 \right] = 1,
$$
from which it follows that 
\begin{eqnarray*}
\langle \mathcal{L}_0^{-1} 1, 1 \rangle = \frac{\lambda_4 \langle \varphi_0, 1 \rangle
- \lambda_0 \langle \varphi_4, 1 \rangle}{2 \sqrt{1 - k^2 + k^4} \lambda_0 \lambda_4} = - 4 \left[ 2 E(k) - K(k) \right].
\end{eqnarray*}
Since
$$
\frac{d}{dk} k E(k) = 2 E(k) - K(k) \quad \mbox{\rm and} \quad \frac{d^2}{dk^2} k E(k) = \frac{(1-k^2) [E(k) - K(k)] - k^2 E(k)}{k(1-k^2)} < 0,
$$
in addition to (b), there exists exactly one value of $k$, labeled as $k^* \approx 0.909$ in \cite{DK}, such
that $\langle \mathcal{L}_0^{-1} 1, 1 \rangle < 0$ for $k \in (0,k^*)$ and $\langle \mathcal{L}_0^{-1} 1, 1 \rangle > 0$
for $k \in (k^*,1)$. Up to a positive scaling factor, $\langle \mathcal{L}_0^{-1} 1, 1 \rangle$ gives the value of $\sigma_0 = \langle \mathcal{L}^{-1} 1, 1 \rangle$.
By Lemma \ref{lemma-characterization}, this implies that $n(\mathcal{L}|_{\{1, \psi^3 \}^{\perp}}) = 0$
for $k \in (0,k^*]$ and $n(\mathcal{L}|_{\{1, \psi^3 \}^{\perp}}) = 1$ for $k \in (k^*,1)$.
Therefore, there exists a bifurcation at $k = k^*$ such that the odd periodic wave (\ref{cnoidal})
is a local minimizer of the variational problem (\ref{infB}) with two constraints for $k \in (0,k^*)$
and a saddle point for $k \in (k^*,1)$. The value of $k^*$ defines uniquely a value $c^* \approx 1.425$
by (\ref{cnoidal-speed}).

Finally, we obtain
$$
\| \psi \|^2_{L^2} = \frac{8}{\pi} K(k) \left[ E(k) - (1-k^2) K(k) \right] > 0
$$
and
$$
\frac{\pi}{9} \frac{d}{dk} \| \psi \|^2_{L^2} = \frac{1}{k(1-k^2)} \left[ (1-k^2) K(k) [K(k) - E(k)]
+ E(k) [E(k) - (1-k^2) K(k)] \right] > 0,
$$
for every $k \in (0,1)$, where the property (\ref{property-elliptic}) has been used. By Theorem \ref{theorem-stability} due to the stability and instability criteria (\ref{stability-constraints}) and (\ref{instability-constraints}),
the odd periodic wave (\ref{cnoidal}) with the speed (\ref{cnoidal-speed}) is spectrally stable
for $c \in (-1,c^*]$ and is spectrally unstable with exactly one real positive eigenvalue if $c \in (c^*,\infty)$.
\end{proof}

\begin{remark}
The statement of Proposition \ref{prop-cnoidal} agrees with the outcomes obtained in \cite{DK} and \cite{AN1}.
\end{remark}

\subsection{Numerical approximations}

Here we numerically compute solutions of the stationary equation
(\ref{ode-stat}) using Newton's method in the Fourier space, 
similar to our previous work \cite{NPU}.
For better performance, the odd periodic wave with profile $\psi$
in Theorem \ref{theorem-existence} is translated by a quarter period $\pi/2$
to an even function of $x$. The starting iteration is generated from the Stokes
expansion (\ref{Stokes}) after the translation and this solution
is uniquely continued in $c$ for all $c \in (-1,\infty)$.
This family of solutions correspond to $b = 0$ in the stationary equation (\ref{ode-stat}).


Additionally, we add a perturbation to the profile $\psi$ to preserve the even
symmetry but to break the odd symmetry after the translation. Numerical iterations
converge back to the same family of solutions with $b = 0$ for $c < c_*$, where
$c_* \in (-1,\infty)$ is the bifurcation point for which a nontrivial solution
in Lemma \ref{lem-mean-value} exists. The value of $c_*$ exists for all $\alpha \in \left( \frac{1}{2},2\right]$.
When $c > c_*$, numerical iterations converge
to a new family of solutions to the stationary equation (\ref{ode-stat}) with $b \neq 0$,
which is then continued with respect to $c$. Convergence of numerical iterations is measured by the $L^2$ norm of the residue for the stationary equation (\ref{ode-stat}), with the tolerance equals to $10^{-10}$.


\begin{figure}[htp]
	\centering
	\begin{subfigure}{0.5\textwidth}
		\centering
		\includegraphics[width=1.0\linewidth]{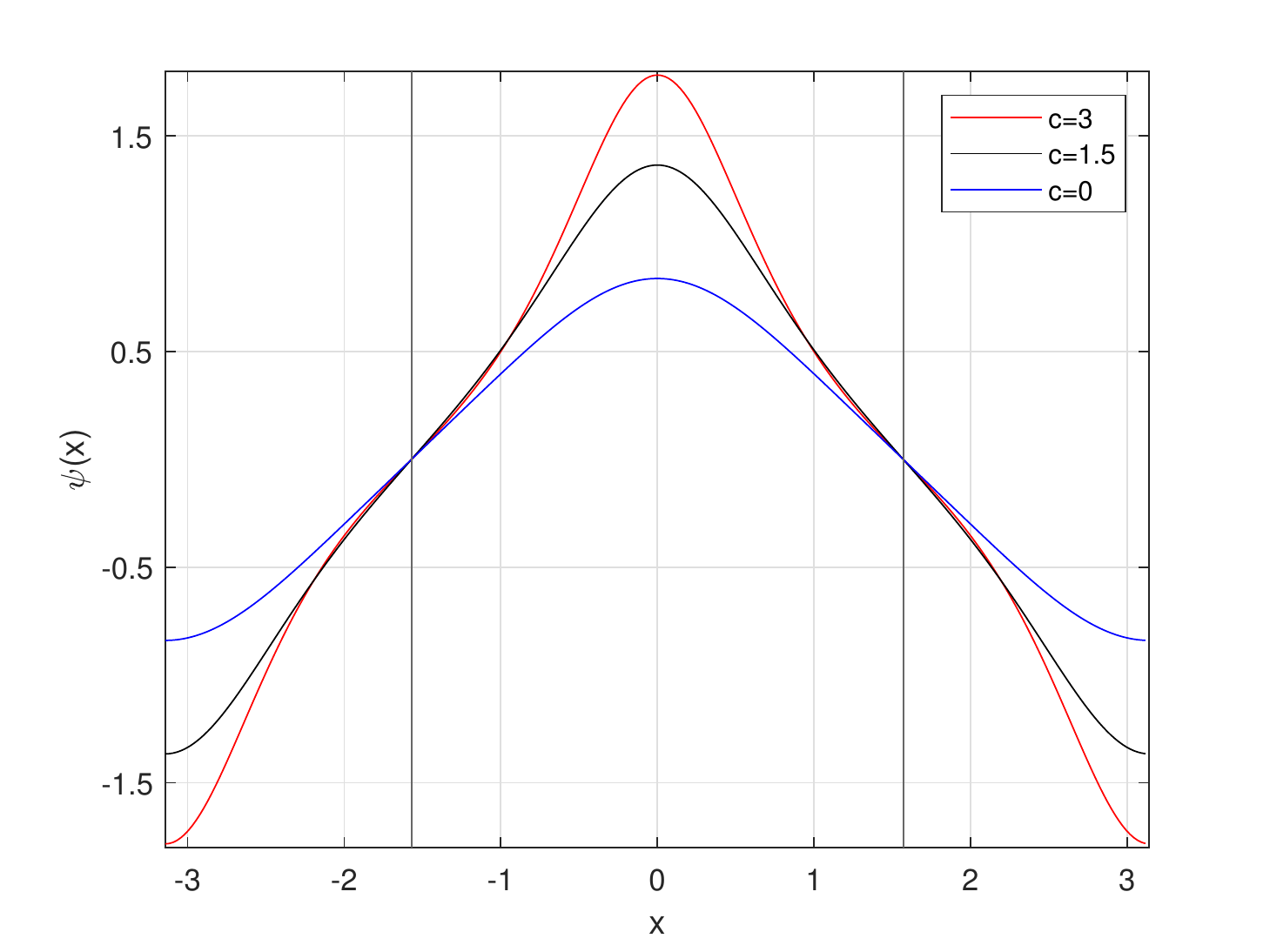}
	\end{subfigure}%
	\begin{subfigure}{0.5\textwidth}
		\centering
		\includegraphics[width=1.0\linewidth]{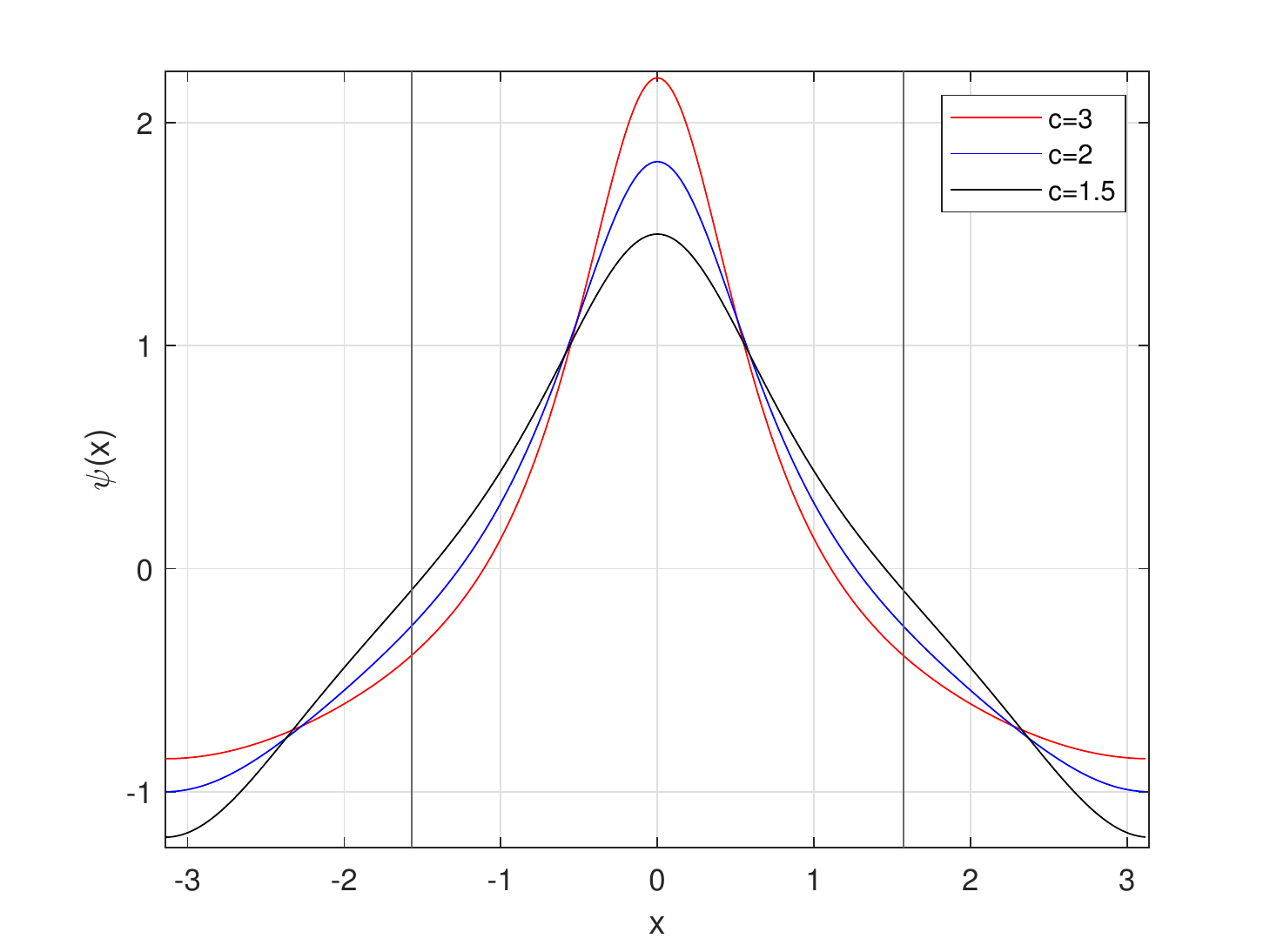}
	\end{subfigure}
	\begin{subfigure}{0.5\textwidth}
		\centering
		\includegraphics[width=1.0\linewidth]{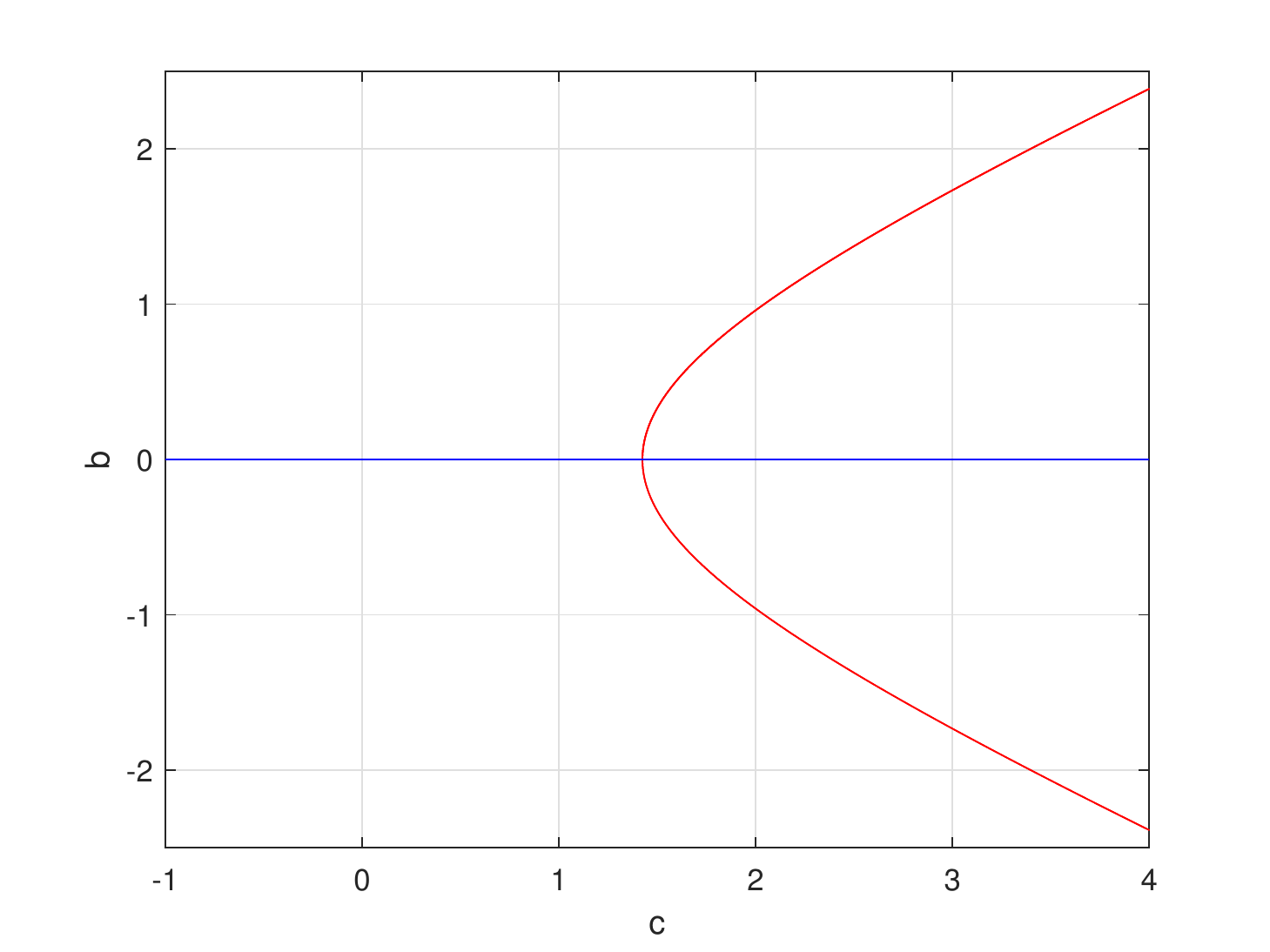}
	\end{subfigure}%
	\begin{subfigure}{0.5\textwidth}
		\centering
		\includegraphics[width=1.0\linewidth]{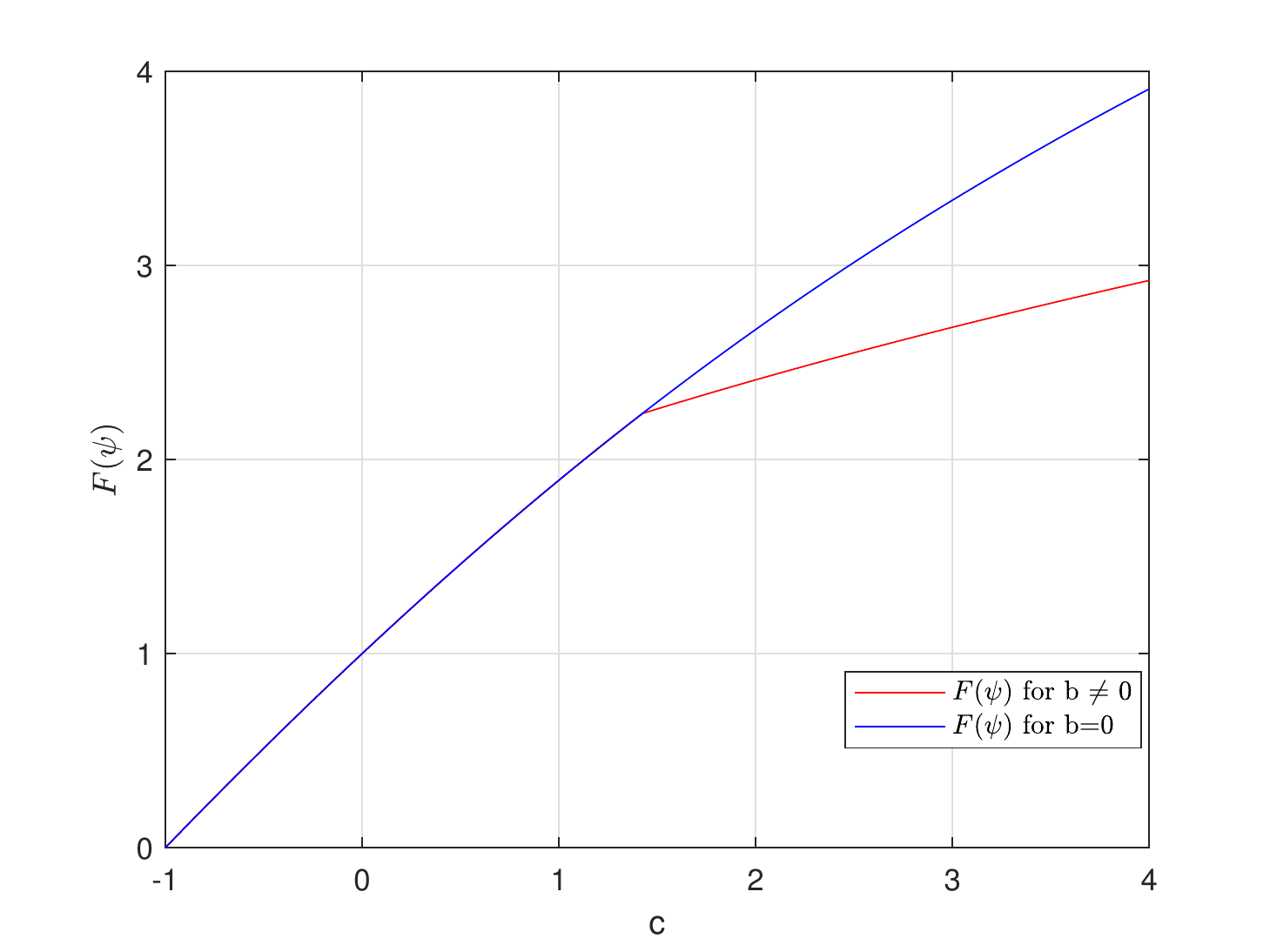}
	\end{subfigure}
	\begin{subfigure}{0.5\textwidth}
	\centering
	\includegraphics[width=1.0\linewidth]{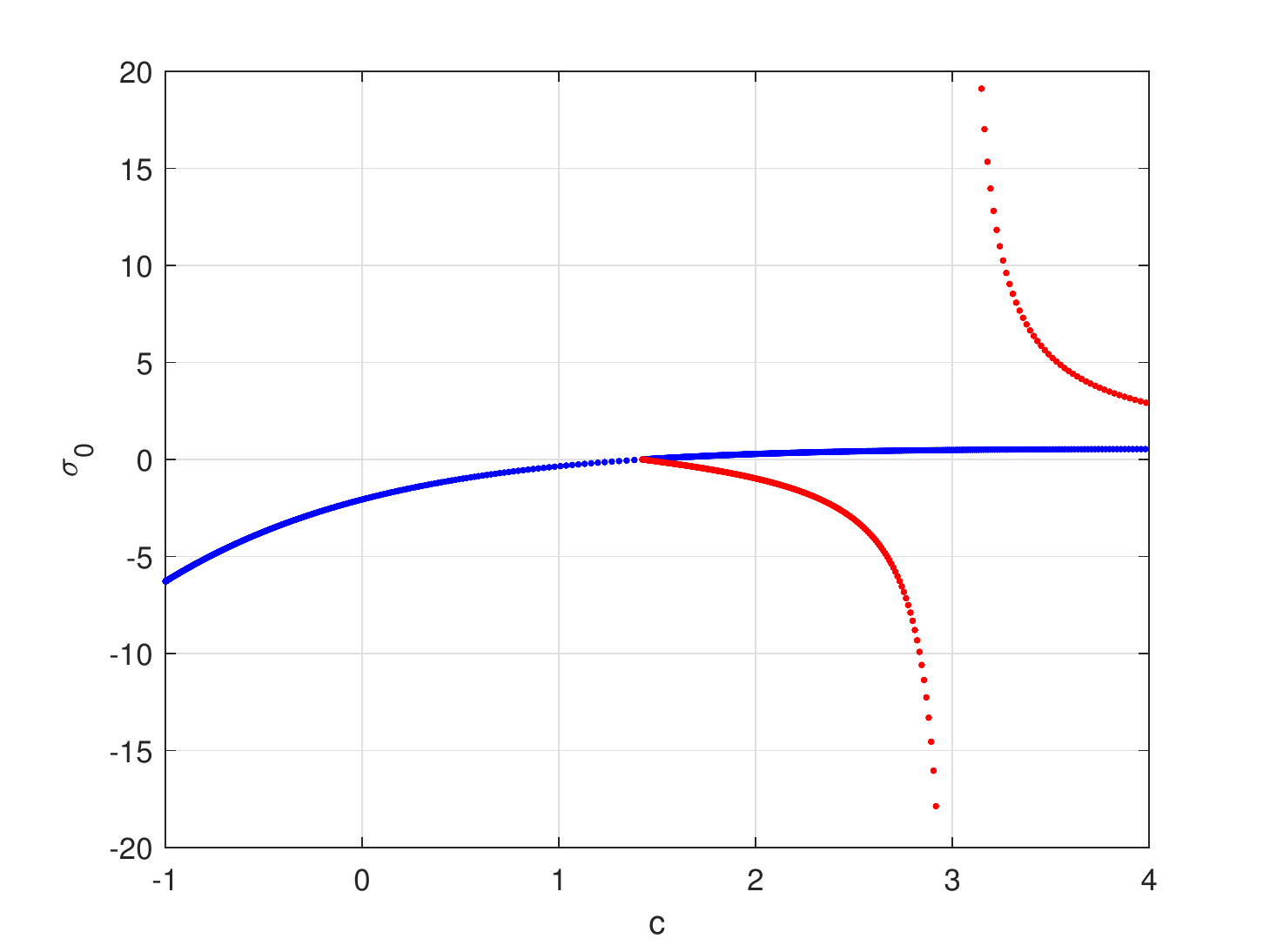}
		\end{subfigure}
\begin{subfigure}{0.49\textwidth}
		\centering
		\includegraphics[width=1.0\linewidth]{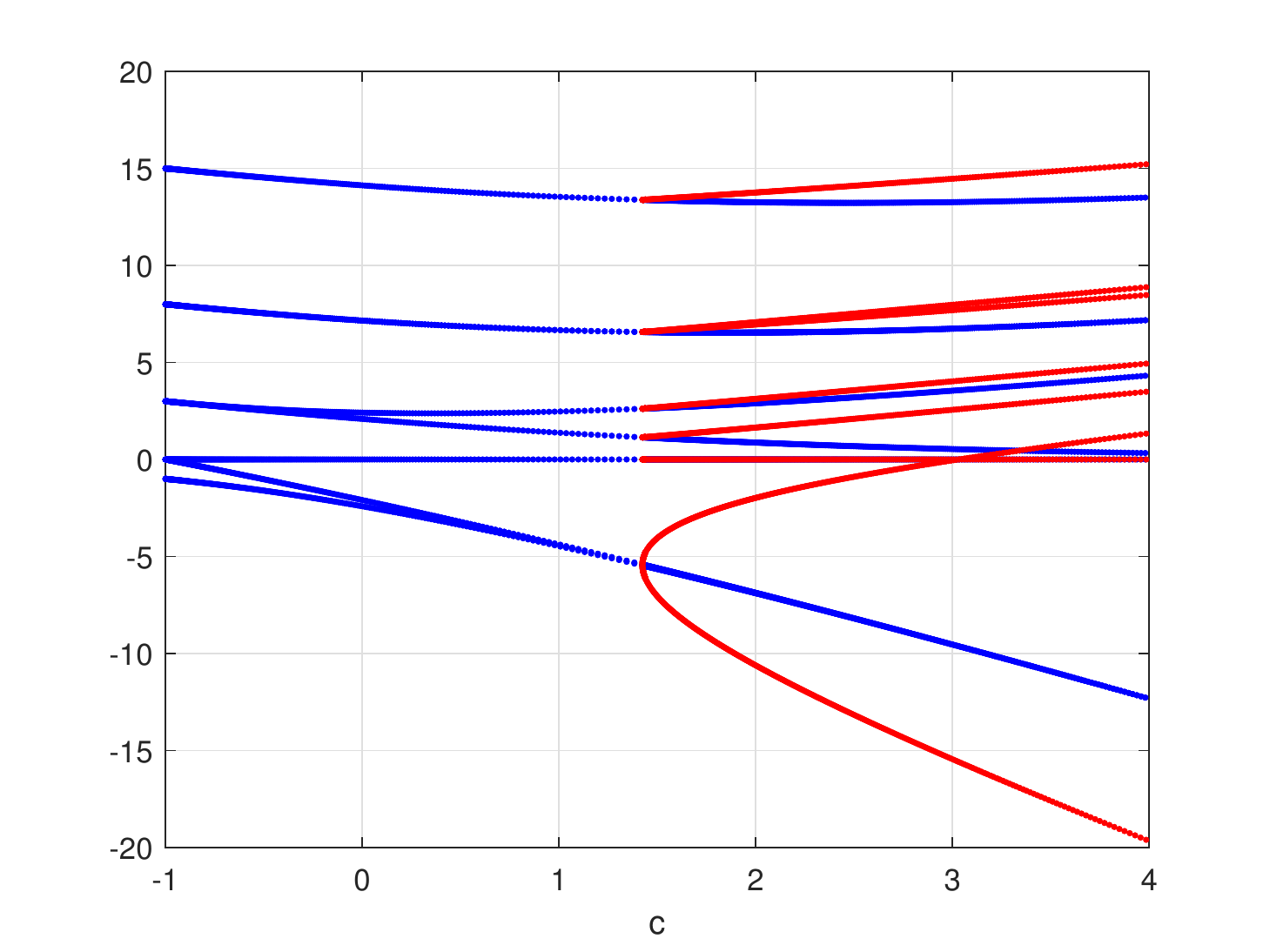}
	\end{subfigure}
	\caption{Periodic waves for $\alpha = 2$. Top left: Profiles of $\psi$ with $b = 0$
for three different values of $c$. Top right: Profiles of $\psi$ with $b \neq 0$ for three values of $c$.
Middle left: Dependence of $b$ versus $c$ showing the pitchfork bifurcation point $c_*$.
Middle right: Dependence of the momentum $F(\psi)$ versus $c$. Bottom left: Dependence of $\sigma_0$ versus $c$. Bottom right: The lowest eigenvalues of $\mathcal{L}$ versus $c$. The blue (red) line corresponds to the family with $b = 0$ ($b\neq 0$). }
	\label{fig:bifurcationalpha2}
\end{figure}

\begin{figure}[htp]
	\centering
	\begin{subfigure}{0.5\textwidth}
		\centering
		\includegraphics[width=1.0\linewidth]{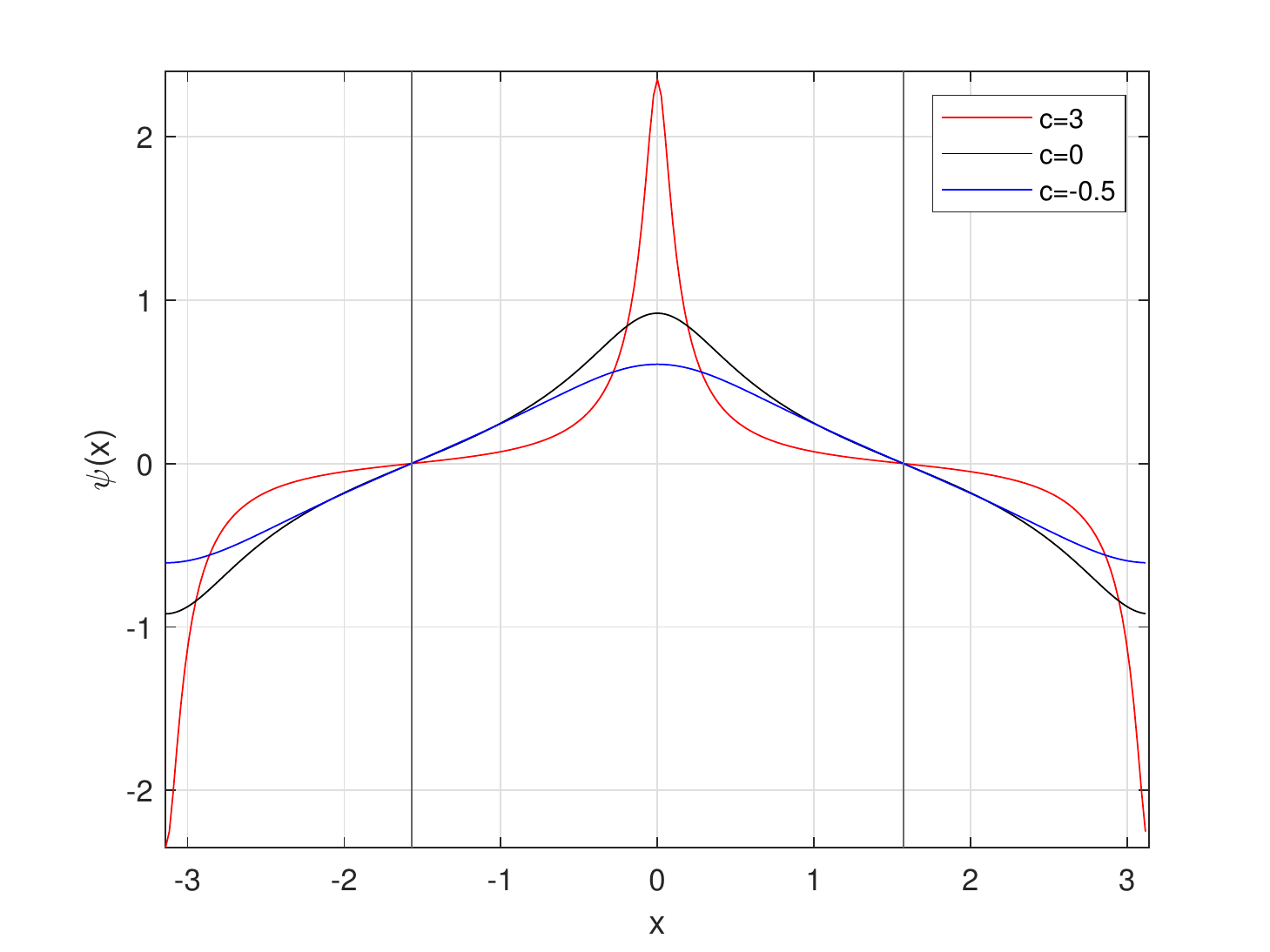}
	\end{subfigure}%
	\begin{subfigure}{0.5\textwidth}
		\centering
		\includegraphics[width=1.0\linewidth]{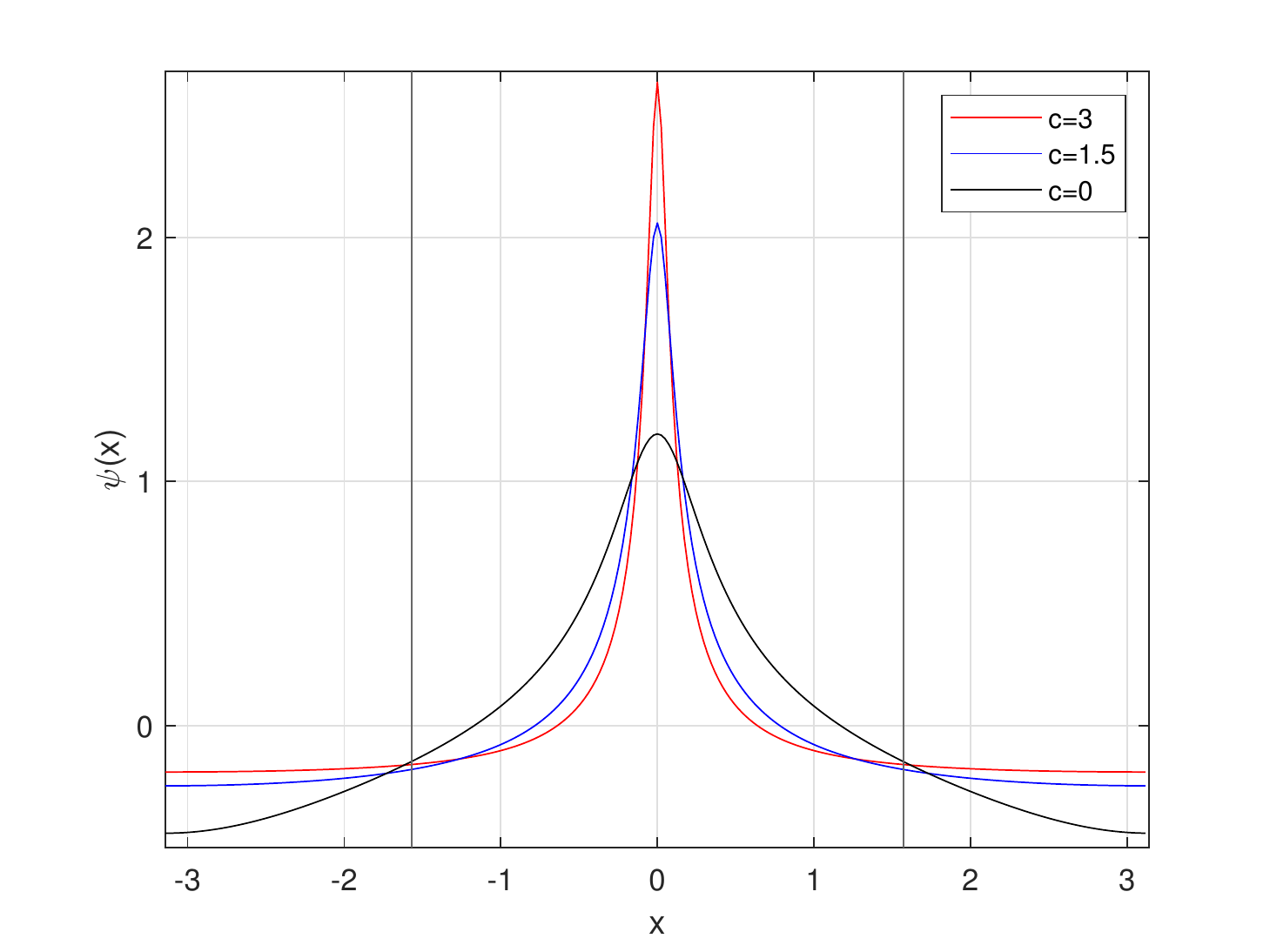}
	\end{subfigure}
	\begin{subfigure}{0.5\textwidth}
		\centering
		\includegraphics[width=1.0\linewidth]{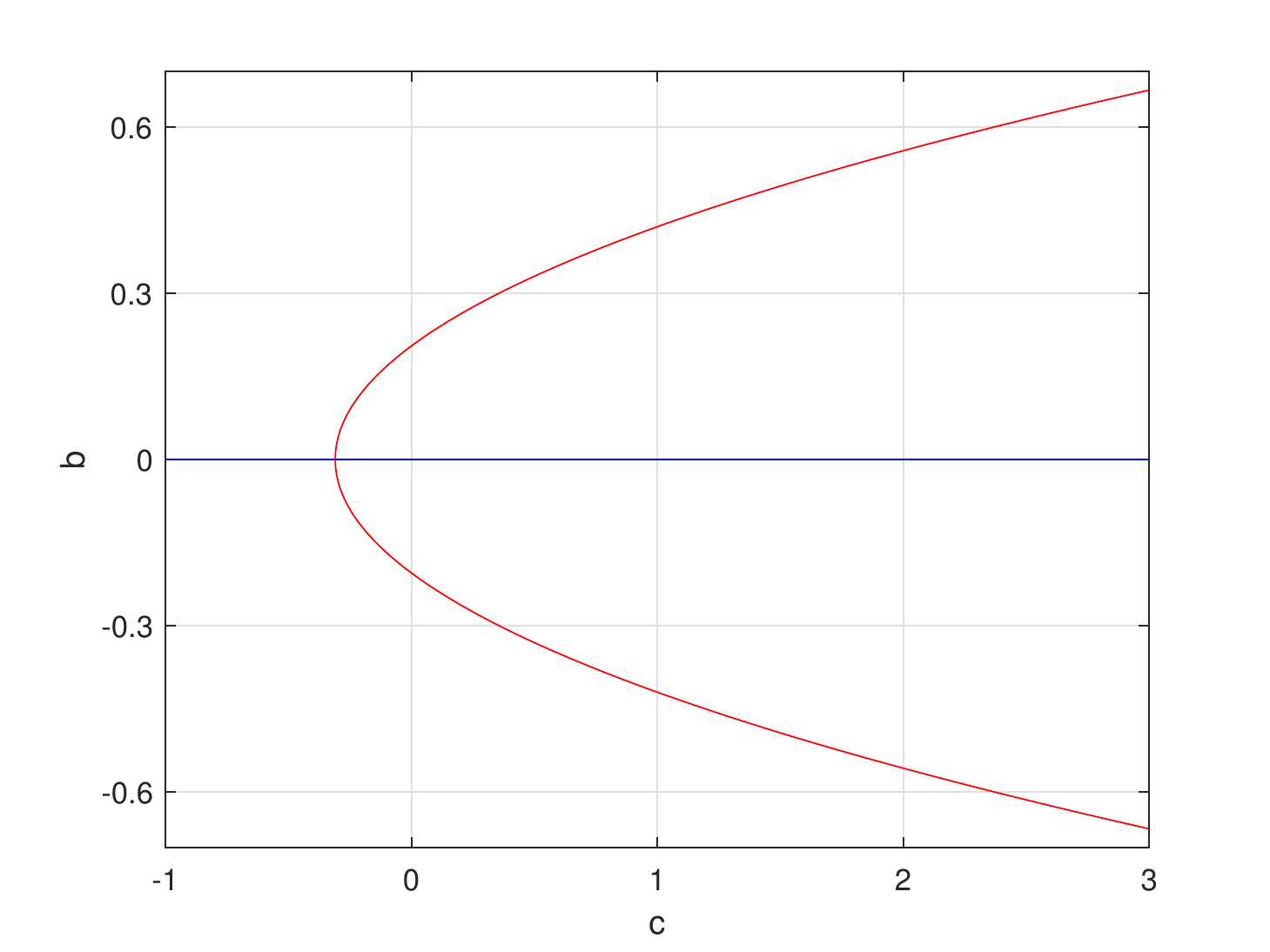}	
	\end{subfigure}%
	\begin{subfigure}{0.5\textwidth}
		\centering
		\includegraphics[width=1.0\linewidth]{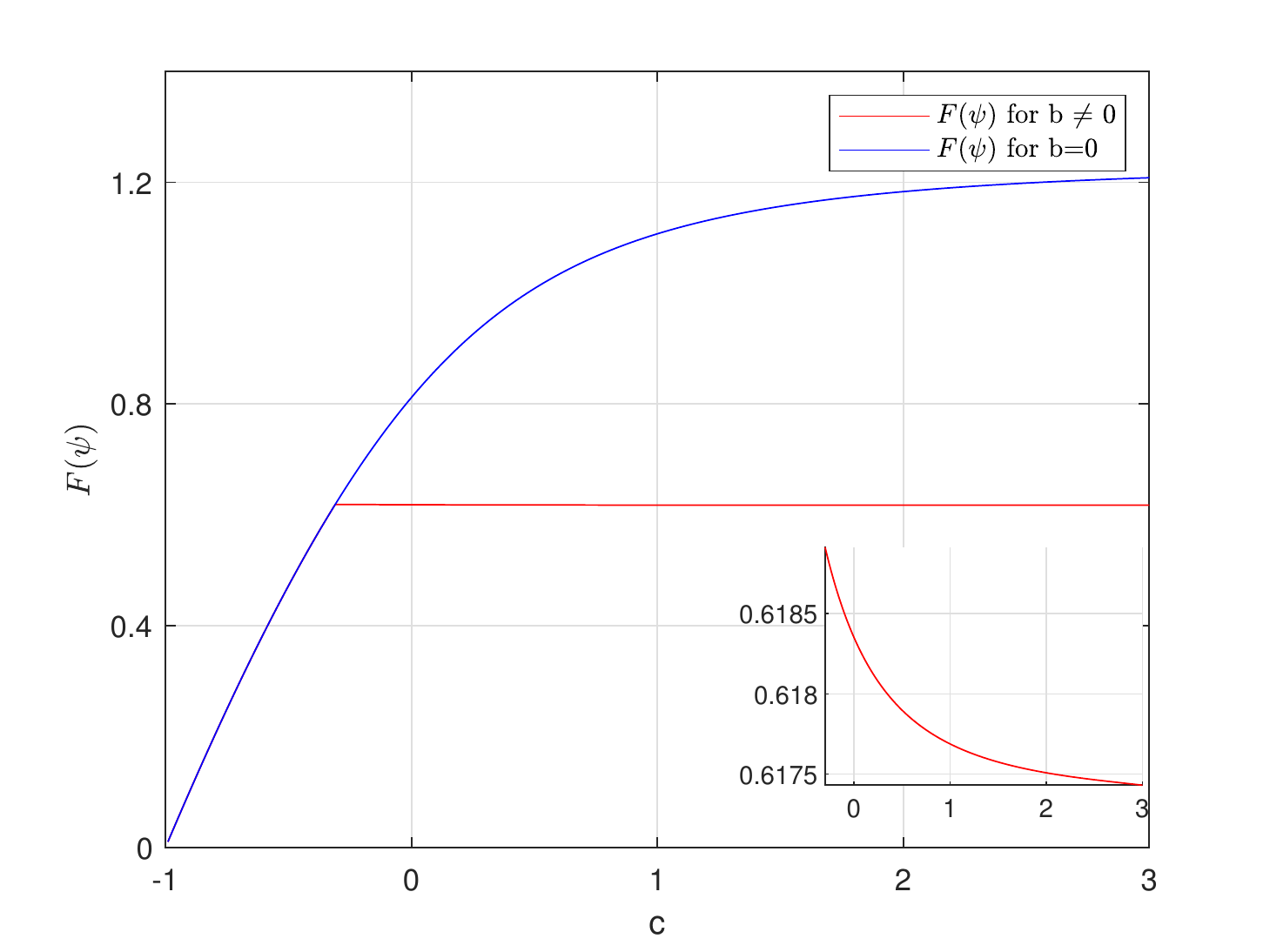}
	\end{subfigure}
	\begin{subfigure}{0.5\textwidth}
		\centering
		\includegraphics[width=1.0\linewidth]{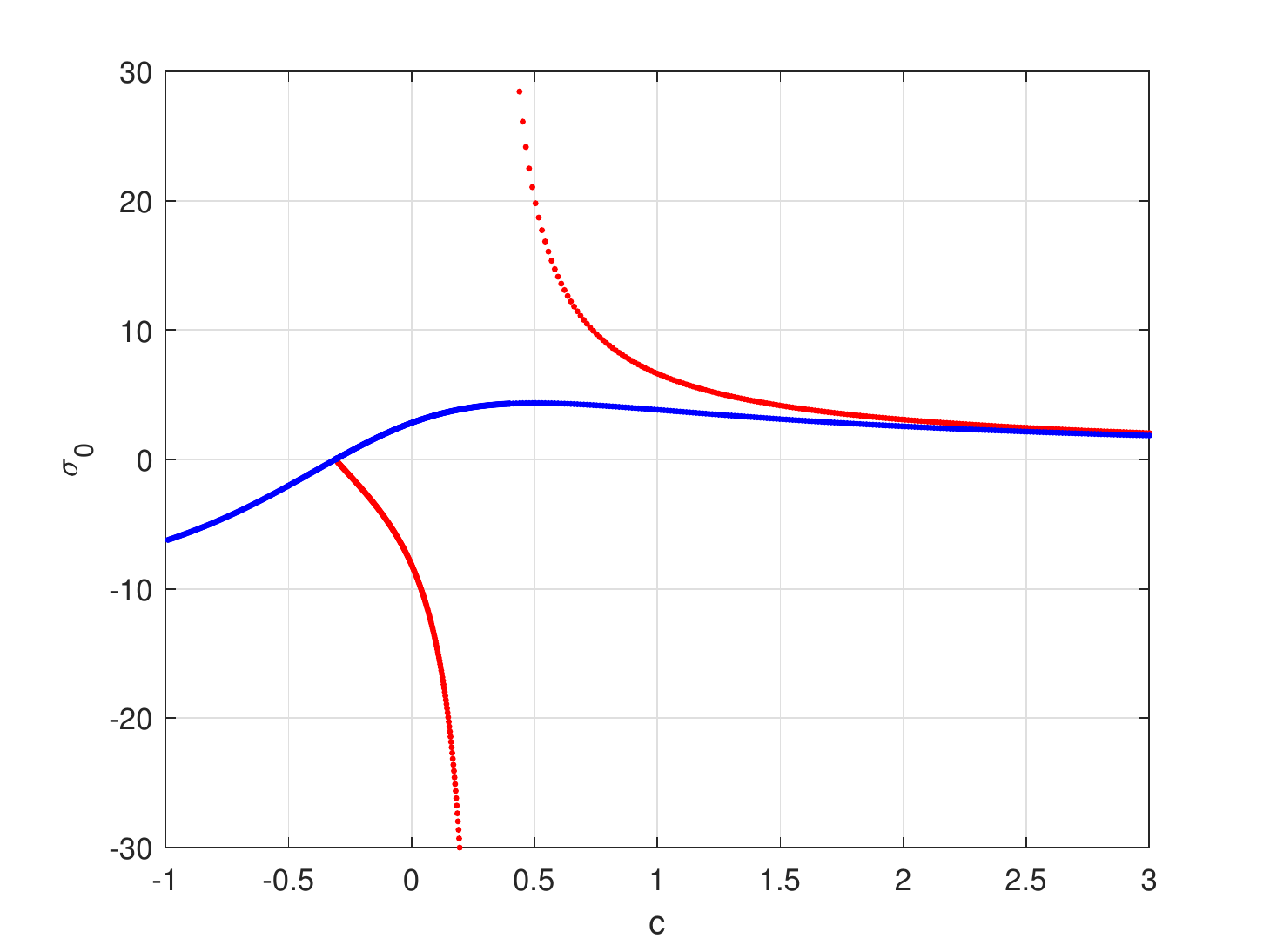}
	\end{subfigure}
	\begin{subfigure}{0.49\textwidth}
		\centering
		\includegraphics[width=1.0\linewidth]{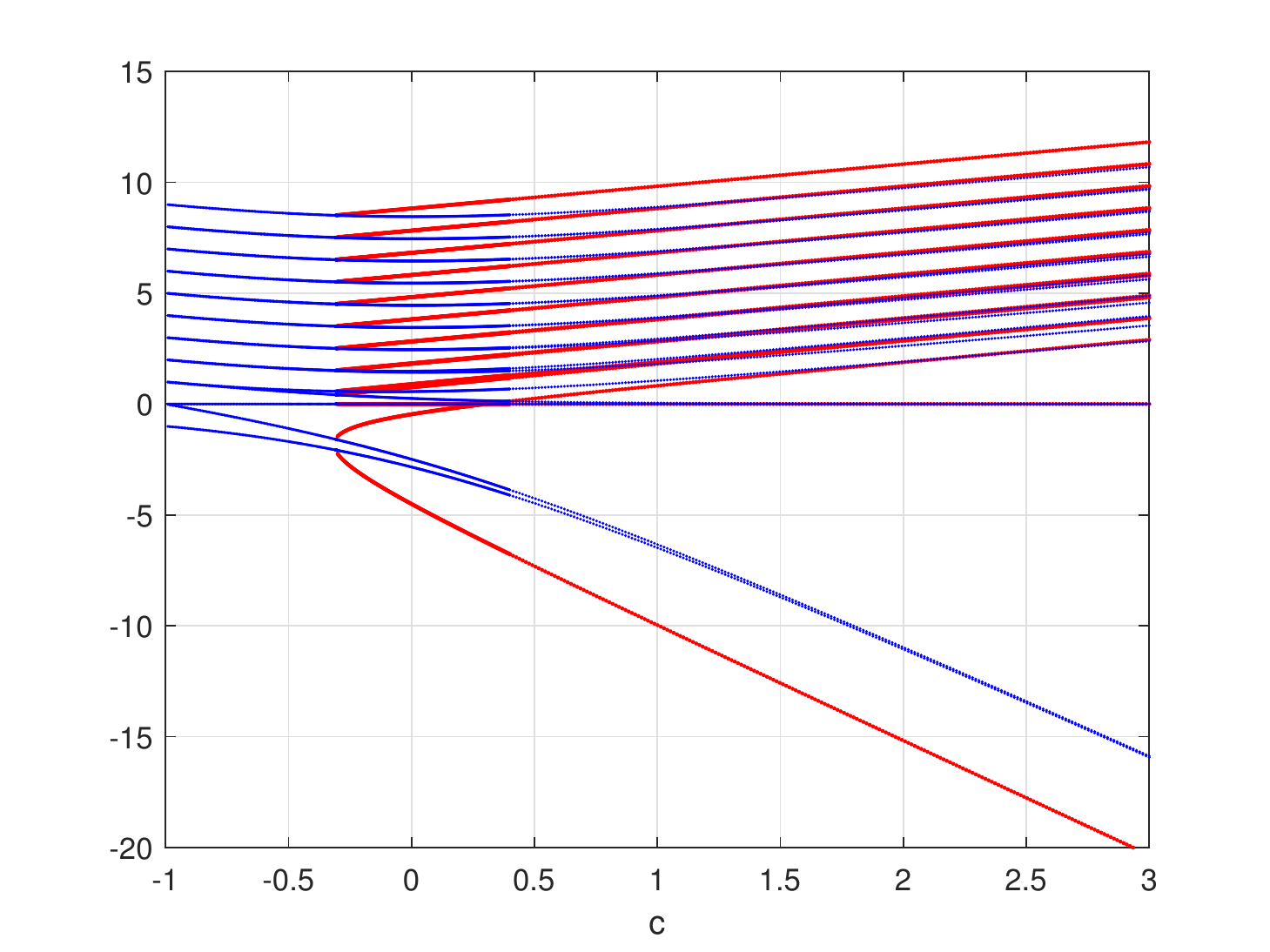}
	\end{subfigure}
	\caption{The same as Figure \ref{fig:bifurcationalpha2} but for $\alpha =1$.}
	\label{fig:bifurcationalpha1}
\end{figure}

Figure \ref{fig:bifurcationalpha2} presents the periodic wave solutions to the stationary equation
(\ref{ode-stat}) for $\alpha = 2$. The top left panel shows the profiles of $\psi$
of the family with $b = 0$ for three different values of $c$: near the Stokes wave limit (blue curve),
near the bifurcation point $c_*$ (black curve) and when $c$ is away from the bifurcation point $c_*$ (red curve).
The top right panel shows the profiles of $\psi$ of the bifurcating family with $b \neq 0$
near the bifurcation point $c_*$ (black curve) and increasingly away from the bifurcation point (blue and red curves).
The vertical lines show the symmetry points at $x = \pm \pi/2$. The family with $b = 0$ has odd symmetry with
respect to these points, whereas the family with $b \neq 0$ does not have this symmetry; both families
are even at $x = 0$ and $x = \pm \pi$.

The middle left panel of Figure \ref{fig:bifurcationalpha2} shows the dependence of $b$ in the stationary equation
(\ref{ode-stat}) versus speed $c$.
The pitchfork bifurcation point is located at $c_* \approx 1.425$. 
The two symmetric branches of solutions with $b > 0$ and $b < 0$ 
are obtained by using the positive and negative perturbations
to the family of solutions with $b = 0$.

The middle right panel of Figure \ref{fig:bifurcationalpha2} shows the momentum $F(\psi)$ versus $c$. The bottom left panel shows the dependence of $\sigma_0$ versus $c$. The bottom right panel displays the lowest eigenvalues of $\mathcal{L}$ versus $c$. The blue curve shows the family of solutions with $b = 0$, whereas the red curve shows the family of solutions with $b \neq 0$. 

The momentum $F(\psi)$ is increasing function of $c$ for both the families. In agreement with the theory, $\sigma_0$ for the family with $b = 0$
changes sign from negative to positive when $c$ passes through the bifurcation point $c_*$. By Theorem \ref{theorem-stability}, it follows that the family of solutions with $b = 0$ is spectrally stable
for $c < c_*$ and spectrally unstable for $c > c_*$. 

On the other hand, the bifurcating family with $b \neq 0$ has $\sigma_0 < 0$ near the bifurcation point but there exists another point $\hat{c}_* > c_*$ such that $\sigma_0$ diverges at $c = \hat{c}_*$ and becomes positive for $c > \hat{c}_*$. This agrees with the behavior of the lowest eigenvalues of $\mathcal{L}$ since $z(\mathcal{L}) = 2$ at $c = \hat{c}_*$, 
$n(\mathcal{L}) = 2$ for $c < \hat{c}_*$ and $n(\mathcal{L}) = 1$ for $c > \hat{c}_*$. Lemma \ref{lemma-characterization} and Theorem \ref{theorem-stability} are trivially extended to the family with $b \neq 0$ 
and they confirm that for both cases of $c < \hat{c}_*$ and $c > \hat{c}_*$, the periodic waves of the family with $b \neq 0$ correspond to
minimizers of the constrained variational problem (\ref{infB}) and they are spectrally stable for $c > c_*$.

Figure \ref{fig:bifurcationalpha1} presents the periodic wave solutions to the stationary equation
(\ref{ode-stat}) for $\alpha = 1$. Note that bifurcation point $c_*$ moves to the left
and becomes $c_* \approx -0.310$. The existence and stability of the family of solutions with $b = 0$
is very similar with the only difference that the dependence of the momentum $F(\psi)$ versus speed $c$
approaches the horizontal asymptote as $c \to \infty$ since $\alpha = 1$ is the $L^2$-critical modified Benjamin--Ono equation \cite{Kalish,KMR,Martel-Pilod} and the periodic waves 
with the single-lobe profile converge to the solitary waves 
in the limit $c \to \infty$.

The stability of the family of solutions with $b\ne 0$ is however
different. The momentum $F(\psi)$ is a decreasing function of the speed $c$, as the insert shows, hence the family of solutions
is spectrally unstable for all $c > c_*$. It also approaches to the horizontal asymptote as $c \to \infty$.
Profiles of both the families in the limit of large $c$ approach the soliton profile, but the family
with $b = 0$ contains two solitons on the period, whereas the family with $b \neq 0$ contains a single
soliton on the period. Hence the momentum $F(\psi)$ of the family with $b = 0$ approaches the double horizontal asymptote
as $c \to \infty$ compared to the momentum $F(\psi)$ of the family with $b \neq 0$. 

We have checked again that $\sigma_0$ along the family with $b = 0$ changes sign from negative to positive at the bifurcation point $c = c_*$, 
whereas $\sigma_0$ along the family with $b \neq 0$ is negative for 
$c \in (c_*,\hat{c}_*)$ and positive for $c \in (\hat{c}_*,\infty)$, 
where $\hat{c}_*$ is the point where $z(\mathcal{L}) = 2$ along the family with $b \neq 0$.

\section{Even periodic waves}

Here we consider {\em the even periodic waves} and provide a proof of Theorem \ref{theorem-main-2}.

The odd periodic wave constructed in Theorem \ref{theorem-existence} and Corollary \ref{corollary-existence}
is even after translation $x \mapsto x - \pi/2$. However, since $n(\mathcal{L}) = 2$
by Lemma \ref{lemma-index}, the odd periodic wave translated into an even function
cannot be a solution of the constrained minimization problem with a single constraint. Therefore, the same constrained
minimization problem (\ref{minBfunc}) in a subspace of even functions yields a different branch
of periodic waves.

The following theorem gives the construction and properties of the even periodic waves.

\begin{theorem}
\label{theorem-existence-even}
Let $\alpha > \frac{1}{2}$ be fixed. For every $c > 0$, there exists
the ground state (minimizer) $\chi \in H^{\frac{\alpha}{2}}_{\rm per, even}$
of the following constrained minimization problem:
\begin{equation}
\label{minBfunc-even}
q_{c, {\rm even}} := \inf_{u\in H^{\frac{\alpha}{2}}_{\rm per, even}} \left\{ \mathcal{B}_c(u) : \quad \int_{-\pi}^{\pi} u^4 dx = 1 \right\},
\end{equation}
with the same $\mathcal{B}_c(u)$ as in (\ref{def-B}).
There exists $C > 0$ such that $\psi(x) = C \chi(x)$ satisfies the stationary equation (\ref{ode-stat}) with $b = 0$.
If $\alpha \leq 2$, 
the ground state is the constant solution for  $c \in \left(0,\frac{1}{2}\right]$ and has the single-lobe profile for $c \in \left(\frac{1}{2},\infty\right)$.
\end{theorem}

\begin{proof}
It follows that $\mathcal{B}_c$ is a smoot, bounded, and coercive functional on $H^{\frac{\alpha}{2}}_{\rm per, even}$ if $c > 0$, hence $q_{c, {\rm even}} \geq 0$. 
It follows from the bound (\ref{Sob-embedding}) and the constraint in (\ref{minBfunc-even}) that $q_{c, {\rm even}} > 0$. The existence of the minimizer $\chi \in H^{\frac{\alpha}{2}}_{\rm per, even}$ of the constrained minimization problem
\eqref{minBfunc-even} is proven exactly like in the proof of Theorem \ref{theorem-existence}.
Moreover, for $\alpha \in (0,2]$, the symmetric rearrangements suggest that the minimizer $\chi \in H^{\frac{\alpha}{2}}_{\rm per, even}$
is either constant or it must decrease symmetrically away from the maximum point.

In order to ensure that the minimizer has the single-lobe profile, we need to eliminate the constant solution
in $H^{\frac{\alpha}{2}}_{\rm per, even}$. By Lagrange's Multiplier Theorem, the constrained minimizer $\chi \in H^{\frac{\alpha}{2}}_{\rm per, even}$ satisfies
the stationary equation
\begin{equation}
\label{lagrange-even}
D^{\alpha}\chi +c \chi = \mu \chi^3,
\end{equation}
where $\mu = 2 \mathcal{B}_c(\chi)$ due to the normalization in (\ref{minBfunc-even}).
Since $\mathcal{B}_c(\chi) > 0$, the scaling transformation $\psi = C \chi$
with $C := \sqrt{\mathcal{B}_c(\chi)}$ maps the stationary equation (\ref{lagrange-even})
to the form (\ref{ode-stat}) with $b = 0$.
The constant nonzero solution to the stationary equation (\ref{ode-stat}) with $b = 0$  is given by
$\psi(x) = \sqrt{c/2}$ up to a sign choice. The linearization operator
$\mathcal{L}$ in (\ref{operator}) evaluated at the constant solution is given by
$$
\mathcal{L} = D^{\alpha} + c - 6 \psi^2 = D^{\alpha} - 2c.
$$
Since $n(\mathcal{L}) = 1$ if and only if $c \in \left( 0, \frac{1}{2} \right]$,
the constant wave is a constrained minimizer of
(\ref{minBfunc-even}) for $c \in \left( 0, \frac{1}{2} \right]$
and a saddle point of (\ref{minBfunc-even}) for $c \in \left( \frac{1}{2}, \infty \right)$. By the symmetric rearrangements, the global minimizer
is given by the constant solution in the former case and by a  non-constant solution with the single-lobe
profile in the latter case.
\end{proof}

\begin{remark}
Periodic waves obtained in Theorem \ref{theorem-existence-even} do not satisfy
the boundary-value problem (\ref{ode-bvp}) if the mean value of $\psi \in H^{\alpha}_{\rm per, even}$ is nonzero.  
\end{remark}

Let $\psi \in H^{\alpha}_{\rm per, even}$ be a solution to the
stationary equation (\ref{ode-stat}) with $b = 0$ for $c \in \left(\frac{1}{2},\infty\right)$
obtained by Theorem \ref{theorem-existence-even}. We introduce again the
linearized operator $\mathcal{L}$ by (\ref{operator}) and (\ref{operator-c}).
Equalities (\ref{range-1}) and (\ref{range-2}) hold true for the even periodic wave
and so do Propositions \ref{prop-nodal}, \ref{prop-kernel}, and \ref{prop-odd}
(no translation is needed for the last proposition).

The following lemma presents the count of $n(\mathcal{L})$ and $z(\mathcal{L})$ for the even periodic wave.

\begin{lemma}
\label{lemma-index-even}
Let $\alpha \in (\frac{1}{2},2]$ and $\psi \in H^{\alpha}_{\rm per, even}$ be a 
solution obtained in Theorem \ref{theorem-existence-even}. Then, $n(\mathcal{L}) = 1$ and
$$
z(\mathcal{L}) = \left\{ \begin{array}{l} 1, \quad \mbox{\rm if  } 1 \in {\rm Range}(\mathcal{L}), \\
2, \quad \mbox{\rm if  } 1 \notin {\rm Range}(\mathcal{L}). \end{array} \right.
$$
\end{lemma}

\begin{proof}
Since $\psi \in H^{\alpha}_{\rm per, even}$ is a minimizer of the constrained variational problem
(\ref{minBfunc-even}) with only one constraint, we have $n (\mathcal{L} |_{L^2_{\rm even}}) \leq 1$.
On the other hand, we have
$$
\langle \mathcal{L} \psi, \psi \rangle_{L^2} = - 4 \| \psi \|^4_{L^4} < 0,
$$
with even $\psi$, hence $n(\mathcal{L} |_{L^2_{\rm even}}) \geq 1$, so that $n(\mathcal{L} |_{L^2_{\rm even}}) = 1$.
By Proposition \ref{prop-odd} (without translation), $n(\mathcal{L} |_{L^2_{\rm odd}}) = 0$ and
$z(\mathcal{L} |_{L^2_{\rm odd}}) = 1$.
Hence, $n(\mathcal{L}) = 1$.

It remains to consider $z(\mathcal{L}) \geq 1$. Since $0$ is the second eigenvalue of $\mathcal{L}$,
Proposition \ref{prop-nodal} suggests that if $z(\mathcal{L}) = 2$, then the even eigenfunction of ${\rm Ker}(\mathcal{L})$
has at most two symmetric nodes on $\mathbb{T}$. If the periodic wave has the single-lobe profile $\psi$, then $\psi^3$ has also the single-lobe profile. By using the same argument as in the proof of Proposition 3.1 in \cite{hur},
it follows that $z(\mathcal{L}) = 1$ if and only if $\{ 1, \psi^3 \} \in {\rm Range}(\mathcal{L})$.

Indeed, if $h \in {\rm Ker}(\mathcal{L})$ is an even eigenfunction in the case $z(\mathcal{L}) = 2$
and $\{ 1, \psi^3 \} \in {\rm Range}(\mathcal{L})$, then $\langle h, 1 \rangle = 0$ and $\langle h, \psi^3 \rangle = 0$.
The first condition suggests that $h$ is sign-indefinite with exactly two symmetric nodes at $\pm x_0$ with $x_0 \in (0,\pi)$,
but then $\langle h, \psi^3 - \psi^3(x_0) \rangle$ is sign-definite and cannot be zero, so that no $h \in {\rm Ker}(\mathcal{L})$
exists.

Since $\psi^3 \in {\rm Range}(\mathcal{L})$ due to equation (\ref{range-2}), it follows that
$z(\mathcal{L}) = 1$ if and only if $1 \in {\rm Range}(\mathcal{L})$.
\end{proof}

The definition of $\mathcal{L}|_{X_0}$, where $X_0 \subset L^2(\mathbb{T})$
is defined by (\ref{zero}), is the same as in (\ref{operator-zero}). The result of Lemma \ref{lem-mean-value}
holds for the even periodic wave $\psi \in H^{\alpha}_{\rm per, even}$. In order to count
the indices $n(\mathcal{L} |_{X_0})$ and $z(\mathcal{L} |_{X_0})$, we shall re-parameterize
the even periodic wave to the zero-mean periodic waves.

Let us define $\psi(x) = a + \phi(x)$, where $a := \frac{1}{2\pi} \int_{-\pi}^{\pi} \psi(x) dx$.
Then, $\phi \in H^{\alpha}_{\rm per, even} \cap X_0$ is a solution of the stationary equation:
\begin{equation}
\label{ode-stat-phi}
D^{\alpha} \phi + \omega \phi + \beta = 2 \left( \phi^3 + 3 a \phi^2 \right),
\end{equation}
where $\omega := c - 6a^2$ and $\beta := c a - 2 a^3$. Since $\phi$ has zero mean, $\beta$ can be equivalently written as
\begin{equation}
\label{beta-expression}
\beta := \frac{1}{\pi} \int_{-\pi}^{\pi} \left( \phi^3 + 3 a \phi^2 \right) dx,
\end{equation}
so that the stationary equation (\ref{ode-stat-phi}) can be rewritten as the boundary-value problem:
\begin{equation}
\label{ode-bvp-phi}
D^{\alpha} \phi + \omega \phi = 2 \Pi_0 \left( \phi^3 + 3 a \phi^2 \right).
\end{equation}
The following lemma presents the computation of $n(\mathcal{L} |_{X_0})$ and $z(\mathcal{L} |_{X_0})$.

\begin{lemma}
\label{lemma-characterization-even}
Let $\alpha \in (\frac{1}{2},2]$ and $\psi \in H^{\alpha}_{\rm per, even}$ be a 
solution obtained in Theorem \ref{theorem-existence-even}. Assume that $\omega \in (-1,\infty)$
after the transformation to the stationary equation (\ref{ode-stat-phi}).
Then, $n(\mathcal{L} |_{X_0}) = 1$ and $z(\mathcal{L} |_{X_0}) = 1$.
\end{lemma}

\begin{proof}
Transformation $\psi = a + \phi$ and $\omega = c - 6 a^2$ changes $\mathcal{L}$ given by (\ref{operator})
into the equivalent form:
\begin{equation}
\label{operator-tilde}
\mathcal{L} = D^{\alpha} + c - 6 \psi^2 = D^{\alpha} + \omega - 6 \phi^2 - 12 a \phi =: \tilde{\mathcal{L}}.
\end{equation}
Then, it follows directly that
\begin{equation}
\label{neg-L0-X0}
\langle \mathcal{L} |_{X_0} \phi, \phi \rangle = - 4 \int_{-\pi}^{\pi} \phi^4 dx - 6 a \int_{-\pi}^{\pi} \phi^3 dx.
\end{equation}
Taking an inner product of the stationary equation (\ref{ode-stat-phi}) with $\phi$ yields the Pohozhaev-type identity
\begin{equation}
B_{\omega}(\phi) = \int_{-\pi}^{\pi} \phi^4 dx + 3 a \int_{-\pi}^{\pi} \phi^3 dx.
\end{equation}
where $B_{\omega}(\phi)$ is defined by (\ref{def-B}). Since $\omega \in (-1,\infty)$ and $\phi \in H^{\alpha}_{\rm per, even} \cap X_0$,
we have $B_{\omega}(\phi) \geq 0$, so that the equality (\ref{neg-L0-X0}) can be estimated by
\begin{equation}
\label{neg-L0-X0-inequality}
\langle \mathcal{L} |_{X_0} \phi, \phi \rangle \leq  - 2 \int_{-\pi}^{\pi} \phi^4 dx < 0.
\end{equation}
Hence $n(\mathcal{L} |_{X_0}) \geq 1$ and since $n(\mathcal{L}) = 1$ by Lemma \ref{lemma-index-even},
we have $n(\mathcal{L} |_{X_0}) = 1$. By Theorem 4.1 in \cite{Pel-book}, we have
the following identities:
\begin{equation}
\label{identneg-even}
\left\{ \begin{array}{l}
n(\mathcal{L} |_{X_0}) = n(\mathcal{L}) - n_0 - z_0, \\
z(\mathcal{L} |_{X_0}) = z(\mathcal{L}) + z_0 - z_{\infty},
\end{array} \right.
\end{equation}
where $n_0 + z_0 = 1$ if $1 \in {\rm Range}(\mathcal{L})$ and $\sigma_0 := \langle \mathcal{L}^{-1} 1, 1 \rangle \leq 0$
and $z_{\infty} = 1$ if  $1 \notin {\rm Range}(\mathcal{L})$.
It follows from the first equality in (\ref{identneg-even}) 
that $n_0 = z_0 = 0$ since $n(\mathcal{L}) = n(\mathcal{L} |_{X_0}) = 1$.
Then, the second equality yields $z(\mathcal{L} |_{X_0}) = z(\mathcal{L}) - z_{\infty}$.
If $z(\mathcal{L} |_{X_0}) = 2$, then $z(\mathcal{L}) \geq 2$, which is in contradiction with
Lemma \ref{lem-mean-value} extended to the even periodic wave $\psi \in H^{\alpha}_{\rm per, even}$.
Hence, $z(\mathcal{L} |_{X_0}) = 1$, in which case $z(\mathcal{L}) = 1 + z_{\infty}$ in agreement with
Lemma \ref{lemma-index-even}.
\end{proof}

\begin{remark}
\label{remark-sigma-0}
It follows from the proof of Lemma \ref{lemma-characterization-even} that $\sigma_0 > 0$ if $1 \in {\rm Range}(\mathcal{L})$,
where $\sigma_0 := \langle \mathcal{L}^{-1} 1, 1 \rangle$.
\end{remark}

In order to derive the spectral stability result, we shall now extend solutions
to the stationary equation (\ref{ode-stat-phi}) with respect to two independent parameters
$(\omega,a)$ with $\beta$ being a $C^1$ function of $(\omega,a)$. Since
the periodic waves satisfy the stationary equation (\ref{ode-stat}) with $b = 0$, where $c$ is the only parameter, parameters $\omega$, $a$, and $\beta$ in the stationary equation (\ref{ode-stat-phi}) 
are parametrized by $c$, hence $a$ is not independent of $\omega$. 
The following lemma allows us to extend zero-mean solutions to the boundary-value
problem (\ref{ode-bvp-phi}) with respect to independent parameters $(\omega,a)$ near each
uniquely defined point $(\omega_0,a_0)$.

\begin{lemma}
\label{lem-continuation-even}
Assume $\alpha \in (\frac{1}{2},2]$ and $\phi_0 \in H^{\alpha}_{\rm per, even} \cap X_0$ be a solution to the boundary-value problem (\ref{ode-bvp-phi}) with $\omega = \omega_0 \in (-1,\infty)$
and $a = a_0 \in \mathbb{R}$. Then, there exists a $C^1$ mapping in an open subset of
$(\omega_0,a_0)$ denoted by $\mathcal{O} \subset \mathbb{R}^2$:
\begin{equation}
\label{map-in-omega-a}
\mathcal{O} \ni (\omega,a) \mapsto \phi(\cdot;\omega,a) \in H_{\rm per, even}^{\alpha} \cap X_0,
\end{equation}
such that $\phi(\cdot;\omega_0,a_0) = \phi_0$.
\end{lemma}

\begin{proof}
The proof repeats the arguments in the proof of Lemma \ref{prop-continuation}.
Let $\Upsilon:(-1,\infty)\times \mathbb{R} \times H_{\rm per,even}^{\alpha} \cap X_0 \to L^2_{\rm even}(\mathbb{T}) \cap X_0$ be defined by
\begin{equation}\label{Ups}
\Upsilon(\omega,a,g) := D^{\alpha} g + \omega g - 2 \Pi_0 \left( g^3 + 3 a g^2 \right).
	\end{equation}
By hypothesis we have $\Upsilon(\omega_0,a_0,\phi_0)=0$. Moreover, since $\Upsilon$ is smooth, its Fr\'echet derivative with respect to $g$ evaluated at $(\omega_0,a_0,\phi_0)$ is given by
\begin{equation}\label{operaIFT}
D_g \Upsilon (\omega_0,a_0,\phi_0) = D^{\alpha} + \omega_0 - 6 \Pi_0 \left( \phi_0^2 + 2 a_0 \phi_0 \right)
= D^{\alpha} + c_0 - 6 \Pi_0 \psi_0^2 = \mathcal{L}|_{X_0},
\end{equation}
where we have unfolded the previous transformation $\psi_0 = a_0 + \phi_0$ and $\omega_0 = c_0 - 6 a_0^2$
and used the same operator as in (\ref{operator-zero}) computed at $\psi_0$.

Since $z(\mathcal{L} |_{X_0}) = 1$ by Lemma \ref{lemma-characterization-even}
and ${\rm Ker}(\mathcal{L} |_{X_0}) = {\rm span}\{ \partial_x \phi_0 \}$ with $\partial_x \phi_0 \notin
H_{\rm per,even}^{\alpha} \cap X_0$, we conclude that $D_g \Upsilon (\omega_0,a_0,\phi_0)$ is one-to-one.
Next, we show that $D_g \Upsilon (\omega_0,a_0,\phi_0)$ is onto. Since $H_{\rm per,even}^{\alpha} \cap X_0$
is compactly embedded in $L^2_{\rm even}(\mathbb{T})\cap X_0$ if $\alpha > 1/2$, the operator $\mathcal{L}|_{X_0}$
has compact resolvent. In addition, $\mathcal{L}|_{X_0}$ is a self-adjoint operator,
hence its spectrum $\sigma(\mathcal{L}|_{X_0})$ consists of isolated eigenvalues with finite algebraic multiplicities.
Since $D_g \Upsilon (\omega_0,a_0,\phi_0)$ is one-to-one, it follows that $0$ is not in the spectrum
of $D_g \Upsilon (\omega_0,a_0,\phi_0)$, so that it is onto. Hence, $D_g \Upsilon (\omega_0,a_0,\phi_0)$ is a bounded linear operator
with a bounded inverse. Thus, since $\Upsilon$ and its derivative with respect to $g$  are smooth maps on their domains,
the result follows from the implicit function theorem.
\end{proof}

Recall that $\mathcal{L} = \tilde{\mathcal{L}}$ in (\ref{operator-tilde}).
Extension of relations (\ref{range-1}) and (\ref{range-2}) yields
\begin{eqnarray}
\label{range-1-tilde}
\tilde{\mathcal{L}} 1 = \omega - 12 a \phi - 6 \phi^2
\end{eqnarray}
and
\begin{eqnarray}
\label{range-2-tilde}
\tilde{\mathcal{L}} \phi = - \beta - 6a \phi^2 - 4 \phi^3,
\end{eqnarray}
where $\beta = \beta(\omega,a)$ is a $C^1$ function by Lemma \ref{lem-continuation-even}
and the representation (\ref{beta-expression}). Therefore, we also obtain
two more relations:
\begin{eqnarray}
\label{range-3-tilde}
\tilde{\mathcal{L}} \partial_{\omega} \phi = - \partial_{\omega} \beta - \phi,
\end{eqnarray}
and
\begin{eqnarray}
\label{range-4-tilde}
\tilde{\mathcal{L}} \partial_a \phi = - \partial_a \beta + 6 \phi^2.
\end{eqnarray}
These relations allow us to completely characterize ${\rm Ker}(\mathcal{L})$,
which can be two-dimensional if $1 \notin {\rm Range}(\mathcal{L})$ by Lemma \ref{lemma-index-even}.

\begin{lemma}
\label{lemma-kernel}
Assume $\alpha \in (\frac{1}{2},2]$ and $\phi \in H^{\alpha}_{\rm per, even} \cap X_0$ be a single-lobe
solution to the boundary-value problem (\ref{ode-bvp-phi}) with $\omega \in (-1,\infty)$
and $a \in \mathbb{R}$. Then, $z(\mathcal{L}) = 1$ if and only if
$s_0 := \omega - \partial_a \beta + 12 a \partial_{\omega} \beta \neq 0$.
\end{lemma}

\begin{proof}
Eliminating $\phi$ and $\phi^2$ from (\ref{range-1-tilde}), (\ref{range-3-tilde}), and (\ref{range-4-tilde}) yields
\begin{equation}
\label{range-kernel}
\tilde{\mathcal{L}} \left( 1 + \partial_a \phi - 12 a \partial_{\omega} \phi \right) =
\omega - \partial_a \beta + 12 a \partial_{\omega} \beta =: s_0.
\end{equation}
Recall that $\tilde{\mathcal{L}} = \mathcal{L}$ in (\ref{operator-tilde}).
If $s_0 \neq 0$, then $1 \in {\rm Range}(\mathcal{L})$, so that $z(\mathcal{L}) = 1$ holds by Lemma \ref{lemma-index-even}.
If $s_0 = 0$, then $1 + \partial_a \phi - 12 a \partial_{\omega} \phi \in {\rm Ker}(\mathcal{L})$
in addition to $\partial_x \phi \in {\rm Ker}(\mathcal{L})$.
\end{proof}

We are now ready to provide the criterion for spectral stability of the even periodic waves. This result is
given by the following theorem.

\begin{theorem}
	\label{theorem-stability-even}
Assume $\alpha \in (\frac{1}{2},2]$ and $\phi \in H^{\alpha}_{\rm per, even} \cap X_0$ be a single-lobe
solution to the boundary-value problem (\ref{ode-bvp-phi}) with $\omega \in (-1,\infty)$
and $a \in \mathbb{R}$. The periodic wave is spectrally stable if and only if
\begin{equation}
\frac{\partial}{\partial \omega} \| \phi \|_{L^2}^2 \geq 0,
\label{stability-constraints-even}
\end{equation}
independently of either $z(\mathcal{L}) = 1$ or $z(\mathcal{L}) = 2$.
\end{theorem}

\begin{proof}
We proceed similarly to the proof of Theorem \ref{theorem-stability}.
If $1 \in {\rm Range}(\mathcal{L})$, we use (\ref{range-kernel}) and compute
$$
\sigma_0 := \langle \mathcal{L}^{-1} 1, 1 \rangle = \frac{2\pi}{s_0},
$$
where $s_0 \neq 0$ by Lemma \ref{lemma-kernel}. For the even periodic wave, we have $\sigma_0 > 0$ (see Remark \ref{remark-sigma-0}),
so that $s_0 > 0$. Eliminating constant term from (\ref{range-3-tilde}) and (\ref{range-kernel}) yields
\begin{equation}
\label{range-phi}
\tilde{\mathcal{L}} \left[ \partial_{\omega} \phi + s_0^{-1} \partial_{\omega} \beta
\left( 1 + \partial_a \phi - 12 a \partial_{\omega} \phi \right)
\right] = - \phi,
\end{equation}
By projecting (\ref{range-3-tilde}) to $\partial_a \phi$ and (\ref{range-4-tilde}) to $\partial_{\omega} \phi$,
it is easy to verify that
\begin{equation}
\label{tech-equality-1}
6 \langle \phi^2, \partial_{\omega} \phi \rangle + \langle \phi, \partial_a \phi \rangle = 0.
\end{equation}
Using (\ref{range-phi}) yields
\begin{eqnarray*}
\langle \mathcal{L}^{-1} 1, \phi \rangle = \langle \mathcal{L}^{-1} \phi, 1 \rangle = - \sigma_0 \partial_{\omega} \beta,
\end{eqnarray*}
where we have used that
$$
2 \pi \partial_{\omega} \beta = 6 \langle \phi^2, \partial_{\omega} \phi \rangle + 12 a \langle \phi, \partial_{\omega} \phi \rangle
= - \langle \phi, \partial_a \phi \rangle + 12 a \langle \phi, \partial_{\omega} \phi \rangle,
$$
which follows from (\ref{beta-expression}) and (\ref{tech-equality-1}). Finally, we obtain from (\ref{range-phi}) and (\ref{tech-equality-1})
that
\begin{eqnarray*}
\langle \mathcal{L}^{-1} \phi, \phi \rangle = - \langle \phi, \partial_{\omega} \phi \rangle + \sigma_0 \left( \partial_{\omega} \beta \right)^2.
\end{eqnarray*}
By Theorem 4.1 in \cite{Pel-book}, we have the following identities:
	\begin{equation}\label{identnegLL-even}
	\left\{ \begin{array}{l}
	n(\mathcal{L} \big|_{\{1,\psi\}^{\bot}}) = n(\mathcal{L}) - n_0 - z_0, \\
	z(\mathcal{L} \big|_{\{1,\psi\}^{\bot}}) = z(\mathcal{L}) + z_0,
	\end{array} \right.
	\end{equation}
where $n_0$ and $z_0$ are the numbers of negative and zero eigenvalues of $D(0)$
in the proof of Theorem \ref{theorem-stability}. Since
$$
{\rm det} D(0) = -\sigma_0 \langle \phi, \partial_{\omega} \phi \rangle.
$$
and $\sigma_0 > 0$, we have $n_0 + z_0 = 1$ if the condition (\ref{stability-constraints-even}) is satisfied
and $n_0 + z_0 = 0$ if it is not satisfied.
Since $n(\mathcal{L}) = 1$ and $z(\mathcal{L}) = 1$ by Lemmas \ref{lemma-index-even} and \ref{lemma-kernel},
the count (\ref{identnegLL-even}) implies $n(\mathcal{L} \big|_{\{1,\psi\}^{\bot}}) = 0$
if the condition (\ref{stability-constraints-even}) is satisfied and
$n(\mathcal{L} \big|_{\{1,\psi\}^{\bot}}) = 1$ if it is not satisfied.
This gives the assertion of the theorem if  $1 \in {\rm Range}(\mathcal{L})$.

If $1 \notin {\rm Range}(\mathcal{L})$, then $z(\mathcal{L}) = 2$ and $s_0 = 0$.
In this case, the count (\ref{identnegLL-even}) should be adjusted as
	\begin{equation}\label{identnegLL-even-singular}
	\left\{ \begin{array}{l}
	n(\mathcal{L} \big|_{\{1,\psi\}^{\bot}}) = n(\mathcal{L}) - n_0 - z_0, \\
	z(\mathcal{L} \big|_{\{1,\psi\}^{\bot}}) = z(\mathcal{L}) + z_0 - z_{\infty},
	\end{array} \right.
	\end{equation}
where $z_{\infty} = 1$. At the same time, $n_0 + z_0 = 1$ if and only if the same
condition (\ref{stability-constraints-even}) is satisfied and $n_0 + z_0 = 0$ if it is not satisfied.
Hence the stability conclusion remains unchanged if $1 \notin {\rm Range}(\mathcal{L})$.
\end{proof}

\begin{remark}
	\label{remark-slope}
The momentum (\ref{Fu}) computed at the even periodic wave with the profile $\psi$ and the decomposition $\psi = a + \phi$ is given by
$$
F(\psi) = F(\phi) + \pi a^2.
$$
If $\omega$ and $a$ are independent parameters, it is true that
\begin{equation}
\label{slope-cond}
\frac{\partial}{\partial \omega} F(\phi) = \frac{\partial}{\partial \omega} F(\psi),
\end{equation}
however, this spectral stability criterion is not derived from the dependence
of the momentum $F(\psi)$ from the original wave speed $c$. In addition, 
if $\psi$ satisfies the stationary equation (\ref{ode-stat}) with $b = 0$, 
then $a$ depends on $\omega$, therefore, the dependence of $F(\psi)$ versus $\omega$ 
does not generally provide information about the slope condition (\ref{slope-cond}).
\end{remark}

\section{Examples of even periodic waves}

Here we present three examples of the even periodic waves obtained in Theorem \ref{theorem-existence-even}. These examples mimic the corresponding examples of the odd periodic waves. 

\subsection{Stokes expansion of small-amplitude waves}

Stokes expansion gives again a direct way to illustrate 
small-amplitude periodic waves bifurcating from the constant solutions at  $c = \frac{1}{2}$. In order to eliminate the constant wave, we set
\begin{equation}
\label{Stokes-expansion}
\psi(x) = \frac{\sqrt{c}}{\sqrt{2}} + \varphi(x),
\end{equation}
where $\varphi$ is not required to satisfy the zero-mean property.
The stationary equation (\ref{ode-stat}) with $b = 0$ is written in the equivalent form:
\begin{equation}
\label{ode-stat-stokes}
D^{\alpha} \varphi - 2c \varphi = 2 \varphi^3 + 3  \sqrt{2c} \varphi^2.
\end{equation}
By using the Stokes expansion in terms of small amplitude $A$:
\begin{equation}
\label{Stokes-even}
\varphi(x) = A \varphi_1(x) + A^2 \varphi_2(x) + A^3 \varphi_3(x) + \mathcal{O}(A^4), \quad
2c = 1 + A^2 \gamma_2 + \mathcal{O}(A^4),
\end{equation}
we obtain recursively: $\varphi_1(x) = \cos(x)$,
$$
\varphi_2(x) = -\frac{3}{2} + \frac{3}{2(2^{\alpha}-1)} \cos(2x),
$$
$$
\varphi_3(x) = \frac{1}{2(3^{\alpha} - 1)} \left[ 1 + \frac{9}{2^{\alpha}-1} \right] \cos(3x),
$$
and
$$
\gamma_2 = \frac{15}{2} - \frac{9}{2(2^{\alpha}-1)}.
$$
It follows that $\gamma_2 = 0$ if and only if $2^{\alpha} = \frac{8}{5}$, which is true at
\begin{equation}
\label{alpha-0}
\alpha_0 := \frac{\log 8 - \log 5}{\log 2} \approx 0.6781.
\end{equation}
The following proposition summarizes properties of the small-amplitude periodic waves.

\begin{proposition}
Let $\alpha_0$ be given by (\ref{alpha-0}). For each $\alpha \in \left(\alpha_0,2\right]$,
there exists $c_0 > \frac{1}{2}$ such that the even periodic wave exists for $c \in \left(\frac{1}{2},c_0\right)$
with $n(\mathcal{L}) = 1$, $z(\mathcal{L}) = 1$ and is spectrally stable.
For each $\alpha \in \left(\frac{1}{2},\alpha_0 \right)$, there exists $c_0 < \frac{1}{2}$
such that the even periodic wave exists for $c \in \left(c_0,\frac{1}{2}\right)$
with $n(\mathcal{L}) = 2$, $z(\mathcal{L}) = 1$, and is spectrally stable.
\label{prop-Stokes-even}
\end{proposition}

\begin{proof}
The existence statement follows from the Stokes expansion (\ref{Stokes-even}) with small
wave amplitude $A$ since $\gamma_2 > 0$ for $\alpha > \alpha_0$ and $\gamma_2 < 0$ for $\alpha < \alpha_0$.

In order to compute $n(\mathcal{L})$ and $z(\mathcal{L})$, we substitute
(\ref{Stokes-expansion}) and (\ref{Stokes-even}) into (\ref{operator}) and obtain
\begin{eqnarray*}
\mathcal{L} = D^{\alpha} - 1 - A \cos(x) - A^2 \left[ \gamma_2 + 6 \varphi_2(x) - 6 \cos^2(x) \right] + \mathcal{O}(A^3).
\end{eqnarray*}
We solve the spectral problem $\mathcal{L} v = \lambda v$ perturbatively near the eigenvalue $\lambda = 0$
associated with the subspace of even functions in $L^2(\mathbb{T})$. Hence, we expand
$$
u = \cos(x) + A u_1(x) + A^2 u_2(x) + \mathcal{O}(A^3), \quad \lambda = A^2 \lambda_2 + \mathcal{O}(A^4),
$$
and obtain recursively: $u_1(x) = 2 \varphi_2(x)$ and $\lambda_2 = 2 \gamma_2$. Hence, $\lambda > 0$ if $\gamma_2 > 0$
and $\lambda < 0$ if $\gamma_2 < 0$. The zero eigenvalue associated with the subspace of odd functions
in $L^2(\mathbb{T})$ is preserved at zero for every $A$ due to $\partial_x \psi \in {\rm Ker}(\mathcal{L})$.
In addition, there exists a negative eigenvalue of $\mathcal{L}$ associated with the constant functions at $A = 0$.
Hence, we confirm that $n(\mathcal{L}) = 1$ for $\alpha > \alpha_0$ and $n(\mathcal{L}) = 2$
for $\alpha < \alpha_0$, whereas $z(\mathcal{L}) = 1$ for both $\alpha > \alpha_0$ and $\alpha < \alpha_0$.

In order to deduce the spectral stability conclusion, we use transformation $\psi(x) = a + \phi(x)$, where the zero-mean function $\phi$ satisfies the boundary-value problem (\ref{ode-bvp-phi}). Computing the mean value
$$
a := \frac{1}{2\pi} \int_{-\infty}^{\infty} \psi(x) dx = \frac{1}{2} + \frac{3}{8} \left[ 1 - \frac{3}{2^{\alpha}-1} \right] A^2 + \mathcal{O}(A^4)
$$
we obtain
$$
\omega := c - 6 a^2 = -1 + \frac{3}{2} \left[ 1 + \frac{3}{2^{\alpha}-1} \right] A^2 + \mathcal{O}(A^2)
$$
and
$$
\beta := c a - 2 a^3 = \frac{3}{2} A^2 + \mathcal{O}(A^4).
$$
No fold point occurs in the expansion of $\omega$ with respect to the Stokes amplitude $A$, in particular,
$$
\frac{d \omega}{d A^2} = \frac{3}{2} \left[ 1 + \frac{3}{2^{\alpha}-1} \right] + \mathcal{O}(A^2) > 0.
$$
Since $\| \phi \|_{L^2}^2 = \pi A^2 + \mathcal{O}(A^4)$, we have
\begin{equation}
\label{der-small-amplitude}
\frac{d}{d \omega} \| \phi \|^2_{L^2} = \frac{2\pi}{3} 
\frac{2^{\alpha} - 1}{2^{\alpha} + 2}  + \mathcal{O}(A^2) > 0.
\end{equation}
In Appendix A, we show that $\sigma_0 > 0$ for $\alpha > \alpha_0$ and $\sigma_0 < 0$ for $\alpha < \alpha_0$. In addition, we show that
$$ 
\frac{\partial}{\partial \omega} \| \phi \|^2_{L^2} = \frac{d}{d \omega} \| \phi \|^2_{L^2}  + \mathcal{O}(A^2) > 0.
$$
By Theorem \ref{theorem-stability-even}, the periodic waves are spectrally stable for small $A$ both for $\alpha > \alpha_0$ and
$\alpha < \alpha_0$.
\end{proof}

\begin{remark}
For $\alpha > \alpha_0$, the small-amplitude periodic wave in Proposition \ref{prop-Stokes-even} corresponds to the even periodic wave
in Theorem \ref{theorem-existence-even} with $n(\mathcal{L}) = 1$, $z(\mathcal{L}) = 1$, and $\sigma_0 > 0$.
However, for $\alpha < \alpha_0$, the small-amplitude periodic wave is not a minimizer of the constrained
variational problem (\ref{minBfunc-even}) in Theorem \ref{theorem-existence-even} because $n(\mathcal{L}) = 2$ and $\sigma_0 < 0$.
Nevertheless, spectral stability of the periodic wave with $n(\mathcal{L}) = 2$, $\sigma_0 < 0$, and the slope condition
(\ref{stability-constraints-even}) follows from the same 
computation as in the proof of Theorem \ref{theorem-stability-even}.
\end{remark}

\subsection{Local case with $\alpha = 2$}

In the case of the modified KdV equation ($\alpha = 2$), 
the stationary equation (\ref{ode-stat}) with $b = 0$ can be solved
in the space of even functions by using the Jacobian dnoidal function 
\cite{angulo1,DK}.
Let us recall the normalized solution $\psi_0(z) = {\rm dn}(z;k)$ of the
second-order differential equation
\begin{equation}
\label{normalized-dn}
\psi_0''(z) + (k^2 - 2) \psi_0(z) + 2 \psi_0(z)^3 = 0.
\end{equation}
Adopting an elementary scaling transformation yields the exact solution in the form:
\begin{equation}
\label{dnoidal}
\psi(x) = \frac{1}{\pi} K(k) {\rm dn}\left[\frac{1}{\pi} K(k) x; k\right]
\end{equation}
with
\begin{equation}
\label{dnoidal-speed}
c = \frac{1}{\pi^2} K(k)^2 (2 - k^2),
\end{equation}
where $K(k)$ is the complete elliptic integral of the first kind.
The following proposition summarizes properties of the even periodic waves for $\alpha = 2$.

\begin{proposition}
Fix $\alpha = 2$. The even periodic wave (\ref{dnoidal}) 
exists and is spectrally stable for every $c \in \left(\frac{1}{2},\infty \right)$.
Moreover, $n(\mathcal{L}) = 1$, $z(\mathcal{L}) = 1$, and $\sigma_0 > 0$ for every $c \in \left(\frac{1}{2},\infty \right)$.
\label{prop-dnoidal}
\end{proposition}

\begin{proof}
The mapping $(0,1) \ni k \mapsto c(k) \in \left(\frac{1}{2},\infty \right)$ is one-to-one and onto.
This follows from
\begin{equation}
\label{monotonicity-dn}
\frac{\pi^2}{2} \frac{dc}{dk} = \frac{K(k)}{k(1-k^2)} \left[ (2-k^2) E(k) - 2 (1-k^2) K(k) \right] > 0,
\end{equation}
where the latter inequality was proved in \cite{angulo1} (see also \cite{DK}). Indeed,
if
$$
f(k) := (2-k^2) E(k) - 2 (1-k^2) K(k),
$$
then $f(0) = 0$, whereas $f'(k) = 3k [K(k) - E(k)] > 0$ so that $f(k) > 0$ for $k \in (0,1)$.

The mean value of the periodic wave in (\ref{dnoidal}) is computed explicitly by
$$
a := \frac{1}{2\pi} \int_{-\pi}^{\pi} \psi(x) dx =
\frac{1}{\pi} \int_0^{K(k)} {\rm dn}(z;k) dz = \frac{1}{2}.
$$
Hence, the zero-mean function $\phi(x) := \psi(x) - a$ is
a solution to the boundary-value problem (\ref{ode-bvp-phi}) with
$$
\omega = c - \frac{3}{2}, \quad \beta = \frac{1}{2} \left( c - \frac{1}{2} \right).
$$
This gives the straight line dependence $\beta = \frac{1}{2} (\omega + 1)$ for the
periodic waves with the single-lobe profile. Furthermore, we can compute
$$
\| \phi \|_{L^2}^2 = \frac{2}{\pi} K(k) E(k) - \frac{\pi}{2},
$$
from which we verify that
$$
\langle \phi, \partial_k \phi \rangle = -\frac{1}{\pi k (1-k^2)} \left[ K(k)^2 (1-k^2) - E(k)^2 \right] > 0.
$$
The latter inequality is also proven directly by setting
$$
f(k) := K(k)^2 (1-k^2) - E(k)^2
$$
such that $f(0) = 0$ and $f'(k) = -2 k^{-1} [K(k)-E(k)]^2 < 0$ so that $f(k) < 0$ for $k \in (0,1)$.
By Theorem \ref{theorem-stability-even}, the even periodic wave (\ref{dnoidal}) with the speed (\ref{dnoidal-speed}) satisfying (\ref{monotonicity-dn}) is spectrally stable
for $c \in \left(\frac{1}{2},\infty\right)$. 

Other properties such as $1 \in {\rm Range}(\mathcal{L})$, $\sigma_0 > 0$, and $n(\mathcal{L}) = 1$ for every $k \in (0,1)$
can be confirmed by explicit computations. The normalized linearized operator is given by
\begin{equation}
\label{operator-dn}
\mathcal{L}_0 = -\partial_z^2 - 4 + 5 k^2 - 6 k^2 {\rm cn}(z;k)^2.
\end{equation}
Eigenvalues of $\mathcal{L}_0$ in (\ref{operator-dn}) are given by subtracting $3(1-k^2)$ from
eigenvalues of $\mathcal{L}_0$ in (\ref{operator-cn}). However, $\mathcal{L}_0$ in (\ref{operator-dn}) is
considered in space $L^2(-K(k),K(k))$ so that the eigenvalues $\lambda_1$ and $\lambda_2$ given below (\ref{operator-dn}) 
with the eigenfunctions in $L^2(-2K(k),2K(k))$ are not relevant. Hence, the first three
eigenvalues of $\mathcal{L}_0$ in (\ref{operator-dn}) are given by
\begin{eqnarray*}
& \lambda_0 = - 2 + k^2 - 2 \sqrt{1-k^2+k^4}, \quad & \varphi_0(z) = 1 + k^2 + \sqrt{1-k^2+k^4} - 3 k^2 {\rm sn}(z;k)^2,\\
& \lambda_1 = 0, \quad & \varphi_1(z) = {\rm sn}(z;k) {\rm cn}(z;k), \\
& \lambda_2 = -2 + k^2 + 2 \sqrt{1-k^2+k^4}, \quad & \varphi_2(z) = 1 + k^2 - \sqrt{1-k^2+k^4} - 3 k^2 {\rm sn}(z;k)^2.
\end{eqnarray*}
Eigenvalues and eigenvectors of the linearized operator $\mathcal{L}$ are obtained after
the same scaling transformation as in (\ref{dnoidal}). In agreement with Lemma \ref{lemma-index-even},
we have $n(\mathcal{L}) = 1$ and $z(\mathcal{L}) = 1$. The property $1 \in {\rm Range}(\mathcal{L})$
follows from the representation
$$
\frac{1}{2\sqrt{1-k^2+k^4}} \mathcal{L}_0 \left[ \frac{\lambda_2 \varphi_0 - \lambda_0 \varphi_2}{\lambda_0 \lambda_2} \right] = 1
\quad \mbox{\rm and} \quad \frac{1}{2\sqrt{1-k^2+k^4}} \left[ \varphi_0 - \varphi_2 \right] = 1.
$$
Direct computations yield
\begin{eqnarray*}
\langle \mathcal{L}_0^{-1} 1, 1 \rangle = \frac{\lambda_2 \langle \varphi_0, 1 \rangle
- \lambda_0 \langle \varphi_2, 1 \rangle}{2 \sqrt{1 - k^2 + k^4} \lambda_0 \lambda_2} = \frac{2}{k^4} \left[ (2-k^2) K(k) -2 E(k)\right] > 0,
\end{eqnarray*}
where the latter inequality is justified by assigning
$$
f(k) := (2-k^2) K(k) - 2E(k)
$$
with $f(0) = 0$ and $f'(k) = k (1-k^2)^{-1} [E(k) - (1-k^2) K(k)] > 0$ so that $f(k) > 0$ for $k \in (0,1)$.
\end{proof}

\begin{remark}
Explicit computations in the proof of Proposition \ref{prop-dnoidal} repeat computations in \cite{DK},
however, the expression for $\langle \mathcal{L}_0^{-1} 1, 1 \rangle$ was typed incorrectly in \cite{DK}. This stability conclusion agrees with the results in \cite{angulo1,DK}.
\end{remark}

\subsection{Numerical approximations}

Here we numerically compute solutions of the stationary equation
(\ref{ode-stat}) with $b = 0$ using Newton's method in the Fourier space.
The starting iteration is generated from the Stokes
expansion (\ref{Stokes-even}) and this solution
is uniquely continued in $c$ for all $c \in \left(\frac{1}{2},\infty \right)$ if $\alpha > \alpha_0$. 

Figure \ref{fig:alpha2} presents the periodic wave solutions for $\alpha = 2$. The top panel shows the profiles of $\psi$ for three different values of $c$. 
The bottom panels show the dependence of $F(\psi)$ versus $c$ (left)
and the dependence of $F(\phi)$ versus $\omega$ (right), where $\phi$ and $\omega$ was computed from the transformation $\phi(x) = \psi(x) - a$ 
and $\omega = c - 6 a^2$ with $a := \frac{1}{2\pi} \int_{-\pi}^{\pi} \psi(x) dx$. The even periodic wave 
with the single-lobe profile $\psi$ (red line) bifurcates at $c = \frac{1}{2}$ 
from the constant wave (grey line) shown on the bottom left. 
Since $F(\phi)$ is increasing in $\omega$ and $a = \frac{1}{2}$ is independent of $\omega$, the even periodic wave 
is stable by Theorem \ref{theorem-stability-even}. 

\begin{figure}[htpb]
	\centering
	\includegraphics[width=0.55\linewidth]{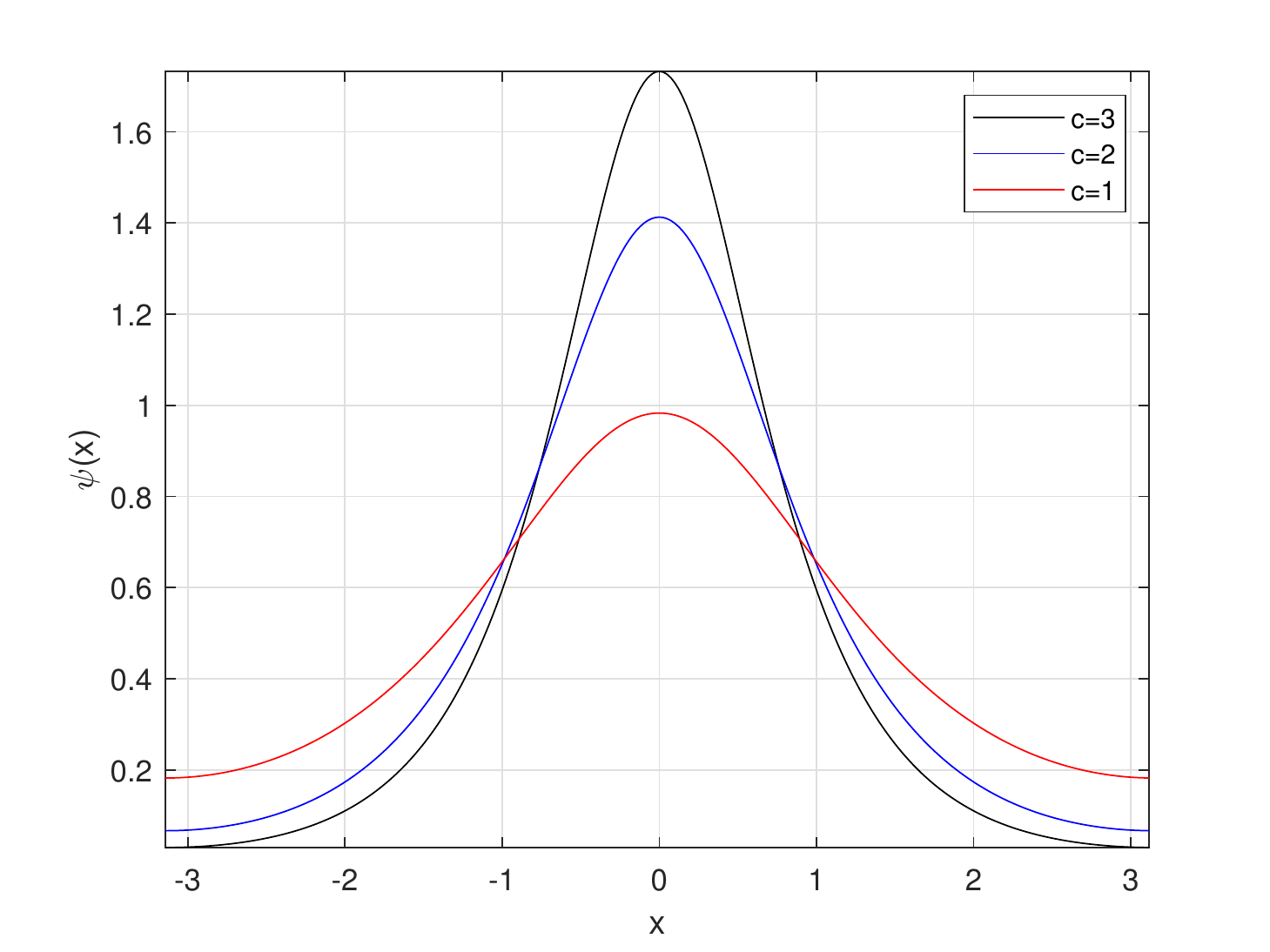} \\ 
	\includegraphics[width=0.45\linewidth]{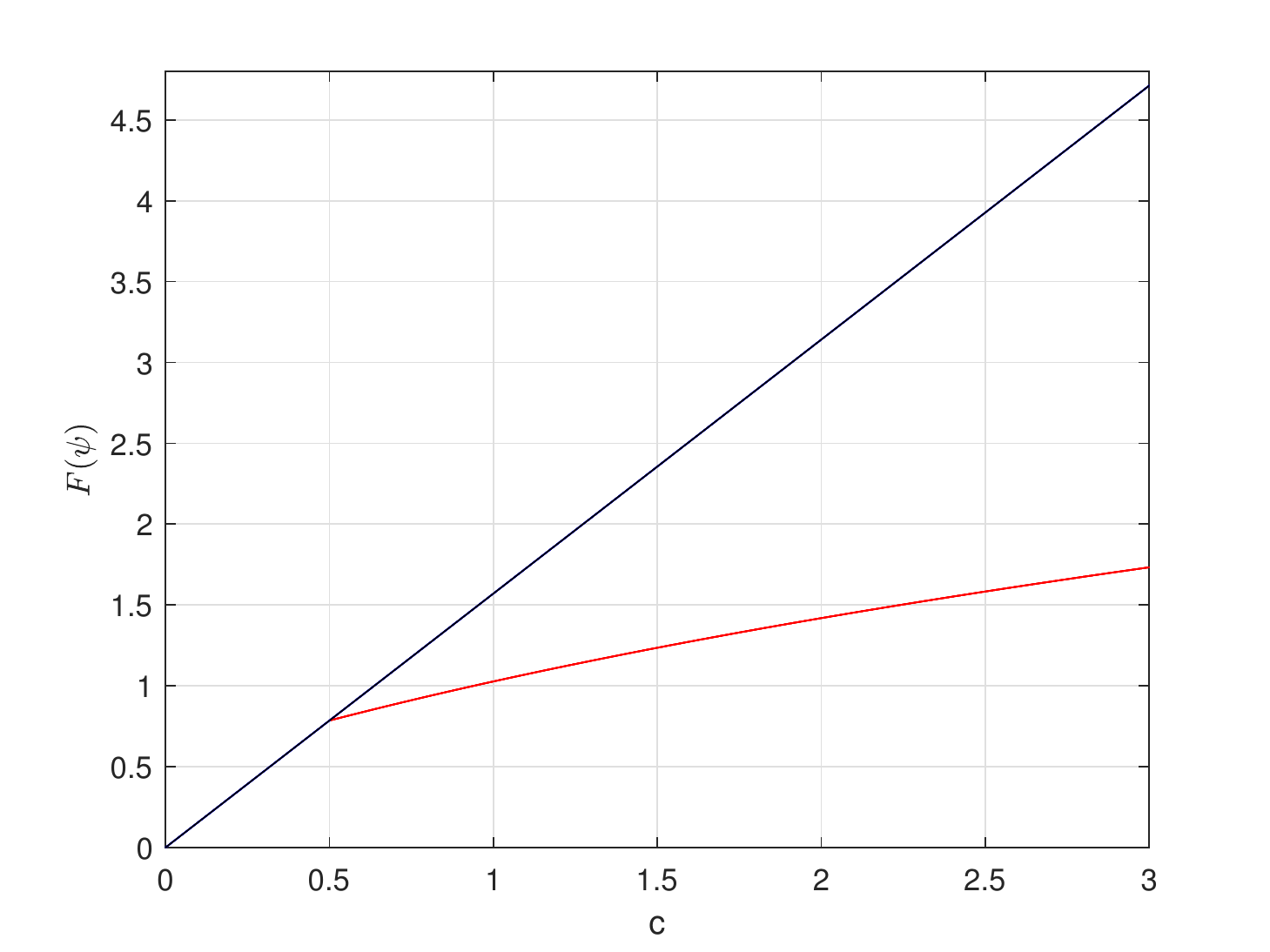}
	\includegraphics[width=0.45\linewidth]{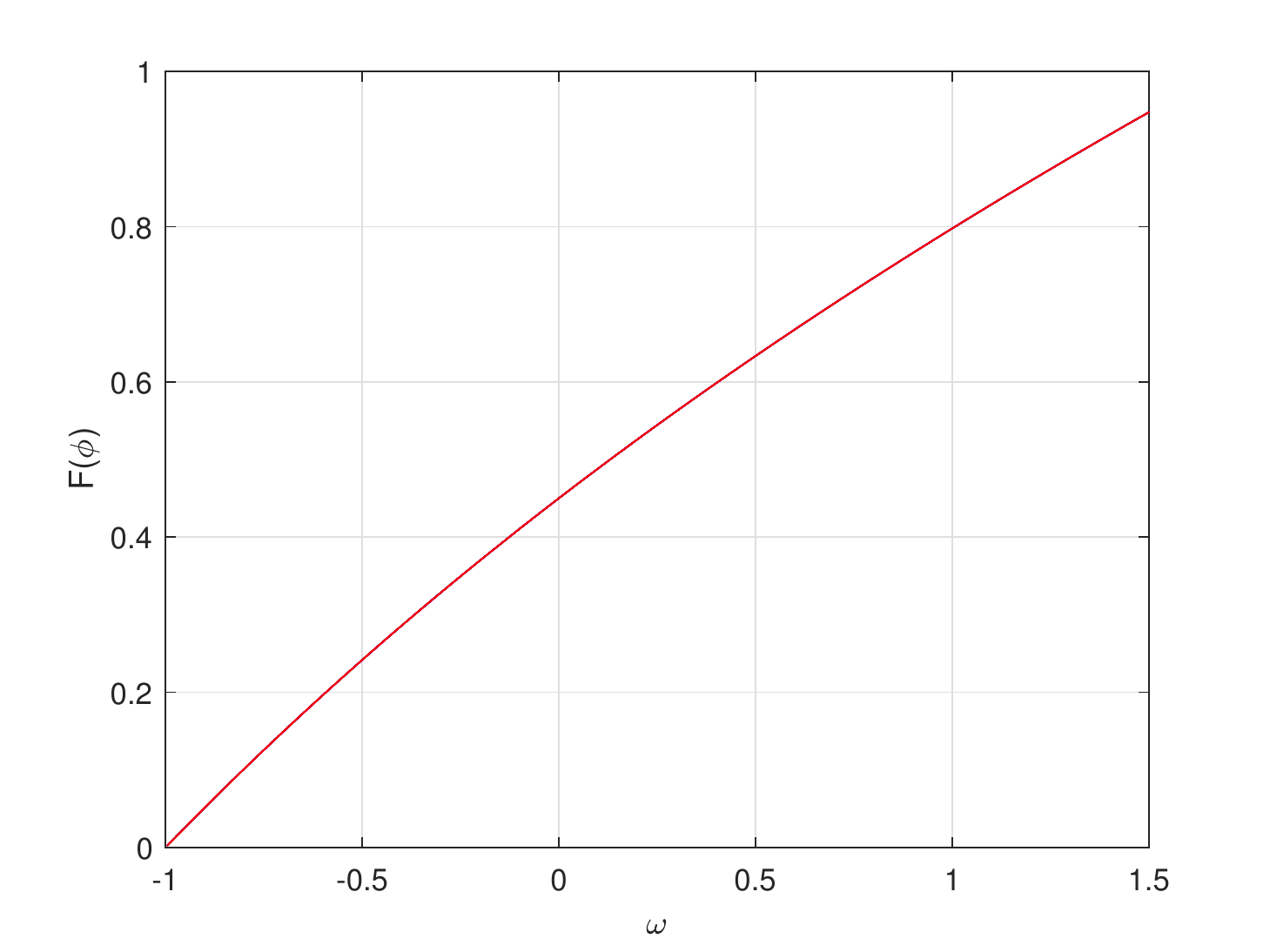}
	\caption{Periodic waves for $\alpha = 2$. Top: Profiles of $\psi$ for three different values of $c$. Bottom: Dependence of the momentum $F(\psi)$ versus $c$ (left) and $F(\phi)$ versus $\omega$ (right). }
	\label{fig:alpha2}
\end{figure} 


Figure \ref{fig:alpha1} presents similar results but for $\alpha = 1$. 
The periodic wave  (red line on the top right panel)
still bifurcates from the constant wave (grey line on the top right panel) to the right of the bifurcation point at $c = \frac{1}{2}$.
However, $a$ depends on $\omega$ for the even periodic wave, hence 
$$
\frac{d}{d \omega} F(\phi) = \frac{\partial}{\partial \omega} F(\phi) + \frac{da}{d\omega} \frac{\partial}{\partial a} F(\phi)
$$
by the chain rule. In the Stokes limit, we have shown in Appendix A 
that $\frac{\partial}{\partial a} F(\phi) = \mathcal{O}(A^2)$ for small $A$
so that 
$$
 \frac{\partial}{\partial \omega} F(\phi) = \frac{d}{d \omega} F(\phi) 
 + \mathcal{O}(A^2) > 0.
$$
However, a discrepancy between partial and ordinary derivatives of $F(\phi)$ in $\omega$ exists away from the Stokes limit. The additional bottom right panel on Fig. \ref{fig:alpha1} (compared to Fig. \ref{fig:alpha2}) shows the partial and ordinary derivatives on the same graph by the thin and thick lines respectively. 
Since $\frac{\partial}{\partial \omega} F(\phi)$ remains positive, the even periodic wave 
is stable by Theorem \ref{theorem-stability-even}. Since 
$F(\psi)$ for the even periodic wave is decreasing in $c$ 
towards the horizontal asymptote as $c \to \infty$, it is clear that the stability conclusion does not follow from the dependence of the momentum $F(\psi)$ versus the wave speed $c$ (see Remark \ref{remark-slope}).

\begin{figure}[htpb]
	\centering
	\includegraphics[width=0.45\linewidth]{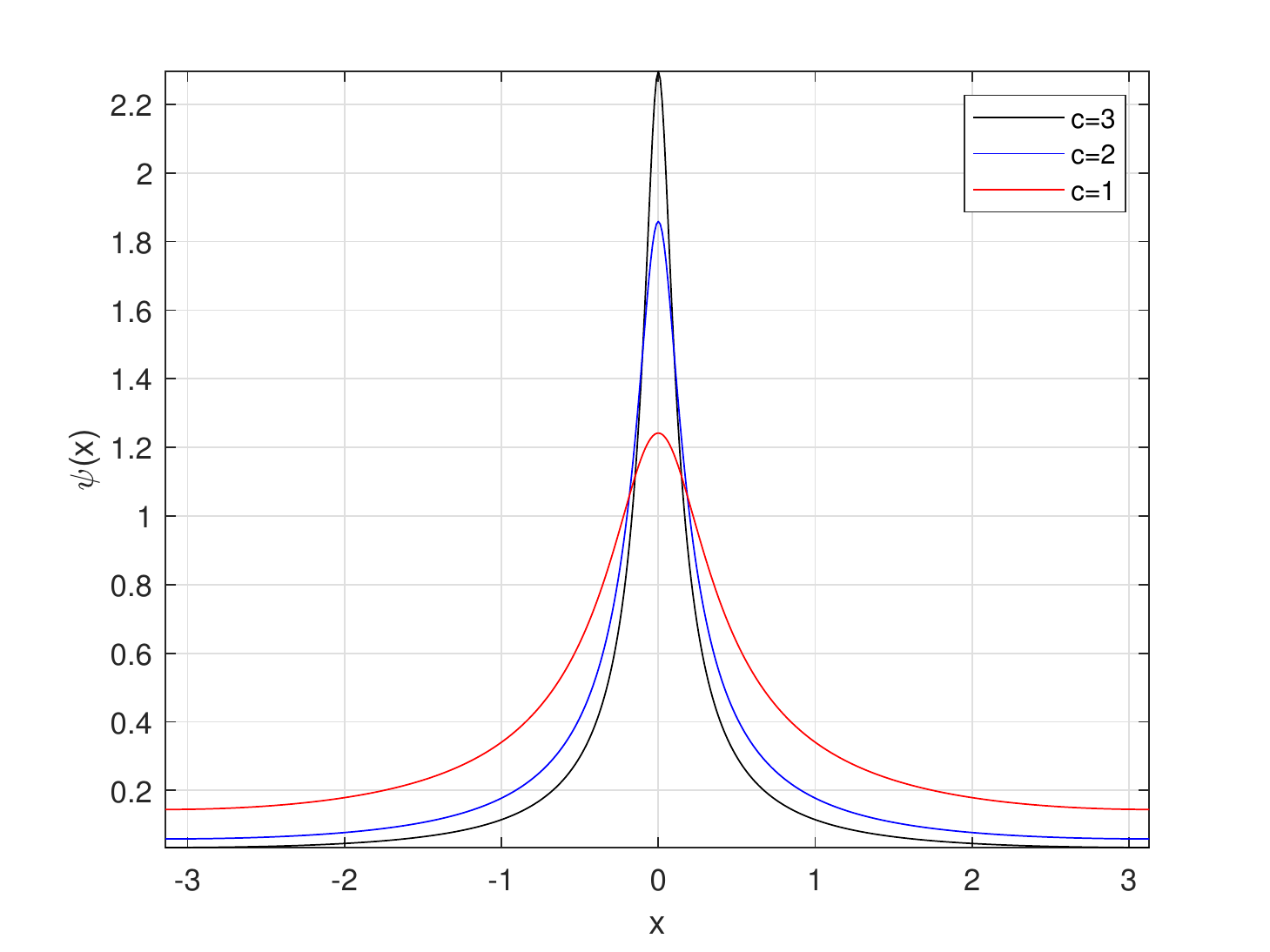}
	\includegraphics[width=0.45\linewidth]{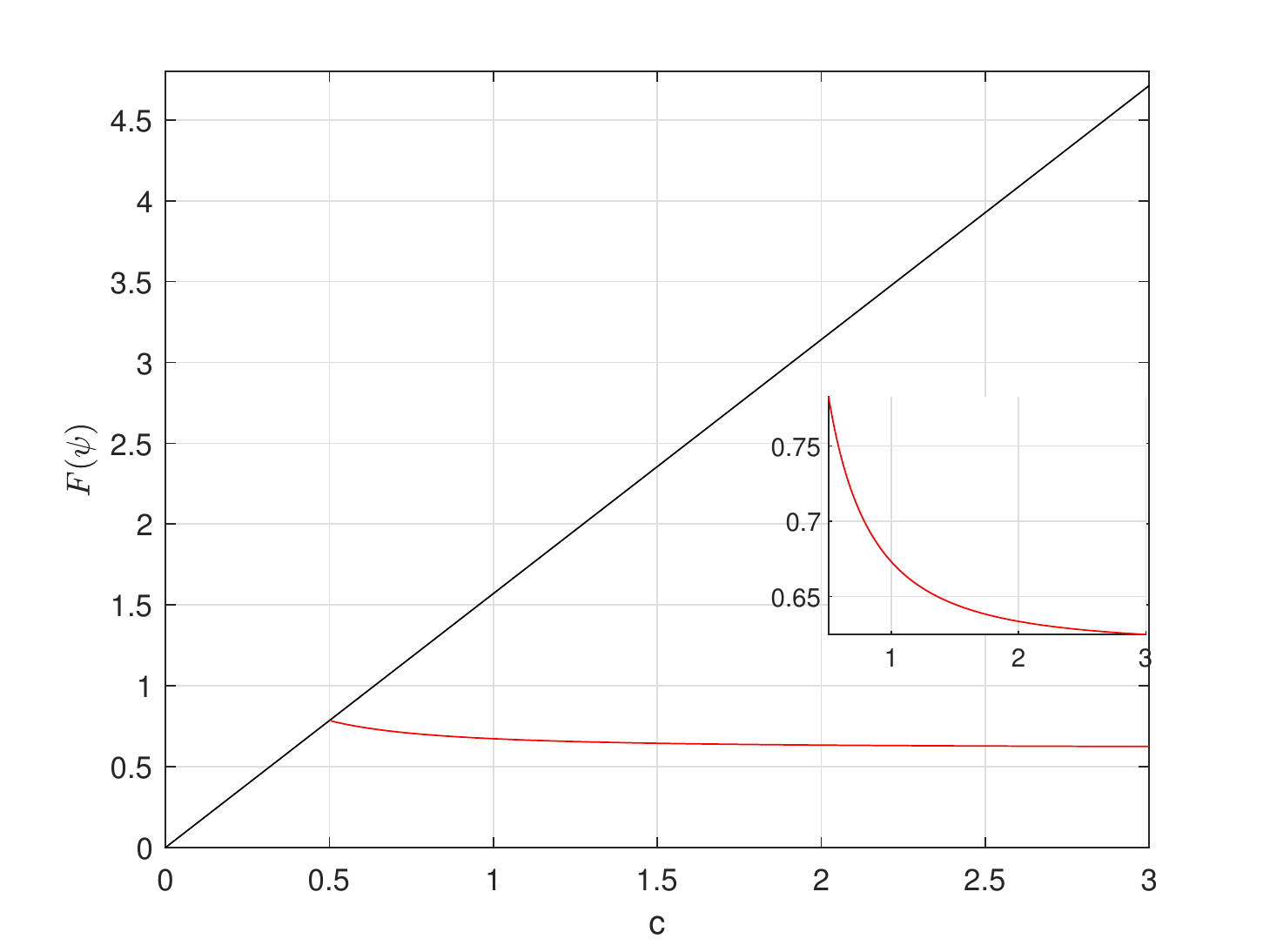} \\
	\includegraphics[width=0.45\linewidth]{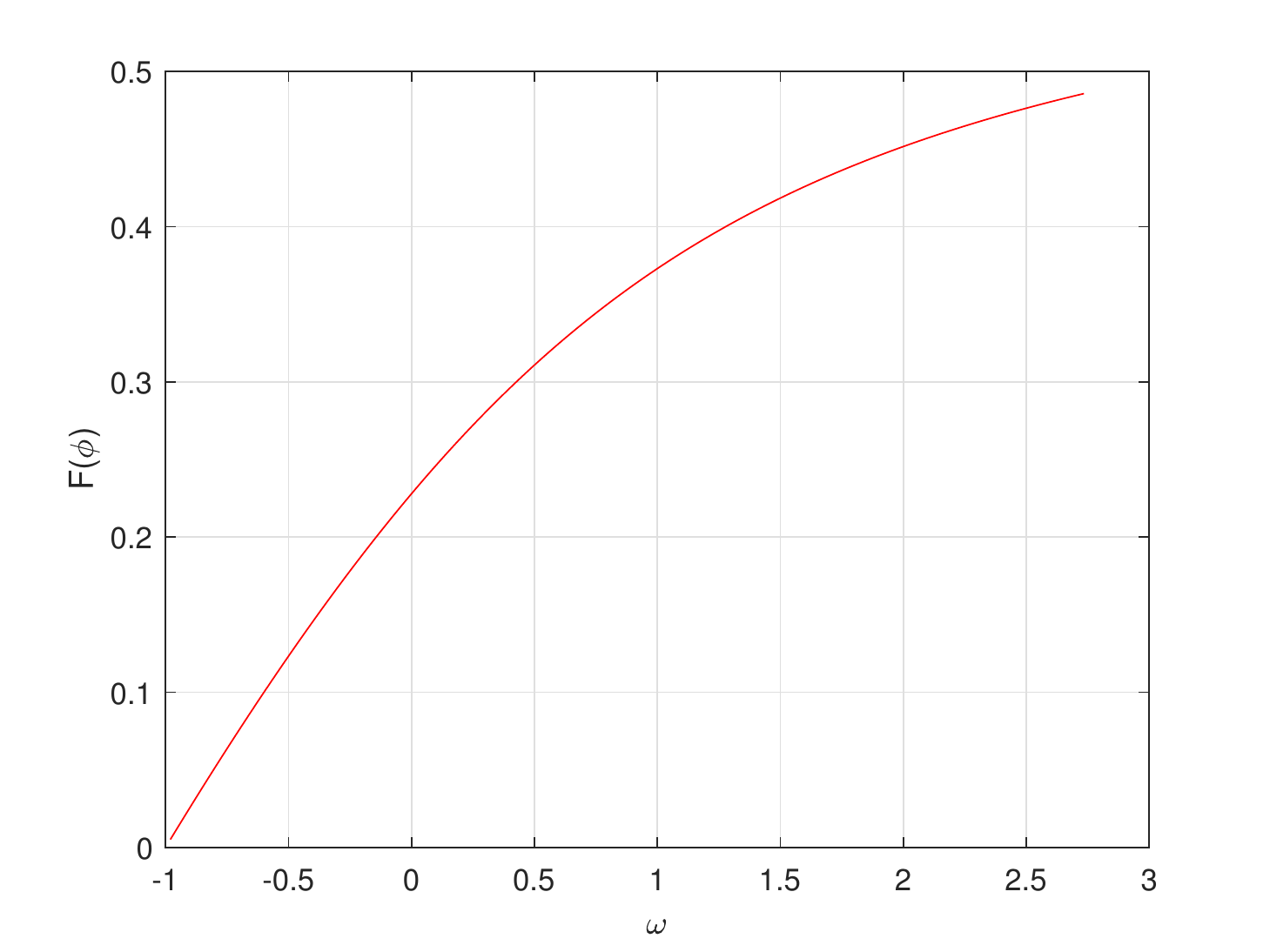}
	\includegraphics[width=0.45\linewidth]{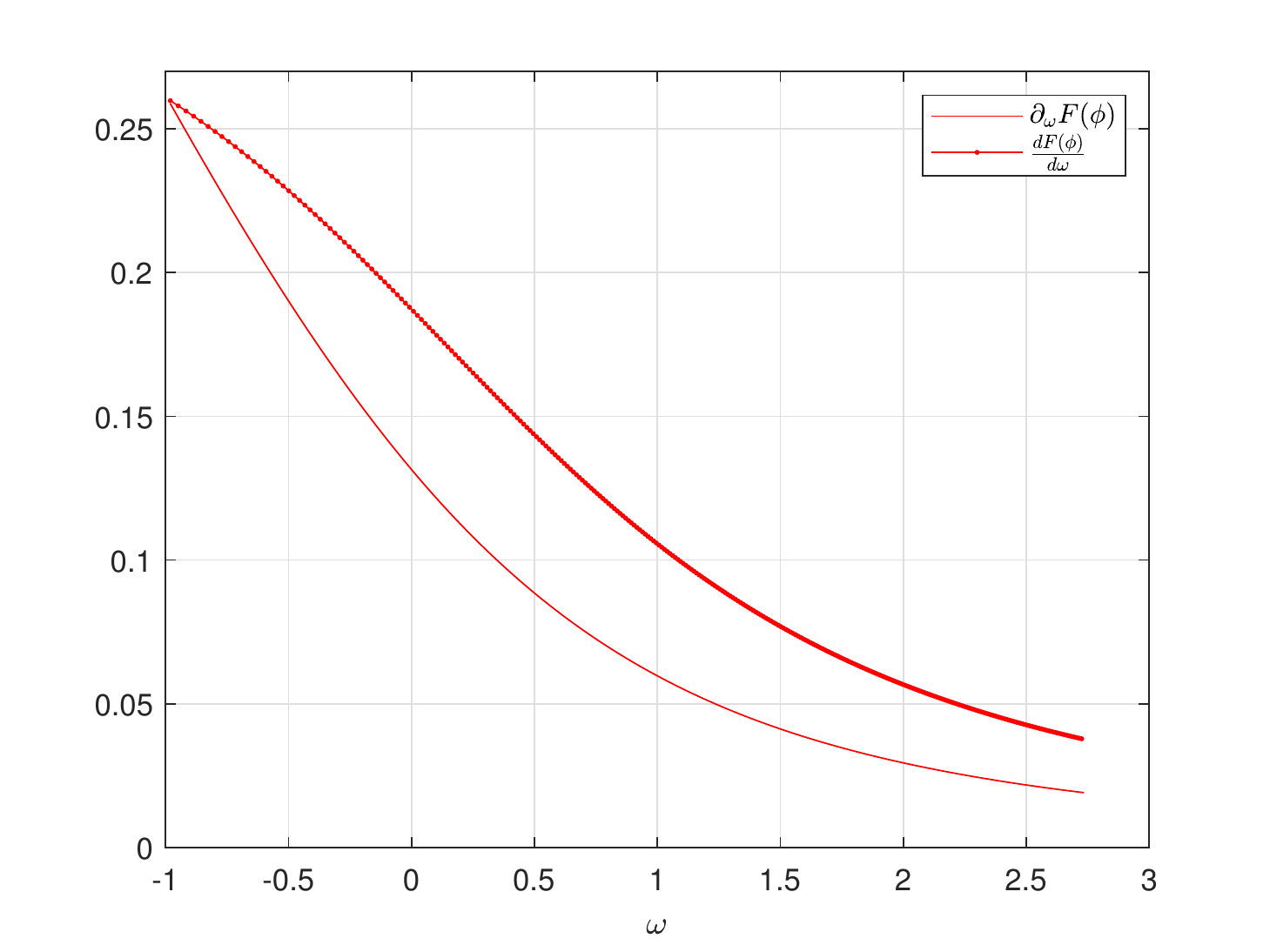}
	\caption{Periodic waves for $\alpha = 1$. Top: Profiles of $\psi$ for three different values of $c$ (left). Dependence of the momentum $F(\psi)$ versus $c$ (right). Bottom: Dependence of $F(\phi)$ versus $\omega$ (left).
		and derivatives of $F(\phi)$ in $\omega$ (right). The thin (thick) line shows the partial (ordinary) derivative in $\omega$.}
	\label{fig:alpha1}
\end{figure}

Figure \ref{fig:alpha06} presents similar results but for $\alpha = 0.6 < \alpha_0$. The periodic wave with the single-lobe profile $\psi$ bifurcates to the left 
of the bifurcation point at $c  = \frac{1}{2}$. There exists a fold point 
$c = c_0 \approx 0.4722$, where the branch turns and extends to all values 
of $c > c_0$. The upper branch (shown in red line on the top right panel) in $c \in (c_0,\frac{1}{2})$ has $n(\mathcal{L}) = 2$, whereas the lower branch (shown in blue line on the top right panel) has $n(\mathcal{L}) = 1$. The two branches were 
found iteratively from different initial approximations: the Stokes expansion 
was used for the upper branch and the periodic wave with larger $c > \frac{1}{2}$ was used for the lower branch, then the two branches were continued in either direction. The grey line on the top right panel shows the momentum $F(\psi)$ of the constant solution. 

It follows from the graph of $F(\phi)$ versus $\omega$ and its derivatives (bottom panels) that 
the periodic wave is stable near the bifurcation point before and after the fold point but there exists $c_* \approx 0.4774$ such that the 
even periodic wave is stable for $c < c_*$ and unstable for $c > c_*$. 
By comparing the partial and ordinary derivatives of $F(\phi)$ with respect to $\omega$, we can see that the partial derivative becomes zero for a smaller value of $\omega$, which gives the correct transttion from stability to instability at $c = c_*$ by Theorem \ref{theorem-stability-even}.

\begin{figure}[htpb]
	\centering
	\includegraphics[width=0.45\linewidth]{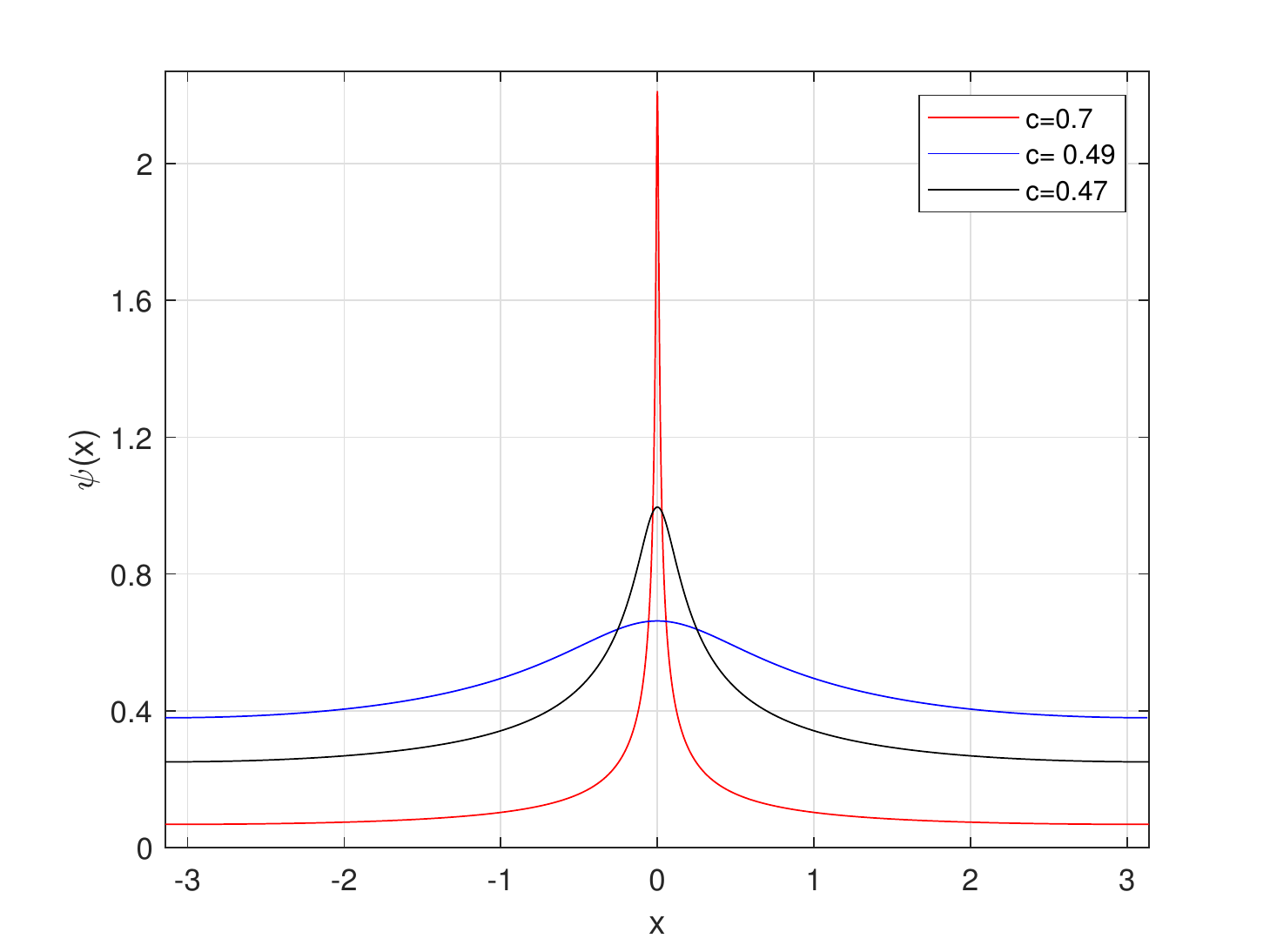} 
	\includegraphics[width=0.45\linewidth]{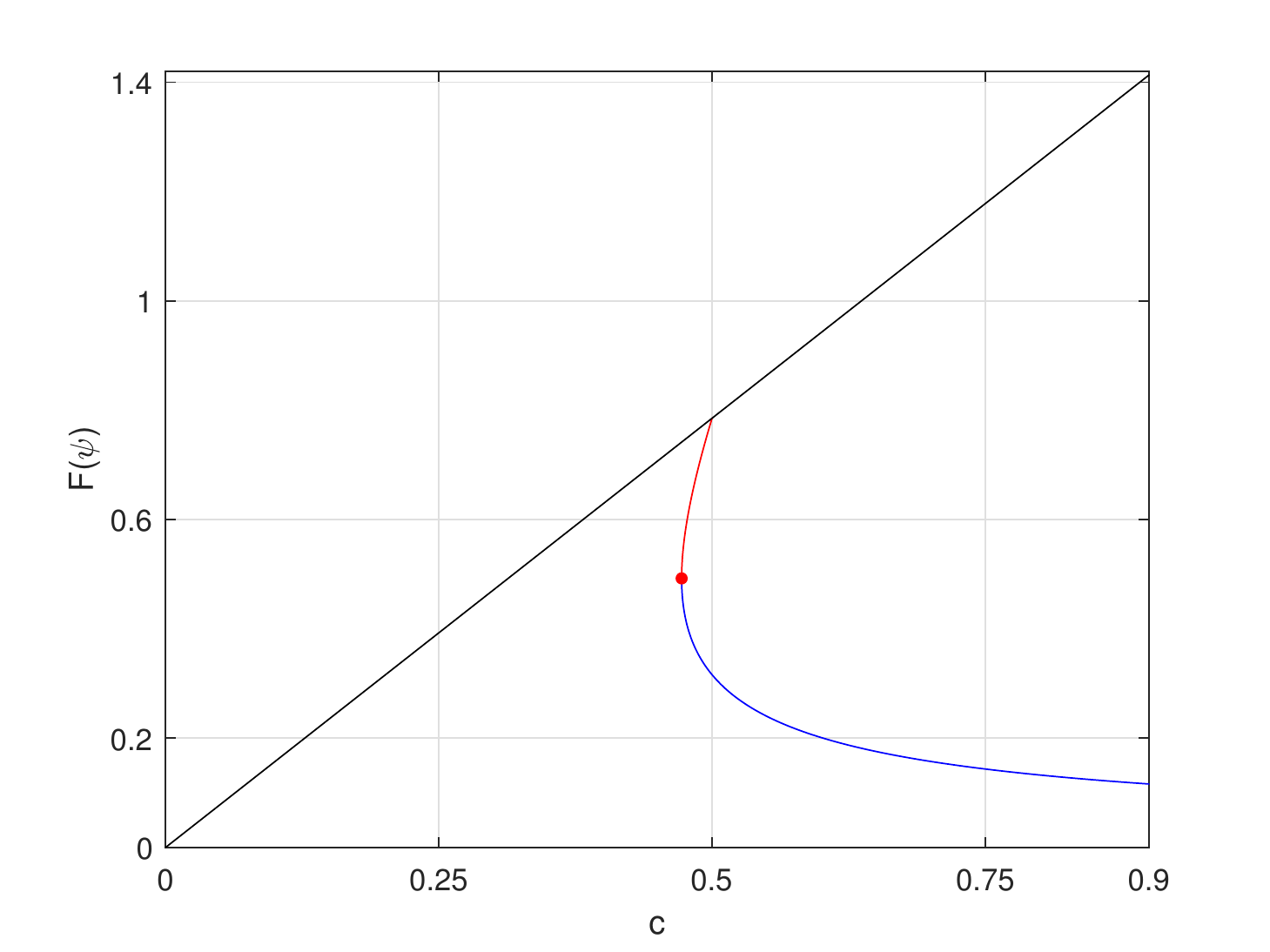} \\
	\includegraphics[width=0.45\linewidth]{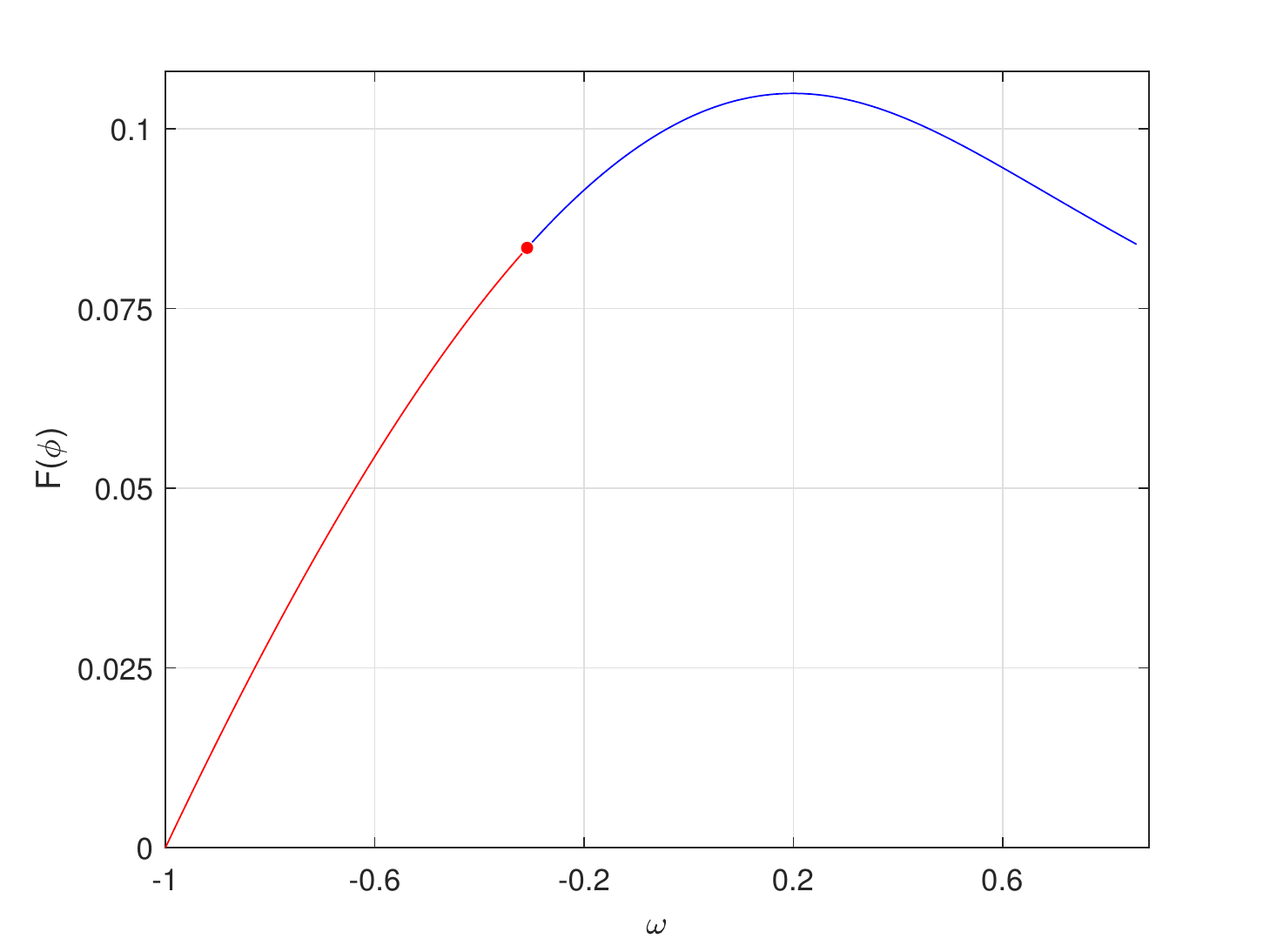}
	\includegraphics[width=0.45\linewidth]{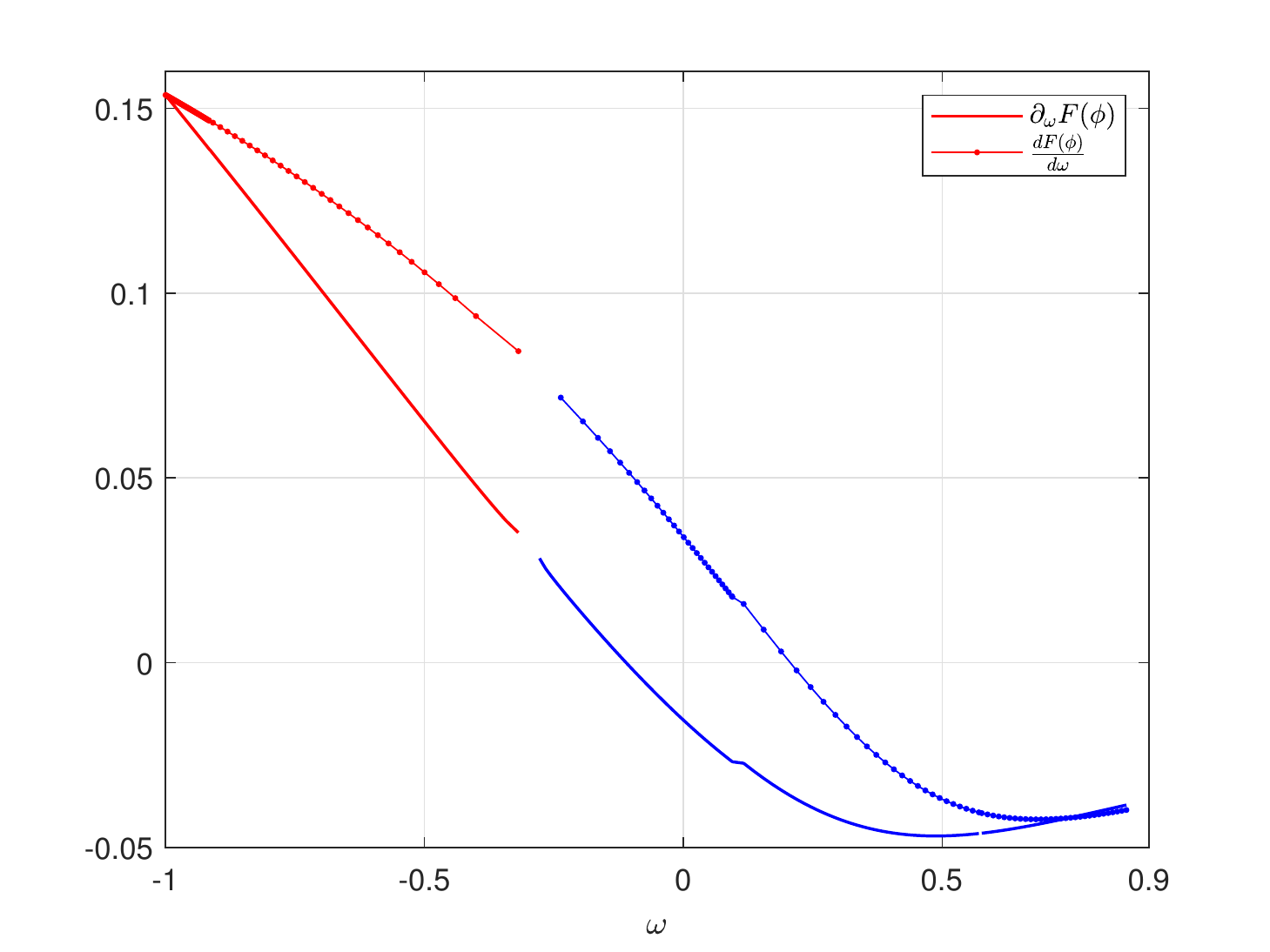}
	\caption{The same as Figure \ref{fig:alpha1} but for $\alpha = 0.6$.}
	\label{fig:alpha06}
\end{figure}

\section{Discussion}

It remains an open problem to characterize the most general solution of the stationary equation (\ref{ode-stat}) with arbitrary $b$. Here we generalize the two alternative parametrizations used in this work for the odd and even periodic waves.

Let $\psi \in H^{\alpha}_{\rm per, even}$ be a periodic wave solution to the stationary equation (\ref{ode-stat}) with
parameters $(c,b)$ defined in an open region $\mathcal{I} \subset \mathbb{R}^2$.
As in Section 5, we can define $\psi(x) = a + \phi(x)$, where $a := \frac{1}{2\pi} \int_{-\pi}^{\pi} \psi(x) dx$
so that $\langle 1, \phi \rangle = 0$. Then $\phi \in H^{\alpha}_{\rm per, even} \cap X_0$ is a solution of the stationary equation
(\ref{ode-stat-phi}), where $\omega := c - 6a^2$ and $\beta := b + c a - 2 a^3$. Parameters $(\omega,a)$
are defined in an open region $\mathcal{O} \subset \mathbb{R}^2$, which is
the image of the transformation $\omega = \omega(c,b)$ and $a = a(c,b)$, whereas
$\beta = \beta(c,b)$ is uniquely determined by (\ref{beta-expression}). Smoothness and
invertibility of this transformation are described as follows.

\begin{proposition}
\label{prop-conc-1}
Assume that $z(\mathcal{L}) = 1$ for a periodic wave with the single-lobe profile
$\psi \in H^{\alpha}_{\rm per, even}$. Then, the mapping $\mathcal{I} \ni (c,b) \mapsto (\omega,a) \in \mathcal{O}$ is $C^1$.
The transformation is invertible if and only if
\begin{equation}
\label{invertibility}
\frac{\partial a}{\partial b} \neq 0.
\end{equation}
\end{proposition}

\begin{proof}
In the case $z(\mathcal{L}) = 1$, the mapping $\mathcal{I} \ni (c,b) \mapsto \psi \in H^{\alpha}_{\rm per, even}$
is $C^1$ by the implicit function theorem (applied similarly to the proof of Lemma \ref{prop-continuation}). Let $\omega := c - 6a^2$
and $a := \frac{1}{2\pi} \int_{\mathbb{T}} \psi(x) dx$.
The mapping $\mathcal{I} \ni (c,b) \mapsto (\omega,a) \in \mathcal{O}$ is $C^1$ and the Jacobian of the transformation is given by
$$
\left| \begin{matrix} 1 - 12 a \frac{\partial a}{\partial c} & - 12 a \frac{\partial a}{\partial b} \\
\frac{\partial a}{\partial c} & \frac{\partial a}{\partial b} \end{matrix} \right| = \frac{\partial a}{\partial b},
$$
so that the transformation is invertible if and only if the condition (\ref{invertibility}) is satisfied.
\end{proof}

\begin{remark}
Since the mapping $\mathcal{I} \ni (c,b) \mapsto \psi \in H^{\alpha}_{\rm per, even}$
is $C^1$ in Proposition \ref{prop-conc-1}, we have 
$$
\mathcal{L} \frac{\partial \psi}{\partial b} = -1 \quad \Rightarrow \quad 
\sigma_0 = \langle \mathcal{L}^{-1} 1,1 \rangle = -\langle \frac{\partial \psi}{\partial b},1 \rangle 
= - 2 \pi \frac{\partial a}{\partial b}.
$$
It follows that the criterion (\ref{invertibility}) is equivalent to the criterion for $z(\mathcal{L}|_{X_0}) = 1$. If $\frac{\partial a}{\partial b} = 0$, 
then $z(\mathcal{L}|_{X_0}) = 2$ and $\frac{\partial \psi}{\partial b} \in {\rm Ker}(\mathcal{L} |_{X_0})$. In the latter case, we have bifurcation studied in the context of the odd periodic wave in Lemma \ref{lemma-characterization}.
\end{remark}

Similarly, but in the opposite direction, let $\phi \in H^{\alpha}_{\rm per, even} \cap X_0$ be
a periodic wave solution to the stationary equation (\ref{ode-stat-phi}) with
parameters $(\omega,a)$ defined in an open region $\mathcal{O} \subset \mathbb{R}^2$
and parameter $\beta = \beta(\omega,a)$ being uniquely defined by (\ref{beta-expression}).
Then, $\psi(x) = a + \phi(x) \in H^{\alpha}_{\rm per, even}$ is a solution of the stationary equation
(\ref{ode-stat}), where $c := \omega + 6a^2$ and $b := \beta(\omega,a) - \omega a - 4 a^3$. Parameters $(c,b)$
are defined in an open region $\mathcal{I} \subset \mathbb{R}^2$, which is
the image of the transformation $c = c(\omega,a)$ and $b = b(\omega,a)$. Smoothness and
invertibility of this transformation are described as follows.
\begin{proposition}
\label{prop-conc-2}
Assume that $z(\mathcal{L}|_{X_0}) = 1$ for a zero-mean periodic wave with the single-lobe profile
$\phi \in H^{\alpha}_{\rm per, even} \cap X_0$. Then, the mapping $\mathcal{O} \ni (\omega,a) \mapsto (c,b) \in \mathcal{I}$ is $C^1$.
The transformation is invertible if and only if
\begin{equation}
\label{invertibility-2}
\omega - \frac{\partial \beta}{\partial a} + 12 a \frac{\partial \beta}{\partial \omega} \neq 0.
\end{equation}
\end{proposition}

\begin{proof}
In the case $z(\mathcal{L}|_{X_0}) = 1$, the mapping $\mathcal{O} \ni (\omega,a) \mapsto \phi \in H^{\alpha}_{\rm per, even} \cap X_0$
is $C^1$ by the implicit function theorem (applied similarly to the proof of Lemma \ref{lem-continuation-even}). Let $c := \omega + 6 a^2$
and $b := \beta(\omega,a) - \omega a - 4 a^3$. 
The mapping $\mathcal{O} \ni (\omega,a) \mapsto (c,b) \in \mathcal{I}$ is $C^1$ and  the Jacobian of the transformation is given by
$$
\left| \begin{matrix} 1 & 12 a  \\
\frac{\partial \beta}{\partial \omega} - a & \frac{\partial \beta}{\partial a} - \omega - 12 a^2 \end{matrix} \right|
= \frac{\partial \beta}{\partial a} + 12 a \frac{\partial \beta}{\partial \omega} - \omega,
$$
so that the transformation is invertible if and only if the condition (\ref{invertibility-2}) is satisfied.
\end{proof}

\begin{remark}
The criterion (\ref{invertibility-2}) is equivalent to the criterion 
for $z(\mathcal{L}) = 1$ as in Lemma \ref{lemma-kernel}. If 
$$
\omega - \frac{\partial \beta}{\partial a} + 12 a \frac{\partial \beta}{\partial \omega} = 0, 
$$
then $z(\mathcal{L}) = 2$ as follows 
from equality (\ref{range-kernel}).
\end{remark}

Both parameters $c$ and $b$ arise as Lagrange multipliers of the following variational problem with two constraints
\begin{equation}
\label{infB-gen}
r_{c,m} := \inf_{u\in H^{\frac{\alpha}{2}}_{\rm per}} \left\{ \mathcal{B}_c(u) : \quad
 \int_{-\pi}^{\pi} u^4 dx = 1, \quad \frac{1}{2\pi} \int_{-\pi}^{\pi} u dx = m \right\}.
\end{equation}
When $m = 0$, the variational problem (\ref{infB-gen}) reduces to the form (\ref{infB}) which was used in the context of the odd periodic waves.
Without loss of generality, it suffices to consider (\ref{infB-gen}) for $m \geq 0$ as solutions for $m \leq 0$ are mapped
to solutions for $m \geq 0$ by the transformation $u \mapsto -u$. 
As is shown in Appendix B, the variational problem 
(\ref{infB-gen}) defines constant solutions if $m = m_0 := (2 \pi)^{-\frac{1}{4}}$ 
and periodic waves with the single-lobe profile if $m \in (-m_0,m_0)$.

Further studies are needed to investigate how the variational problem (\ref{infB-gen}) with two constraints recovers the most general periodic solution to the stationary equation (\ref{ode-stat}) with two parameters $(c,b)$.

\vspace{0.25cm}

{\bf Acknowledgements:} F. Natali is supported by Funda\c{c}\~ao Arauc\'aria (grant 002/2017) and CNPq (grant 304240/2018-4). 
He would like to express his gratitude to members of the Department of Mathematics at McMaster University for their hospitality during his stay. U.Le is supported by the graduate scholarship from McMaster University. D.Pelinovsky is supported by the NSERC Discovery grant.

\appendix

\section{Stokes expansion of general small-amplitude waves}

Here we generalize the Stokes expansion of Section 5.1 in order to prove that
$\sigma_0 > 0$ for $\alpha > \alpha_0$ and $\sigma_0 < 0$ for $\alpha < \alpha_0$,
where $\sigma_0 = \langle \mathcal{L}^{-1} 1, 1 \rangle$ is computed on
the small-amplitude wave of Proposition \ref{prop-Stokes-even} for small amplitude $A$.

Let $\psi$ satisfy the stationary equation (\ref{ode-stat}) for $(c,b)$ defined
in an open neighborhood $\mathcal{I} \subset \mathbb{R}^2$ of the point
$\left(\frac{1}{2},0\right)$. We generalize the decomposition (\ref{Stokes-expansion}) by setting
\begin{equation}
\label{Stokes-expansion-app}
\psi(x) = \psi_0 + \varphi(x),
\end{equation}
where $\psi_0 = \psi(c,b)$ is a root of the cubic equation $b + c \psi_0 = 2 \psi_0^3$ and
$\varphi$ is not required to satisfy the zero-mean property. Since three roots exist for $\psi_0$
at $b = 0$, we are picking uniquely the positive root by using the expansion
\begin{equation}
\label{psi-0-expansion}
\psi_0(c,b) = \frac{1}{2} \sqrt{2c} + \frac{b}{2c} + \mathcal{O}(b^2).
\end{equation}
The stationary equation (\ref{ode-stat}) is written in the equivalent form:
\begin{equation}
\label{ode-stat-stokes-app}
D^{\alpha} \varphi + (c - 6 \psi_0^2) \varphi = 2 \varphi^3 + 6 \psi_0 \varphi^2,
\end{equation}
which generalizes (\ref{ode-stat-stokes}).
By using the Stokes expansion in terms of small amplitude $A$:
\begin{equation}
\label{Stokes-even-app}
\varphi(x) = A \varphi_1(x) + A^2 \varphi_2(x) + A^3 \varphi_3(x) + \mathcal{O}(A^4), \quad
c - 6 \psi_0^2 = -1 + A^2 \omega_2 + \mathcal{O}(A^4),
\end{equation}
we obtain recursively: $\varphi_1(x) = \cos(x)$,
$$
\varphi_2(x) = -3 \psi_0 + \frac{3 \psi_0}{2^{\alpha}-1} \cos(2x),
$$
$$
\varphi_3(x) = \frac{1}{3^{\alpha} - 1} \left[ \frac{1}{2} + \frac{18 \psi_0^2}{2^{\alpha}-1} \right] \cos(3x),
$$
and
$$
\omega_2 = \frac{3}{2} - 36 \psi_0^2 + \frac{18 \psi_0^2}{2^{\alpha}-1}.
$$
By substituting $\omega_2$ to the expansion for $c$ in (\ref{Stokes-even-app}) and using
expansion (\ref{psi-0-expansion}), we obtain
\begin{equation}
\label{A-squared}
\gamma_2 A^2 = 2c - 1 + 6 b + \mathcal{O}((2c-1)^2 + b^2),
\end{equation}
where $\gamma_2 = -\omega_2 |_{\psi_0 = \frac{1}{2}}$ is the same as in 
(\ref{Stokes-even}). By using (\ref{A-squared}), we obtain perturbatively:
\begin{eqnarray*}
a & = & \frac{1}{2\pi} \int_{-\pi}^{\pi} \psi(x) dx = \psi_0 \left[ 1 - 3 A^2 + \mathcal{O}(A^4) \right] \\
& = & \frac{1}{2} - a_1(2c-1) - a_2 b + \mathcal{O}((2c-1)^2 + b^2),\\
\omega & = & c - 6 a^2 \\
& = & -1 + \frac{1}{2} (1 + 12 a_1) (2c-1) + 6 a_2 b + \mathcal{O}((2c-1)^2 + b^2),
\end{eqnarray*}
and
\begin{eqnarray*}
\beta & = & b + c a - 2 a^3 \\
& = & \frac{1}{4} (1 + 4 a_1)(2c-1) + (1+ a_2) b + \mathcal{O}((2c-1)^2 + b^2),
\end{eqnarray*}
where
\begin{eqnarray*}
a_1 := \frac{3}{8 \gamma_2} \frac{4 - 2^{\alpha}}{2^{\alpha} - 1}, \quad
a_2 := \frac{3}{2 \gamma_2} \frac{2 + 2^{\alpha}}{2^{\alpha} - 1}.
\end{eqnarray*}
If $\gamma_2 \neq 0$ for $\alpha \neq \alpha_0$ given by (\ref{alpha-0}), the transformation $\mathcal{I} \ni (c,b) \mapsto (\omega,a) \in \mathcal{O}$
is $C^1$ and invertible with the inverse transformation
\begin{eqnarray*}
c & = & \frac{1}{2} + (\omega + 1) + 6 (a -\frac{1}{2}) + \mathcal{O}((\omega + 1)^2 + (2a-1)^2),\\
b & = & -\frac{1}{a_2} \left[ 2 a_1 (\omega + 1) + \frac{1}{2} (1 + 12 a_1) (2a - 1) + \mathcal{O}((\omega + 1)^2 + (2a-1)^2) \right],
\end{eqnarray*}
from which we obtain $\beta = \beta(\omega,a)$:
$$
\beta = \frac{a_2 - 4 a_1}{2a_2} (\omega + 1) + \frac{2 a_2 - 1 - 12 a_1}{2a_2}  (2a - 1) + \mathcal{O}((\omega + 1)^2 + (2a-1)^2)
$$
and
$$
s_0 = \omega - \partial_a \beta + 12 a \partial_{\omega} \beta  = \frac{1}{a_2} + \mathcal{O}((\omega + 1)^2 + (2a-1)^2).
$$
Since $\sigma_0 = \frac{2\pi}{s_0}$, we have ${\rm sign}(\sigma_0) = {\rm sign}(a_2) = {\rm sign}(\gamma_2)$,
from which it follows that $\sigma_0 > 0$ for $\alpha > \alpha_0$ and $\sigma_0 < 0$ for $\alpha < \alpha_0$.

Furthermore, it follows from (\ref{A-squared}) that 
$$
\gamma_2 A^2 = \frac{2 (a_2 - 6 a_1)}{a_2} (\omega + 1) + \frac{3(2a_2 - 1 - 12 a_1)}{a_2} (2 a - 1) + \mathcal{O}((\omega+1)^2 + (2a-1)^2).
$$
Explicit computation shows that $2a_2 - 1 - 12 a_1 = 0$, hence 
$\| \phi \|^2_{L^2} = \pi A^2 + \mathcal{O}(A^4)$ 
as a function of $(\omega,a)$ satisfies 
$$
\frac{\partial}{\partial \omega} \| \phi \|_{L^2}^2 = \frac{2\pi (a_2 - 6 a_1)}{\gamma_2 a_2} + \mathcal{O}(A^2) = \frac{2\pi}{3} 
\frac{2^{\alpha} - 1}{2^{\alpha} + 2}  + \mathcal{O}(A^2) 
= \frac{d}{d \omega} \| \phi \|_{L^2}^2 + \mathcal{O}(A^2),
$$
in agreement with (\ref{der-small-amplitude}). In other words, 
although $a$ is defined by $\omega$ at the periodic waves satisfying $b = 0$ by 
$$
2a - 1 = -\frac{4 a_1}{1 + 12 a_1} (\omega + 1) + \mathcal{O}((\omega + 1)^2),
$$
this dependence does not result in the discrepancy between partial and ordinary 
derivatives of $\| \phi \|_{L^2}^2$ in $\omega$ along the family 
of even periodic waves in the limit $A \to 0$.

\section{On the variational problem (\ref{infB-gen}) with two constraints}

We show that the periodic solutions to the stationary equations (\ref{ode-stat}) with two parameters $(c,b)$ can be recovered from
the ground state of the variational problem (\ref{infB-gen}).

\begin{proposition}
	\label{prop-conc-3}
	Fix $\alpha > \frac{1}{2}$ and $m_0 := (2\pi)^{-\frac{1}{4}}$. For every $m \in [-m_0,m_0]$ and every $c \in (-1,\infty)$,
	there exists the ground state (minimizer)  $\chi \in H^{\frac{\alpha}{2}}_{\rm per}$ of the variational problem (\ref{infB-gen}).
	If $m \in (-m_0,m_0)$, the ground state has the single-lobe profile and
	there exists $C > 0$ such that $\psi(x) = C \chi(x)$ satisfies the stationary equation (\ref{ode-stat}) with some $b$.
\end{proposition}

\begin{proof}
	The bound $|m| \leq m_0$ follows by the H\"{o}lder's inequality
	$$
	\left| \int_{\mathbb{T}} u dx \right| \leq \left( \int_{\mathbb{T}} 1^{\frac{4}{3}} dx \right)^{\frac{3}{4}}
	\left( \int_{\mathbb{T}} u^4 dx \right)^{\frac{1}{4}} = (2\pi)^{\frac{3}{4}}.
	$$
	The quadratic functional $B_c(u)$ in (\ref{def-B}) is bounded from below by the Poincar\'{e} inequality:
	$$
	B_c(u) \geq \frac{1}{2} \| u - m \|^2_{L^2} + \frac{1}{2} c \| u \|^2_{L^2} =
	\frac{1}{2} (1 + c) \| u \|^2_{L^2} - \pi m^2.
	$$
	By the same analysis as in the proof of Theorem \ref{theorem-existence},
	for every $m \in [-m_0,m_0]$ and every $c \in (-1,\infty)$, there exists the ground state
	$\chi \in H^{\frac{\alpha}{2}}_{\rm per}$ of the variational problem (\ref{infB-gen}). Moreover, by the symmetric rearrangements, 
	the ground state is either constant or has the single-lobe profile. The constant solution
	corresponds to $|m| = m_0$, hence the ground state has the single-lobe profile if $m \in (-m_0,m_0)$.
	
	With two Lagrange multipliers $\mu$ and $\nu$, the ground state
	$\chi \in H^{\frac{\alpha}{2}}_{\rm per}$ satisfies the stationary equation
	\begin{equation}
	\label{EL-last}
	D^{\alpha} \chi + c \chi + \nu = \mu \chi^3.
	\end{equation}
	Lagrange multipliers satisfy two relations due to the constraints in (\ref{infB-gen}):
	\begin{equation}
	\label{LM-last}
	\mu = 2 B_c(\chi) + 2 \pi m \nu, \quad \mu \int_{\mathbb{T}} \chi^3 dx = 2 \pi (c m + \nu).
	\end{equation}
	Eliminating $\nu$ yields
	\begin{equation}
	\label{LM-last-relation}
	\left[ 1 - m \int_{\mathbb{T}} \chi^3 dx \right] \mu = 2 \left[ B_c(\chi) - \pi c m^2 \right].
	\end{equation}
	The left-hand side of (\ref{LM-last-relation}) can be rewritten in the equivalent symmetrized form:
	\begin{eqnarray*}
		1 - m \int_{\mathbb{T}} \chi^3 dx & = & \frac{1}{2\pi} \left[ \left( \int_{\mathbb{T}} dx \right)
		\left( \int_{\mathbb{T}} \chi^4 dx \right) - \left( \int_{\mathbb{T}} \chi dx \right) \left( \int_{\mathbb{T}} \chi^3 dx \right) \right] \\
		& = & \frac{1}{16 \pi} \int_{\mathbb{T}} \int_{\mathbb{T}}
		\left( \left[ \chi(x) - \chi(y) \right]^4 + 3 \left[ \chi^2(x) - \chi^2(y) \right]^2 \right) dx dy,
	\end{eqnarray*}
	from which it follows that it is strictly positive if $\chi(x)$ is not constant.
	Similarly, the right-hand side of (\ref{LM-last-relation}) is strictly positive if $\chi(x)$ is not constant due to the following inequality
	$$
	2 B_c(\chi) - 2 \pi c m^2 = \| D^{\frac{\alpha}{2}} \chi \|^2_{L^2} + c \| u - m \|^2_{L^2}
	\geq (1 + c) \| u - m \|^2_{L^2} > 0.
	$$
	Since the ground state is non-constant if $m \in (-m_0,m_0)$,
	we obtain the unique $\mu > 0$ from (\ref{LM-last}) such that the transformation $\psi = C \chi$ with $C := \sqrt{\mu}/\sqrt{2}$
	reduces (\ref{EL-last}) to the stationary equation (\ref{ode-stat})
	with $b = \nu \sqrt{\mu}/\sqrt{2}$, where $\nu$ is uniquely found from (\ref{LM-last}).
\end{proof}

\end{document}